\documentclass[aap]{imsart}

\RequirePackage{amsthm,amsmath,amsfonts,amssymb}
\RequirePackage[numbers,sort&compress]{natbib}
\RequirePackage[colorlinks,citecolor=blue,urlcolor=blue]{hyperref}
\RequirePackage{color}
\RequirePackage{enumitem}
\usepackage[normalem]{ulem}
\usepackage[]{cancel}

\startlocaldefs
\theoremstyle{plain}
\newtheorem{theorem}{Theorem}
\newtheorem{proposition}{Proposition}
\newtheorem{lemma}{Lemma}

\theoremstyle{definition}
\newtheorem{assumption}{Assumption}
\newtheorem{definition}{Definition}
\newtheorem{remark}{Remark}
\newtheorem{example}{Example}
\renewcommand{\P}{\mathbb{P}}
\renewcommand{\d}{\mathrm{d}}
\newcommand{\wto}{\stackrel{\mathrm{w}}{\to}}

\DeclareMathOperator{\eqd}{\stackrel{\mathrm{d}}{=}}

\DeclareMathOperator{\supp}{\mathrm{supp}}
\newcommand{\Nstat}{N^{(\mathrm{stat})}}
\newcommand{\Nstathat}{\hat{N}^{(\mathrm{stat})}}
\newcommand{\NPalm}{N^{(\mathrm{Palm})}}
\newcommand{\lambdastat}{\lambda^{(\mathrm{stat})}}
\newcommand{\Fstat}{\mathcal{F}^{(\mathrm{stat})}}
\newcommand{\taustat}{\tau^{(\mathrm{stat})}}
\newcommand{\gammastat}{\gamma^{(\mathrm{stat})}}
\newcommand{\lstat}{\lambda^{(\mathrm{stat})}}
\newcommand{\cN}{\mathcal{N}}

\newcommand{\E}{\mathbb{E}}
\newcommand{\Z}{\mathbb{Z}}
\newcommand{\R}{\mathbb{R}}
\newcommand{\F}{\mathcal{F}}
\newcommand{\N}{\mathbb{N}}
\renewcommand{\L}{\mathcal{L}}

\DeclareMathOperator{\Unif}{\mathrm{Unif}}

\newcommand{\oN}{\overline{N}}
\newcommand{\oS}{\overline{S}}

\newcommand{\iidsim}{\stackrel{\mathrm{i.i.d.}}{\sim}}

\makeatletter\def\journal@name{}\makeatother
\endlocaldefs

\begin{document}

\begin{frontmatter}
\title{Loglinear Hawkes processes: existence, nonexplosion, stability and moment bounds}
\runtitle{Loglinear Hawkes processes}

\begin{aug}
\author[A]{\fnms{Tomasz~R.}~\snm{Bielecki}\ead[label=e1]{tuslecki@gmail.com}}
\author[B]{\fnms{Jacek}~\snm{Jakubowski}\ead[label=e2]{j.jakubowski@mimuw.edu.pl}}
\author[C]{\fnms{Matthias}~\snm{Kirchner}\ead[label=e3]{matthias.kirchner@phbern.ch}}
\author[D]{\fnms{Mariusz}~\snm{Niew\k{e}g\l{}owski}\ead[label=e4]{mariusz.nieweglowski@pw.edu.pl}}
\address[A]{Hinsdale, Illinois, USA\printead[presep={,\ }]{e1}}

\address[B]{
University of Warsaw\printead[presep={,\ }]{e2}}

\address[C]{
University of Teacher Education Bern\printead[presep={,\ }]{e3}}

\address[D]{
Warsaw University of Technology\printead[presep={,\ }]{e4}}
\end{aug}

\begin{abstract}
This paper discusses a special class of nonlinear Hawkes processes, where the rate function is the exponential function. We call these processes loglinear Hawkes processes. Such processes are used with success in computational neuroscience and we provide a mathematical background for them. After obtaining existence from a standard Poisson embedding, we give sufficient conditions for explosion and nonexplosion that cover a large class of memory functions. In order to give sufficient stability conditions, we construct an autoregressive integer-valued time series that gives pathwise upper bounds for the bin counts of a loglinear Hawkes process. From here, a Markov chain argument shows that the regeneration time for a loglinear Hawkes process has finite expectation if the memory function is negative enough in a neighborhood of zero. Stability in total variation follows. Finally, we show that both the loglinear Hawkes process under a large class of initial conditions and its stationary version have finite exponential moments.
\end{abstract}

\begin{keyword}[class=MSC]
\kwdgroup[type=primary]{\kwd{60G55}
\kwd{60G10}}
\kwdgroup[type=secondary]{\kwd{60J05}\kwd{92C20}}
\end{keyword}

\begin{keyword}
\kwd{Hawkes process}
\kwd{loglinear Hawkes process}
\kwd{explosion}
\kwd{stability}
\kwd{inhibition}
\kwd{spike response model}
\end{keyword}

\end{frontmatter}


\section{Introduction}\label{sec:introduction}
A Hawkes process is a stochastic model for the distribution of points on the real line. Most of the time, the line is interpreted as time, and the points as events in time. The characteristic property of a Hawkes process is that the probability of a new event at some time $t\in\mathbb{R}$ can be represented as a function of the events before time $t$. For the case when this dependence is linear (in a certain sense), the process was introduced in \cite{hawkes71a} and \cite{hawkes71b}. The nonlinear case was introduced in \cite{bremaud96}. We follow the notation and concepts of the latter work.

Let $(\Omega,\mathcal{F},\mathbb{P})$ be a probability space rich enough to carry all random variables involved. All filtrations considered in the paper satisfy the usual conditions (if not, we complete them and make them right-continuous). On $(\Omega,\mathcal{F},\mathbb{P})$, define measurable mappings $
T_n:\Omega \to \mathbb{R}\cup\{\pm\infty\}, n\in\Z,$
 such that
\begin{equation}
\ldots \leq T_{-2}\leq T_{-1}\leq T_0 \leq 0 < T_1\leq T_2\leq \dots \label{eq:jump_times}
\end{equation}
with equality between $T_n$ and $T_{n+1}$ only if both equal $-\infty$ or $\infty$. Furthermore, we assume that
$
 \lim_{n\to -\infty} T_n=-\infty,
$
 whereas the \emph{explosion time}
 \begin{equation*}
T_\infty :=
\lim_{n\to\infty} T_n\quad \in(0,\infty] \quad \text{a.s.}
\end{equation*}
might be finite with positive probability. Note that
\begin{equation*}
N(\cdot) := \sum_{n\in\Z}\delta_{T_n}(\cdot)1\{|T_n| <\infty\}
\end{equation*}
defines a random counting measure on $(\mathbb{R},\mathcal{B}(\mathbb{R}))$. We call $T_n$ a jump time of $N$ if $|T_n| <\infty$. The point process $N$ is \emph{nonexplosive} if $T_\infty =\infty$ a.s. It is \emph{explosive} if $T_\infty <\infty$ with positive probability. For any point process $N$ on $\mathbb{R}$, define the $\sigma$-fields
$$
\mathcal{F}_t^N:=\sigma(N(A), A\in\mathcal{B}((-\infty,t]) ),\quad t\in\mathbb{R}.
$$
The filtration $\{\mathcal{F}_t^N\}_{t\in\mathbb{R}}$ is the \emph{internal history of $N$}. Any filtration $\{\mathcal{F}_t\}_{t\in\mathbb{R}}$ such that $\mathcal{F}_t^N\subseteq\mathcal{F}_t,t\in\mathbb{R},$ is called a \emph{history of $N$}. A nonnegative process  $\{\lambda(t)\}_{t\geq0}$ that is predictable with respect to $\{\mathcal{F}_t\}_{t\geq 0}$, a history of $N$, is called an \emph{$\{\mathcal{F}_t\}_{t\geq 0}$-intensity of $N$} if
\begin{equation}\label{eq:intensity}
\E [N((a,b])|\mathcal{F}_a] = \E \biggl[\int_a^b \lambda(t) \d t\bigg|\mathcal{F}_a\biggr]\quad \text{for all $0\leq  a < b$.}
\end{equation}

A nonlinear
Hawkes process $N$  is a simple point process on $(\mathbb{R},\mathcal{B}(\mathbb{R}))$ following some initial condition  on $( (-\infty,0] ,\mathcal{B}( (-\infty,0] ))$ and such that its {$\{\mathcal{F}_t^N\}_{t> 0}$}-intensity is of the form
\begin{equation}
\lambda(t) = \phi\biggl(\int_{(-\infty,t)}h(t - s)N(\d s)\biggr)1\{t<T_{\infty}\},\quad t> 0,
\label{eq:dynamics}
\end{equation}
with $h$ a \emph{memory function}  supported on $ [0,\infty)   $, such that
 the integral in \eqref{eq:dynamics} is well-defined, and a \emph{rate function} $\phi:\mathbb{R}\to [0,\infty)   $. The functions $h$ and $\phi$ are such that the process $\lambda$ is $\{\mathcal{F}^N_t\}_{t\geq 0}$-predictable. We say that $N$ \emph{follows the dynamics $\lambda$}.

For linear Hawkes processes (i.e., with $\phi(x) = \nu + x$), the memory function $h$ has to be essentially nonnegative because otherwise, for any $t$ with $0<t<T_\infty$,
$$\lambda(t)=\nu+\int_{(-\infty,t)}h(t - s)N(\d s)$$
 might be negative with positive probability. Thus, linear Hawkes processes exclude inhibition effects. However, in many data contexts, e.g., in neuroscience or in high frequency trading, we observe refractory periods: shortly after a neuron has fired, or shortly after the price has jumped, there will be no other event. This is one of the reasons why in the literature authors try to find rate functions $\phi$ mapping $\mathbb{R}$ to  $ [0,\infty)   $ yielding a tractable Hawkes model allowing for inhibition.

In \cite{bremaud96}, stability of univariate nonlinear Hawkes processes is proved under various  Lipschitz conditions on the rate function $\phi$ using contraction arguments. The results are generalized for marked processes in
\cite{massoulie98}. The thesis \cite{zhu13} studies asymptotics of nonlinear Hawkes processes under the additional assumption that the memory function $h$ is nonnegative. Furthermore, the rate function is Lipschitz or linearly  dominated. Also, \cite{ditlevsen17}, \cite{duarte19}, and \cite{gao23} work under a Lipschitz assumption on the rate function. One specific parametric nonlinear Hawkes process caught some attention, namely the truncated linear model with rate function $\phi(x) = (\nu +  x)^+$; see, e.g., \cite{costa20} and \cite{deutsch23}. The present paper examines nonlinear Hawkes processes with rate function $\phi(x) = e^{\nu + x}$, which makes the intensity of the process a loglinear function of the past of the process. As will be discussed below, this rate function is successfully applied in computational neuroscience, although it has not been covered by earlier theoretical results on nonlinear Hawkes processes because the exponential function is neither Lipschitz nor can it be dominated in a linear way.

In Section~\ref{sec:motivation} we motivate the choice of the exponential transfer function.
In Section~\ref{sec:existence} we formally construct a loglinear Hawkes process. In Section~\ref{sec:nonexplosion} we give criteria for explosion and nonexplosion of the process. In Section~\ref{sec:bin_count_domination} we bound the bin counts of a loglinear Hawkes process by an integer-valued autoregressive time series and derive a drift bound for the associated Markov chain. In Section~\ref{sec:stability} we use this drift bound to show that under finite memory support and strong enough inhibition after zero, the regeneration time has finite expectation.  From here, we construct a stationary version of the process and prove convergence to it in total variation, i.e., stability. In Section~\ref{sec:exponential_moments} we derive exponential moment bounds, both for the process started from a large class of initial conditions and for its stationary version. All proofs are collected in Section~\ref{sec:proofs}.

\section{Motivation}\label{sec:motivation}
Loglinear Hawkes processes are very similar to `spike response models' -- a well-recognized modeling class for neural spike trains in the field of computational neuroscience. The starting point of the concept is a deterministic model for the firing times of a neuron counted by a (deterministic)
 point process $N$ on $(0,\infty)$, with
the \emph{membrane potential}
$u$ (the difference between the voltage inside and outside the cell, `inside minus outside', typically measured in mV) defined as
$$
u(t) :=\nu + \int_{(0,t)}h(t - s)N(\d s)  + I(t),\quad t\geq 0;
$$
see, e.g., \cite[eq.~(9.1)]{gerstner14}.
Here, $I(t)$ denotes the \emph{input current}, e.g., describing the influence of the spikes of other neurons or a battery that is connected to the neuron.
Note that in this context, the memory function $h$ is usually chosen to be continuous. Memory functions extracted from data typically exhibit a very negative `refractory period' of about 1--3\,ms directly after zero, followed either by a monotone increase to zero or by a rapid increase to a short positive segment called `overshoot', and then a less steep decreasing tail; see, e.g., the `history filters' in \cite[Fig. 2]{gerhard17}. If the membrane potential $u(t)$ becomes greater than some threshold $\theta\in\mathbb{R}$ (between about $-60$ and $-40$\,mV; see \cite{jolivet06}), then the neuron fires. By the refractory property of the memory function, the membrane potential consequently falls below $\theta$ and then increases again with $h$. As a next conceptual step, in order to model the variability in the data, the point process of firing times $N$ is considered to be random with firing rate
\begin{equation}
\lambda(t) = \rho_0 \exp\biggl\{ \frac{u(t)-\theta}{\Delta}\biggr\},\quad t \geq 0, \label{eq:firing_rate}
\end{equation}
where
$
u
$
is again the membrane potential, $\theta  $ is the \emph{`leaky' threshold}, $\rho_0>0 $ is the \emph{firing rate at threshold},
and $\Delta>0$ is a parameter for the `sharpness of the threshold': for $\Delta \downarrow 0$ one expects a return to the deterministic model. Obviously, in \eqref{eq:firing_rate} $\rho_0$ could be eliminated (for fixed $\Delta$) by picking an alternative threshold $\widetilde{\theta} := \theta/\Delta + \log(\rho_0) $, but this formulation gives a good interpretation of what is happening physiologically.\par

The first mention  of this model seems to be \cite{gerstner92}. Later, e.g., \cite{truccolo05}, \cite{pillow08}, and \cite{truccolo16} apply an exponential transfer function for relating
neural spiking activity to spiking history. \cite{gerstner14}, Chapter 9, gives a summary for the model $\phi(x) = e^{\nu + x}$. \cite{jolivet06} shows an excellent fit of the model to experimental data. This particularly good fit may be attributed to the fact that the exponential transfer function renders negative feedback very potent. In this scenario, a spike in the system reduces the firing rate by a factor  $e^{h(0)}$, with $h(0)$ negative, rather than by a fixed jump down of size around  $h(0)$, as would be the case with the rectified linear unit $\phi(x) = (\nu + x)^+$  or with the softplus transfer function  $\phi(x) = \log(1 + e^{\nu + x})$. As another advantage, note that the passage from the noisy to the deterministic model as $\Delta\downarrow 0$ in \eqref{eq:firing_rate} is satisfactorily fast with the exponential rate function.

The cited literature does not consider theoretical aspects of the model such as existence, nonexplosion, or stability. Nonexplosion issues are circumvented by introducing a strong or even `absolute' refractory period that is present in the data.
If the memory function has an absolute refractory period $\tau>0$, then the point process cannot have more than $1/\tau$ events per time unit. This assumption makes sense in this context: the spike itself -- though modeled as a point in time -- actually has a duration of 1--3\,ms, and during such a spike, we cannot expect another spike. Nevertheless, `stability' is an issue in the neuroscience literature:  for example, \cite{gerhard17} notes that simulations from calibrated leaky spike response models can result in spike trains that are unrealistically active. That paper applies a technique called `quasi renewal approximation' as developed in \cite{naud12} to provide some heuristic stability criteria. See Example~\ref{ex:enveloped_memory} for specific choices of the memory function that are used in the computational neuroscience literature and covered by our results.

Our goal is to provide more formal results on the time asymptotics of the process $N$ in these kinds of models without absolutely refractory periods. In this paper, we consider the one-dimensional case (one neuron) and we ignore the exogenous time-dependent input $I(t)$. In view of the success of the loglinear Hawkes process in computational neuroscience, we expect that there are other fields of application where the multiplicative structure of loglinear models outperforms linear or bounded models. In particular, if the exponential transfer function is the gold standard for biological neurons, it seems promising to adapt exponential transfer to Hawkes-based artificial neural networks as first proposed in \cite{mei17} (with softplus transfer function).
\section{Definition and representation}\label{sec:existence}
We start with the definition of the loglinear Hawkes process.
			\begin{definition}[Loglinear Hawkes process]\label{def:loglinear_hawkes_process}
		Let $N$ be a simple point process on $\mathbb{R}$ with explosion time $T_\infty \in (0,\infty]$, and let $N^-$ be a point process supported on $ (-\infty,0] $. The point process $N$  is called a
		\emph{loglinear Hawkes process with initial condition $N^-$}
		 if $N$ equals $N^-$  on the closed negative half line,
and on the positive half line $N$ has the $\{\mathcal{F}_t^N\}_{t> 0}$-intensity
		\begin{equation}
			\lambda(t) = \exp\biggl\{\nu  +  \int_{(-\infty,t)}h(t - s)N(\d s)\biggr\}1\{t< T_\infty\},\quad t> 0,
			\label{eq:log_lin_dynamics}
		\end{equation}
		with \emph{log-baseline intensity} $\nu\in\mathbb{R}$ and $h$, the \emph{memory function}, a real-valued measurable function with support in $(0,\infty)$, where the integral in \eqref{eq:log_lin_dynamics} is assumed to be well-defined in $[-\infty,\infty)$, with the convention $e^{-\infty}:=0$.
				If $N^-((-\infty,0]) = 0$, then $N^-$ is called a \emph{void initial condition}.
		If $T_\infty = \infty$ a.s., we call the loglinear Hawkes process \emph{nonexplosive} and otherwise \emph{explosive}.
	\end{definition}
	The following assumption guarantees that a loglinear Hawkes process with initial condition $N^-$ exists (Theorem~\ref{thm:representation}). We write $h^+:=\max\{h,0\}$.
		\begin{assumption}\label{ass:h}
	Let $h$ be a memory function and let $N^-$ be a point process on $(-\infty,0]$. The function $h$ is left-continuous, and for all bounded intervals $I\subset(0,\infty)$
\begin{equation}\label{eq:ass_init_cond}
\sup_{t\in I}h(t)<\infty\quad\text{and}\quad \sup_{t\in I}\int_{(-\infty,0]}h^+(t-r)\,N^-(\d r)<\infty.
\end{equation}
\end{assumption}
Under Assumption~\ref{ass:h} the integral in \eqref{eq:log_lin_dynamics} is well-defined in $[-\infty,\infty)$ for every $t>0$.
\begin{example}\label{ex:ass_h} We present classes of pairs of initial conditions $N^-$ and memory functions $h$ satisfying Assumption~\ref{ass:h}; in all cases, $h$ is a left-continuous memory function.\\
\noindent(i) For the void initial condition, Assumption~\ref{ass:h} holds if and only if $h$ is bounded above on bounded intervals.\\
\noindent(ii) If $N^-( (-\infty,0] ) < \infty$ and $h$ is bounded above, then Assumption~\ref{ass:h} holds.\\
 \noindent(iii) If $\sup_{s>0} h(s) < \infty$ and $h$ is eventually nonpositive, i.e., there exists some $t_0>0$ such that $h(t) \leq 0$ for $t > t_0$, then Assumption~\ref{ass:h} holds for all initial conditions $N^-$. Indeed,
\begin{align*}
\sup_{t>0}&\int_{(-\infty,0]} {h^+(t -r)} N^-(\d r)
= \sup_{ t>0}\int_{(-t_0,0]} {{h^+}(t -r)} N^-(\d r)\\
&\leq \sup_{t>0}\int_{(-t_0,0]} \Bigl({\sup_{s>0}h(s)}\Bigr)^+ N^-(\d r)
\leq N^-{((-t_0,0])}\Bigl(\sup_{s>0}h(s)\Bigr)^+ < \infty
\end{align*}
because $N^-((-t_0,0])$ is finite by the assumption that $\lim_{n\to-\infty}T_n=-\infty$.\\
\noindent(iv)  If there exists a finite $c\geq 0$
such that $\E[N^-(B)] \leq c \int_B\d t$ for all $B\in\mathcal{B}( (-\infty,0] )$, then Assumption~\ref{ass:h} holds provided that there exists an integrable decreasing function $\widetilde{h}$ on $ [0,\infty)   $ such that  $h\leq  \widetilde{h}$.
Indeed, $\widetilde h$ is nonnegative, being decreasing and integrable, so that $h^+\leq \widetilde h\leq \widetilde h(0)<\infty$, and for any bounded interval $I\subseteq(0,\infty)$, we have
\begin{align*}
\sup_{t\in I} &\int_{(-\infty,0]} {h^+(t -r)} N^-(\d r)
\leq \sup_{t\in I}\int_{(-\infty,0]} {\widetilde{h}(t -r)} N^-(\d r)\\
&\leq \sup_{t\in I}\int_{(-\infty,0]} {\widetilde{h}( -r)} N^-(\d r)
= \int_{(-\infty,0]} {\widetilde{h}(-r)} N^-(\d r) < \infty\quad\text{a.s.}
\end{align*}
since
$
\E\bigl[\int_{(-\infty,0]} {\widetilde{h}(-r)} N^-(\d r)\bigr] \leq c \int_{0}^\infty {\widetilde{h}(r)} \d r<\infty.
$
\end{example}

In the case of a void initial condition, we have $T_{-n} = -\infty$, $n\in \N_0 := \{0\}\cup \N$, in \eqref{eq:jump_times}. Note that \cite[above Definition 1]{bremaud96} defines an `initial condition' of a nonlinear  Hawkes process
in terms of a property of the restriction of $N$ to $ (-\infty,0] $ whereas we define it in terms of a point process with support on $ (-\infty,0] $. The finiteness of the explosion time $T_\infty$ of a loglinear Hawkes process depends on the choice of the initial condition, as well as on the parameters of the $\{\mathcal{F}_t^N\}_{t\geq 0}$-intensity. We start with a key technical result.

\begin{lemma}[Poisson embedding]\label{lemma:poisson_embedding}
 Let $\mathcal{N}$ be a Poisson process  on $ (0,\infty)   \times  [0,\infty)   $ with intensity~1 and
 $$
 \mathcal{F}_t^{\mathcal{N}}:=\sigma(\mathcal{N}((a, b] \times C):\ 0 < a<b\leq t,\ C\in \mathcal{B}( [0,\infty)   )).
 $$
  Let $\lambda$  be a nonnegative  $\{\sigma(\mathcal{F}_t^{\mathcal{N}},\mathcal{F}_0^{{N}})\}_{t>0}$-predictable process on $ (0,\infty)   $.
 Let $N^-$ be a point process on $ (-\infty,0] $ independent of $\mathcal{N}$. Assume that $N$ is a point process on $\mathbb{R}$ equal to $N^-$ on $(-\infty, 0] $ and such that
\begin{equation}\label{eq:pois_thinning_intensity}
N(A)  = \int_{A \times  (0,\infty)   } 1\{ z \leq\lambda(t)\} \mathcal{N}(\d t \times \d z)\quad \text{for all $A\in\mathcal{B}( (0,\infty)   )$}.
\end{equation}
Then, with $\mathcal{F}_t := \sigma(\mathcal{F}_t^{\mathcal{N}},\mathcal{F}_0^{{N}}), t>0$,  $\lambda$ is the $\{\mathcal{F}_t\}_{t> 0}$-intensity of $N$. Furthermore, if $\lambda$ is $\{\mathcal{F}_t^{{N}}\}_{t> 0}$-predictable, then $\lambda$ is the $\{\mathcal{F}_t^{{N}}\}_{t> 0}$-intensity.
\end{lemma}
In what follows, for a Poisson random measure $\mathcal{M}$ on $\R\times\R_{\geq 0}$ with intensity 1 and a random variable $Z$, we write
\begin{equation*}
\mathcal{M}((a,b]\times [0, Z])(\omega):= \int_{(a,b]\times  [0,\infty)    } 1\{z\leq Z(\omega)\}\mathcal{M}(\d t \times \d z)(\omega).
\end{equation*}
Fix $\nu\in\R$ and a memory function $h$.
Denote by $\mathfrak{M}(\mathcal{X})$ the space of locally finite point measures on a measurable space $\mathcal{X}$ with $\sigma$-algebra generated by the evaluations (see, e.g., \cite[Introduction]{kallenberg2017}). For a memory function $h$ let $\mathfrak{M}_h((-\infty,0])$ be the set of $\kappa\in\mathfrak{M}((-\infty,0])$ such that $(h,\kappa)$ satisfies Assumption~\ref{ass:h}.
For a measure $\mu\in\mathfrak{M}((0,\infty)\times [0,\infty)) $, a nonnegative measurable function $g$ on $\R$ and $a<t$, we write
\begin{equation}\label{eq:acceptance_count}
\mathcal{I}(\mu,g;(a,t])
:= \int 1\{s\in (a,t]\}\,1\{z\leq g(s)\}\,\mu(\d s\times \d z)
\end{equation}
for the number of points of $\mu$ in $(a,t]$ that are \emph{accepted} by $g$, i.e., that lie below or on the graph of $g$. Note that $t\mapsto \mathcal{I}(\mu,g;(a,t])$ is nondecreasing and right-continuous.
Thus, if $t\mapsto \mathcal{I}(\mathcal{M},\lambda;(a,t])$ is adapted to the filtration generated by $\mathcal{M}$ and $\lambda$, then $T$ is a stopping time with respect to this filtration.

With the convention $\inf\emptyset:=\infty$, for $n\in\N_0$ define the functions $\phi_n:  \mathfrak{M}_h((-\infty,0])\times \mathfrak{M}((0,\infty)\times [0,\infty)) \rightarrow (0,\infty]$ recursively by $\phi_0(\kappa, \mu) := 0$ and
		\begin{equation}\label{eq:phi_def}
		\phi_{n}(\kappa, \mu) :=  \inf\{t> \phi_{n-1}(\kappa, \mu):\ \mathcal{I}(\mu,\lambda^{(n)}(\cdot, \kappa, \mu);(\phi_{n-1}(\kappa, \mu),t])>0\},\quad n\in\N,
		\end{equation}
		where
		\[
		\lambda^{(n)}(t,\kappa, \mu):=
		\exp\biggl\{\nu  +  \int_{(-\infty, 0]}h(t-s)\kappa(\d s) + \sum_{k=1}^{n-1}  h(t- \phi_k(\kappa, \mu )) 1\{\phi_k(\kappa, \mu) < \infty \}\biggr\}.
		\]
Finally, define the point measure $\Psi(\kappa,\mu)$ on $(\R,\mathcal{B}(\R))$ by
		\begin{equation*}
		\Psi(\kappa, \mu )(A) := \kappa(A) + \sum_{n \geq 1} 1\{\phi_n(\kappa, \mu) \in A \}, \quad A\in\mathcal{B}(\R).
		\end{equation*}		
\begin{theorem}[Existence and representation]\label{thm:representation}
Let $N^-$ and $h$ be as in Assumption~\ref{ass:h} and $\nu\in\R$. Let $\mathcal{N}$ be a Poisson random measure on $(0,\infty)\times  [0,\infty)$ with intensity 1 and independent of $N^-$. Then $N := \Psi(N^-, \cN)$ is a loglinear Hawkes process with initial condition $N^-$	and $\{\mathcal{F}_t^N\}_{t>0}$-intensity of the form	
\begin{equation}\label{eq:constructed_intensity}
\lambda(t)=
\exp\biggl\{\nu  +  \int_{(-\infty, t)}h(t - s)N(\d s)\biggr\}1\{ t < T_\infty\}, \quad t > 0,
\end{equation}
where $T_\infty:=\lim_{n\to\infty}\phi_n(N^-, \cN).$
 \end{theorem}

\begin{remark}[Uniqueness in law]\label{rem:uniqueness}
Theorem~\ref{thm:representation} establishes existence. By Definition~\ref{def:loglinear_hawkes_process} of a loglinear Hawkes process, its law is in fact determined by its $\{\mathcal F^N_t\}_{t>0}$-intensity \eqref{eq:log_lin_dynamics}. So that on $(T_{n-1},T_n]$ it is the conditional distribution of $T_n$ given $\mathcal F^N_{T_{n-1}}$ is determined by this intensity. Consequently, the conditional law of $(T_n)_{n\geq 1}$ given $N^-$ is determined by $\nu$, $h$ and $N^-$, and any two loglinear Hawkes processes with the same $\nu$ and $h$ and the same law of the initial condition coincide in law on $(-\infty,T_\infty)$; see \cite[Proposition 7.2.]{daley03}. Thus, for the distributional results below, assuming $h$ and $N^-$ satisfy Assumption~\ref{ass:h}, we may and do use without loss of generality the representation $N=\Psi(N^-,\cN)$ as in Theorem~\ref{thm:representation}.
\end{remark}

\begin{remark}\label{rem:ass_h}
The two supremum bounds in Assumption~\ref{ass:h} cannot be dropped: local integrability of $h$ alone does not suffice. For example, let $h(t):=t^{-1/2}$, $t>0$, and consider a void initial condition. In the construction of Theorem~\ref{thm:representation}, the intensity after the first jump time $T_1$ satisfies
$$
\int_{T_1}^{T_1+\varepsilon}\exp\{\nu+h(t-T_1)\}\,\d t = e^\nu\int_0^\varepsilon \exp\{t^{-1/2}\}\,\d t=\infty,\quad \varepsilon>0,
$$
so that the thinning construction yields $T_1=T_2=\dots=T_\infty<\infty$, in contradiction with \eqref{eq:jump_times}. The first supremum in \eqref{eq:ass_init_cond} excludes this; the second one plays the same role for the contribution of the initial condition.
\end{remark}
\begin{remark}\label{rem:fixed point} Note that, in the spirit of \cite{bremaud96}, we are dealing here with a fixed point problem: we want to construct a simple point process $N$ whose predictable compensator is a functional of $N$ itself. Theorem \ref{thm:representation} provides the solution to this problem: $N = \Psi(N^-, \cN)$.\end{remark}

\section{Explosion and nonexplosion}\label{sec:nonexplosion}
In this section we present sufficient conditions for nonexplosive and explosive behavior of loglinear Hawkes processes.

\begin{theorem}[Explosion and nonexplosion] Let $N$ be a loglinear Hawkes process with  log-baseline intensity $\nu\in\mathbb{R}$, memory function $h$, and initial condition $N^-$.
\label{thm:nonexplosion}
\begin{enumerate}[label =(\roman*)]
 \item \label{item:sufficient}  If there exists $\delta>0$ such that $\sup_{t\in(0, \delta]}h(t) \leq 0$, then $N$ is nonexplosive.
\item \label{item:necessary} If there exists $\delta>0$ such that $h_{\delta} := \inf_{t\in(0,\delta]} h(t)> 0$ and
\begin{equation}
J(N^-) := \inf_{t\in(0,\delta]}\int_{(-\infty,0]} {h(t -r)} N^-(\d r)>-\infty,  \label{ass:initial_cond2}
\end{equation}
 then $N$ is explosive.
 \end{enumerate}
\end{theorem}

\begin{example}\label{ex:explosion} We present classes of pairs of initial conditions $N^-$ and memory functions $h$  for which Theorem~\ref{thm:nonexplosion} decides between explosion and nonexplosion; in all cases, $h$ is a left-continuous memory function:\\
\noindent(i) Let $N^-$ be the void initial condition and let $h$ be bounded above on bounded intervals, so that Assumption~\ref{ass:h} holds by Example~\ref{ex:ass_h}(i). Then \eqref{ass:initial_cond2} holds trivially with $J(N^-)=0$, and $N$ is nonexplosive if $\sup_{t\in(0,\delta]}h(t)\leq 0$ and explosive if $\inf_{t\in(0,\delta]}h(t)>0$ for some $\delta>0$.\\
\noindent(ii) The same holds if $N^-((-\infty,0])<\infty$ and $h$ is bounded: Assumption~\ref{ass:h} holds by Example~\ref{ex:ass_h}(ii), and \eqref{ass:initial_cond2} holds since $J(N^-)\geq N^-((-\infty,0])\inf_{s>0}h(s)>-\infty$.\\
\noindent(iii) Using Example~\ref{ex:ass_h}(iii) we see that if $\sup_{s>0}h(s)<\infty$, $h$ is eventually nonpositive, and $\sup_{t\in(0,\delta]}h(t)\leq 0$ for some $\delta>0$, then $N$ is nonexplosive for \emph{every} initial condition $N^-$.\\
\noindent(iv) If $\inf_{s > 0}h(s)>-\infty$ and $h$ is eventually nonnegative, i.e., there exists some $t_0>0$ such that $h(t) \geq 0$ for $t > t_0$, and $\inf_{t\in(0,\delta]}h(t)>0$ for some $\delta>0$, then $N$ is explosive for \emph{every} initial condition $N^-$ satisfying Assumption~\ref{ass:h}, because \eqref{ass:initial_cond2} holds for all initial conditions $N^-$. Indeed, we have for $t\geq 0$ that
\begin{align*}
\int_{(-\infty,0]}& {h(t -r)} N^-(\d r)  =
\int_{(-\infty,-t_0]} {h(t -r)} N^-(\d r) + \int_{(-t_0,0]} {h(t -r)} N^-(\d r)\\
&\geq  \int_{(-t_0,0]} {h(t -r)} N^-(\d r)\geq N^-((-t_0,0]) \inf_{s > 0}h(s) > -\infty
\end{align*}
because $N^-((-t_0,0])$ is finite by the assumption $\lim_{n\to-\infty}T_n=-\infty$. Taking the infimum over $t\in(0,\delta]$ we see that \eqref{ass:initial_cond2} holds.\\
\noindent(v) Using  Example~\ref{ex:ass_h}(iv) we see that if there exists a finite $c\geq 0$ such that $\E[N^-(B)]\leq c\int_B\d t$ for all $B\in\mathcal{B}((-\infty,0])$, and if $h\leq\widetilde h$ for an integrable decreasing function $\widetilde h$ on $[0,\infty)$ with $\sup_{t\in(0,\delta]}h(t)\leq 0$ for some $\delta>0$, then $N$ is nonexplosive.
\end{example}
\section{Bin count domination}\label{sec:bin_count_domination}
For $\Delta >0$, we define bins $((n-1)\Delta, n \Delta],\ n\in\Z,$ and the corresponding bin counts of a loglinear Hawkes process $N$ as
$N_n := N(((n-1)\Delta, n \Delta]), n\in\Z.
$
The technique for deriving sufficient stability conditions exploits an autoregressive integer-valued time series with a sublinear drift that
 dominates these bin counts.
 The starting point is a monotonicity lemma that uses a coupling by thinning from a common Poisson random measure
 $\mathcal{M}$ on
$\R_{\geq 0}\times\R_{\geq 0}$ with intensity 1 independent of some $\sigma$-field  $\mathcal{G}_0$.
For $i = 1,2$ let $h_i$ be a function as in Assumption~\ref{ass:h} (with void initial condition) and $\nu_i$  a $\mathcal{G}_0$-measurable stochastic process.
We construct
the point process $M_i$ by thinning from $\mathcal{M}$ in the way described in Section \ref{sec:existence} (with the constant $\nu$ replaced by the $\mathcal{G}_0$-measurable process $\nu_i$, and with void initial condition), so that $M_i$ has intensity
\begin{equation*}
\lambda_i(t) = \exp\biggl\{\nu_i(t) + \int_{(0,t)} h_i(t-s)\, M_i(\d s)\biggr\},\qquad t\geq 0.
\end{equation*}
\begin{lemma}[Monotonicity]\label{lemma:monotonicity}
If
$$
\nu_1(t) \le \nu_2(t) := \nu_2,\quad t\geq 0,
$$
for some $\mathcal{G}_0$-measurable finite random variable $\nu_2$, and the
memory functions $h_i$, $i = 1,2$, satisfy for some constant $a\geq0$
$$
h_1(t) \leq h_2(t) \equiv -a, \quad t\geq 0,
$$
then
\[
M_1([0,t]) \le M_2([0,t]),\quad t\geq 0.
\]
\end{lemma}

In order to obtain bin count domination,
we introduce binwise constant upper bounds for the memory function. In particular, the upper bound in the first bin is nonpositive which immediately gives us nonexplosion from Theorem~\ref{thm:nonexplosion}. Furthermore, note that this is in line with applications in neuroscience where a refractory period right after a spike is a physical necessity.
\begin{definition}[Memory envelope]\label{def:memory_envelope}
 We call a memory function $h$ an \emph{enveloped memory function} if it satisfies
\begin{align*}
&h(t) \leq -a, & t \in (0,  p^{(-)}\Delta  ], \\
& h(t) \leq b,  &t \in ( p^{(-)}\Delta  , p \Delta  ], \\
&h(t) = 0,     & \text{otherwise},
\end{align*}
for some constants $a>0, b\geq 0, \Delta >0$,  $p^{(-)}\in\N$, $ p^{(+)}\in\N_0$, and $p := p^{(-)}+ p^{(+)}$. We call a loglinear Hawkes process with such $h$ a \emph{loglinear Hawkes process with enveloped memory}.
\end{definition}
Note that for an enveloped memory function $h$, Assumption~\ref{ass:h} holds for every initial condition $N^-$: indeed, $h\leq b$, and for $t$ in a bounded interval $I\subset(0,\infty)$ we have $\int_{(-\infty,0]}h^+(t-r)N^-(\d r)\leq b\,N^-((-p\Delta,0])<\infty$ because $h$ vanishes outside $(0,p\Delta]$. In the special case when $p^{(+)} = 0$, we have $h \leq 0$ so that Assumption~\ref{ass:h} holds trivially. Hence, in what follows, no further assumptions on $N^-$ are needed beyond those explicitly stated. \begin{example}\label{ex:enveloped_memory}
The following memory functions are enveloped in the sense of Definition~\ref{def:memory_envelope}; all of them are in use in computational neuroscience:\\
\noindent(i) Binwise constant type: for $\gamma_1, \dots, \gamma_p\in\R$ with $\gamma_n \leq -a$ for $n \leq  p^{(-)} $ and $\gamma_n \leq b$ for $ p^{(-)} < n \leq p$,
$
h(t) := \sum_{n = 1}^{p} \gamma_n 1\{t \in ((n-1)\Delta, n\Delta]\}.
$
This is essentially the `autoregressive spiking history model' of \cite[eq.~(7)]{truccolo05}. The paper uses this as a building block in a larger model. For a calibration of these autoregressive coefficients on data from a primary motor cortex cell with $p = 120$ and bin width $\Delta = 1$\,ms, see~\cite[Fig.~5B]{truccolo05}.\\
\noindent(ii) Raised cosine bumps: for weights $w_j\in\R$ and parameters $\alpha, c>0$, $\varphi_1<\dots<\varphi_K$,
$$
h(t) := \sum_{j = 1}^{K} w_j\, \beta_j(t),\
\beta_j(t) := \frac{1}{2}\bigl(1 + \cos(\alpha\log(t + c) - \varphi_j)\bigr) 1\{|\alpha\log(t + c) - \varphi_j| \leq \pi\},
$$
provided the weights are such that $h \leq -a$ on $(0,  p^{(-)}\Delta  ]$ and $h\leq b$ on $(  p^{(-)}\Delta  , p\Delta]$. Such log-spaced bumps carry the `post-spike filter' in \cite[Fig.~1d]{pillow08}.\\
\noindent(iii) Truncated sums of exponentials: for $\alpha_j\in\R$ and $\beta_j>0$,
$$
h(t) := \Big(\sum_{j = 1}^{K}\alpha_j e^{-\beta_j t}\Big)1\{t \leq p\Delta\},
$$
again with $ p^{(-)} $ small enough that $h \leq -a$ on $(0,  p^{(-)}\Delta  ]$. Untruncated, this is the choice of \cite[Section~4]{truccolo16} and \cite[eq.~(6.42)]{gerstner14}; for $K = 2$ with $\alpha_1 <0<\alpha_2$, $\alpha_1 + \alpha_2 <0$, and $\beta_1 > \beta_2$, it models a refractory period followed by an overshoot.\\
\noindent(iv) Truncated exponential reset kernel: for $\eta_0, \tau>0$,
$
h(t) := -\eta_0 e^{-t/\tau}1\{t \leq p\Delta\},
$
with $a := \eta_0 e^{- p^{(-)}\Delta  /\tau}$ and $b:=0$, as $h$ is nonpositive. This is the spike afterpotential of \cite[eq.~(6.33)]{gerstner14}.

\end{example}
In what follows we denote by $\mathrm{Exp}(1)$ the exponential distribution with rate parameter 1.
\begin{theorem}[Bin count bounds]\label{thm:bin_count_domination} Let $N$ be a loglinear Hawkes process with enveloped memory. There exist random variables $E_{n,k}\iidsim \mathrm{Exp}(1)$, $n,k\in\N$, such that for each $n\in\N$ the block
$(E_{n,k})_{k\in\N}$ is independent of
$\F^N_{(n-1)\Delta}\vee\sigma((E_{m,k})_{k\in\N},\ m<n)$, and
\begin{equation}\label{eq:bin_count_bound}
N(((n-1)\Delta, n \Delta]) \leq \sum_{k = 1}^{\infty} 1\{S_{n,1} + \dots + S_{n,k} < \Delta \},
\qquad n\in\N,
\end{equation}
where
\begin{equation}\label{eq:S_nk_def}
S_{n,k} :=  \exp\{-(\nu + N(I^{(+)}_n)\, b - (k-1)a)\}E_{n,k}
\end{equation}
and
$I^{(+)}_n := ((n - p - 1)\Delta,\ (n - p^{(-)})\Delta].
$\end{theorem}

Next, we use the exponential random variables from Theorem~\ref{thm:bin_count_domination} to construct an integer-valued time series that bounds the bin counts of the loglinear Hawkes process and has an autoregressive structure.

\begin{definition}\label{def:bcds}
Let $N$ be a loglinear Hawkes process with enveloped memory.
 Define its \emph{bin count dominating sequence}
\begin{equation*}
\overline{N}_n :=
\begin{cases}
N(((n-1)\Delta, n\Delta]), &  n\leq 0,\\
\sum_{k = 1}^{\infty}1\{\overline{S}_{n,1} + \dots + \overline{S}_{n,k} < \Delta\},& n>0,
\end{cases}
\end{equation*}
where
\begin{equation*}
\overline{S}_{n,k} := \exp\Bigl\{-\Bigl(\nu + b\sum_{l = p^{(-)}}^{p}\overline{N}_{n-l} - (k-1)a\Bigr)\Bigr\}E_{n,k}, \quad k\geq 1,
\end{equation*}
and $E_{n,k}\iidsim\mathrm{Exp}(1)$ are
 the random variables from \eqref{eq:S_nk_def} in Theorem~\ref{thm:bin_count_domination}.
\end{definition}
Note that the block $(E_{n,k})_{k\in\N}$ is independent of $\F^N_{(n-1)\Delta}\vee\sigma((E_{m,k})_{k\in\N,m\leq n-1})$ for each $n\in\N$, and hence of $\sigma(\oN_m,\ m < n)$.
Both the sequence $\oN_n$ and the bin count sequence $N_n$ are ultimately constructed from the initial condition $N^-$ and the Poisson random measure $\mathcal{N}$ on $(0, \infty)  \times \R_{\geq 0}$.  We show that $\oN_n$ actually dominates the bin counts:
\begin{theorem}[Bin count domination]\label{thm:bin_count_domination_sequence}
Let $N$ be a loglinear Hawkes process with enveloped memory.  Let $(\overline{N}_n)_{n\in\Z}$ be its bin count dominating sequence. Then
\begin{equation}\label{eq:bin_count_ub}
N(((n-1)\Delta, n\Delta ])\leq \overline{N}_n,\quad n \in\Z.
\end{equation}
\end{theorem}
We apply the autoregressive structure of the constructed time series to derive a linear bound for its drift:
\begin{theorem}[Sublinear autoregression]\label{thm:moment_bound} Let $N$ be a loglinear Hawkes process with enveloped memory.  Let $(\overline{N}_n)_{n\in\Z}$ be its bin count dominating sequence. Then
\begin{equation*}
\E\bigg[\oN_n\mid\sigma(\overline{N}_{n-1} , \overline{N}_{n-2} , \dots)\bigg]
 \leq
\frac{b}{a} \sum_{l = p^{(-)}}^{{p}} \oN_{n- l}+  1+ \frac{\Delta e^\nu}{1 - e^{-a}} ,\quad n\in\N.
\end{equation*}
\end{theorem}
We encode the bin count dominating sequence as a $p$-dimensional Markov chain, so that the derived sublinear drift can serve as a Lyapunov-type criterion.
To this end, we define
 $(X_n)_{n\geq 0}\subseteq\N_0^p$ by
\begin{equation}\label{eq:X_def_alt}
X_n := (\overline{N}_n, \overline{N}_{n-1}, \dots, \overline{N}_{n-p+1}),\quad n\geq 0,
\end{equation}
and $\mathcal{F}^{X}_n := \sigma(X_0, X_1, \dots, X_n)$, $n\geq 0$.
\begin{lemma}\label{lemma:markov_chain}
 The sequence $(X_n)_{n\geq 0}$ defined in \eqref{eq:X_def_alt} is an $(\mathcal{F}^X_n)_{n\geq 0}$-homogeneous Markov chain with initial condition
 $
 X_0 = (N_0, N_{-1}, \dots, N_{-p+1}).
 $
\end{lemma}
We call $(X_n)_{n\geq 0}$ the \emph{bin count dominating Markov chain} of $N$. Below, we use the notation
$
\E_x[G((X_n)_{n\geq 0})] : = \E[G((X_n)_{n\geq 0})|X_{0} = x ], x\in\N_0^p,
$
for bounded measurable $G$.

\begin{lemma}[$mp$-step drift bound]\label{lemma:foster} Let $N$ be a loglinear Hawkes process with enveloped memory and  $(X_n)_{n\geq 0}$  the corresponding Markov chain from \eqref{eq:X_def_alt}. Assume that
\begin{equation}\label{eq:rho_bound}
 \rho := \frac{b}{a}(p^{(+)}+1) < 1,
\end{equation}
for the constants $a, b>0$ and $p^{(+)}\in\N_0$ from Definition~\ref{def:memory_envelope}.
Define the Lyapunov function $V: \N_{0}^p \to [0,\infty)$ by
 \begin{equation}\label{eq:lyapunov}
   V(x) =\sum_{k=1}^p x_k, \quad x = (x_1,\dots, x_p)\in\N_0^p.
 \end{equation} Then, for any $m \in \N$,
\begin{equation}\label{eq:driftR}
\mathbb{E}_x[V(X_{ mp})] \leq p \rho^m V(x) + \frac{pC}{1-\rho}, \quad x\in\N_0^p,
\end{equation}
where $C := 1+ {\Delta e^\nu}/({1 - e^{-a}})$.
\end{lemma}
In the following section, we use the last result to give conditions for a finite expected regeneration time of a loglinear Hawkes process $N$.

\section{Stability}\label{sec:stability}
\noindent
Throughout, for a locally finite counting measure $\mu$ on $\mathbb{R}$ and $t\in\mathbb{R}$, we write
$\mu(\cdot + t)$ for the shifted measure $A\mapsto \mu(A+t)$, $A\in\mathcal{B}(\mathbb{R})$; equivalently,
$$
\int f(s)\,\mu(\d s + t) = \int f(s-t)\,\mu(\d s)
$$
for every nonnegative measurable $f$. Thus the measure $\mu(\cdot + t)$ is the measure $\mu$ with the time origin moved to $t$.
Recall that a point process $\Nstat $ on $\mathbb{R}$ is \emph{stationary}
if the laws of $\Nstat (\cdot)$ and $\Nstat (\cdot + t)$ coincide for each $t\in\mathbb{R}$. Intuitively, `stability' means `convergence to a stationary state'. Specifically, we follow the definition in \cite[Definition 1]{bremaud96}. Recall that
$\|\mu\|_{\mathrm{TV}} := \sup_{A\in\mathcal{E}}| \mu(A)|$ for a signed measure $\mu$ on a measurable space $(E, \mathcal{E})$.

 \begin{definition}[Stability] \label{def:stability}
 	 The dynamics \eqref{eq:log_lin_dynamics}, extended to $\mathbb{R}$ by the same formula, are \emph{stable in variation (resp. stable in distribution) with respect to an initial condition $N^-$}, a point process supported on $ (-\infty,0] $, if
\begin{enumerate}[label =(\roman*)]
	\item \label{item:stationary_version} there exists a stationary point process $\Nstat $ that follows the dynamics \eqref{eq:log_lin_dynamics} with $(\mathcal{F}_t^{\Nstat})_{t\in\mathbb{R}}$-intensity satisfying \eqref{eq:intensity} for all $a,b\in\mathbb{R}$ such that $a<b$,
\item for a point process $N$ with initial condition $N^-$ on $ (-\infty,0] $ and following the dynamics  \eqref{eq:log_lin_dynamics} on $ (0,\infty)$, we have $N(\cdot + t)\to \Nstat (\cdot)$ in variation as $t\to\infty$, i.e.,
\begin{equation*} \Big\|\L((N(\cdot + t))|_{(0,\infty)}) - \L(\Nstat|_{(0,\infty)})\Big\|_{\mathrm{TV}}
\ \longrightarrow\ 0,\qquad t\to\infty,
\end{equation*}
\Big(resp.\ in distribution, i.e.,
$$
\E\biggl[\int_{ (0,\infty)   } f(x)\, N( \d x + t)\biggr]\ \longrightarrow\ \E\biggl[\int_{ (0,\infty)   } f(x)\, \Nstat(\d x)\biggr],\qquad t\to \infty,
$$ for all continuous functions $f: (0,\infty)\to  [0,\infty)   $ with compact support\Big).

\end{enumerate}
\end{definition}

Note that for point processes convergence in distribution coincides with convergence of the finite-dimensional distributions; see \cite[Theorem 11.1.VII]{daley03}. Also note that stability in variation implies stability in distribution.

 It is shown in \cite[Theorem 1]{bremaud96} that a sufficient condition for stability of the dynamics of nonlinear Hawkes processes with mild assumptions on the initial condition is that $\phi$ is Lipschitz with Lipschitz constant strictly less than $(\int_0^\infty |h(s)|\d s)^{-1}$. Exponential rate functions are in general not covered by this result. However, in the special case of a loglinear Hawkes process with nonpositive memory function $h$, we may apply \cite[Theorem 2]{bremaud96} on bounded Lipschitz rate functions.
\begin{theorem}[Stability of purely inhibitive case]\label{thm:sufficient_stability_condition}
Let $h$ be a nonpositive memory function satisfying $\int_{0}^\infty (1+t) |h(t)|  \d t<\infty$, and let $N^-$ be an initial condition such that
\begin{equation}\label{eq:initial_condition_stability}
 \lim_{t\to\infty} \int_{t}^\infty\int_{(-\infty, 0]} |h(s - u)| N^-(\d u)\d s = 0 \quad\text{a.s.}
\end{equation}
Then the dynamics \eqref{eq:log_lin_dynamics} of a loglinear Hawkes process is stable in variation with respect to the initial condition $N^-$. Moreover, for such a memory function $h$, there exists only one stationary process following the dynamics \eqref{eq:log_lin_dynamics} on the whole real line.
\end{theorem}
For cases with partly positive memory functions, we apply the bound from Lemma~\ref{lemma:foster} to derive a finite expected regeneration time.

\begin{theorem}[Finite expected regeneration time]\label{thm:regeneration_time}
Let $N$ be a loglinear Hawkes process with enveloped memory and with initial condition $N^-$ such that
\begin{equation}\label{eq:init_cond_s0}
 \E[N^-((-s_0, 0])] < \infty,
 \end{equation}
 where
$s_0 := p\Delta$. Define
$\tau := \inf\{t \geq s_0: N((t-s_0, t]) = 0\}.
$ If \eqref{eq:rho_bound} holds,
  then $\E[\tau] < \infty$.
\end{theorem}
 We define a `stopped' loglinear Hawkes process with enveloped memory starting from a void
initial condition:
\begin{definition}[Cycle]\label{def:cycle}
Let $N$ be a loglinear Hawkes process with enveloped memory as in
Definition~\ref{def:memory_envelope} with log-baseline intensity $\nu$ and memory function
$h$. Assume \eqref{eq:rho_bound} and set $s_0 := p\Delta$. Let $N^{(0)}$ be the loglinear Hawkes process with the same $\nu$ and $h$ and with void initial condition, and set
$
L := \inf\{t\geq s_0: N^{(0)}((t - s_0, t]) = 0\}.
$
The point process $C:=N^{(0)}|_{(0,L]}$ is called a \emph{generic cycle} of $N$, and $L$ its \emph{generic length}. Note that $C((-\infty,0])=0$, $L=\inf\{t\geq s_0: C((t - s_0, t]) = 0\}$, and $C$ has the stochastic
intensity
\begin{equation}\label{eq:C_intensity}
\lambda_C(t) := \exp\biggl\{\nu + \int_{(-\infty, t)} h(t - s) C(\d s)\biggr\}1\{t\leq L\},
\quad t\geq 0.
\end{equation}
\end{definition}
Theorem~\ref{thm:regeneration_time} yields
$m := \E[L] < \infty$ because a generic cycle $C$ has a void initial condition (so that \eqref{eq:init_cond_s0} holds trivially) and we assume  \eqref{eq:rho_bound}.
We now define two point processes by gluing together cycles. The Palm glue uses i.i.d. copies of the generic cycle.
The stationary glue reuses all of these cycles except the one that contains the origin; there, an independent size-biased first cycle is used, with the origin placed uniformly within it. Write $\P_C$ for the law of the generic cycle $C$ on the space of
locally finite counting measures on $\R$. For any locally finite counting measure $c$, let
$\ell(c):=\inf\{t\geq s_0:\ c((t-s_0,t])=0\}$
denote the deterministic length functional, so that $L=\ell(C)$.

\begin{definition}[Palm and stationary cycle glue]\label{def:cycle_glues}
Let $C$ be a generic cycle with generic length $L$ as in Definition~\ref{def:cycle}, let
$(C_k, L_k)_{k\in\Z}$ be i.i.d.\ copies of $(C, L)$, and define the \emph{Palm spine}
$$
\tau_n :=
\begin{cases}
-\sum_{k=n+1}^{0} L_k, &n\leq -1,\\
0,& n = 0,\\
\sum_{k=1}^{n} L_k, &n\geq 1.
\end{cases}
$$
 Define the \emph{Palm cycle glue} as a sum of shifted measures
\begin{equation}\label{eq:palm}
  \NPalm (\cdot)  := \sum_{k\in\Z} C_k(\cdot - \tau_{k-1}).
\end{equation}
Next, let $\widetilde C_1$ be a random cycle with the
size-biased law
\begin{equation}\label{eq:size_biased_law}
 \P_{\widetilde{C}_1}(\d c):=\frac{\ell(c)}{m}\,\P_C(\d c),
\end{equation}
let $\widetilde L_1:=\ell(\widetilde C_1)$ be its size-biased length, and let $U\sim\Unif(0,1)$, such that $U$, $\widetilde{C}_1$, and the $C_k$ are independent.
Define the \emph{stationary cycle spine}
\begin{equation*}
\taustat_n:=
\begin{cases}
-U\widetilde L_1-\sum_{k=n+1}^{0}L_k, & n\leq-1,\\
-U\widetilde L_1, & n=0,\\
-U\widetilde L_1+\widetilde L_1+\sum_{k=2}^{n}L_k, & n\geq1.
\end{cases}
\end{equation*}
and the \emph{stationary cycle glue}
\begin{equation}\label{eq:stat_version}
\Nstat\ :=\ \sum_{k\in\Z\setminus\{1\}}C_k(\cdot-\taustat_{k-1}) +\widetilde C_1(\cdot-\taustat_{0}).
\end{equation}
\end{definition}

Note that since $L_k\geq s_0>0$ for all $k$, we have $\tau_n\to\pm\infty$ a.s. and $\taustat_n\to\pm\infty$ as
$n\to\pm\infty$. Moreover,
each cycle $C_k$ is supported in $(0,L_k-s_0]$ (the window $(L_k-s_0,L_k]$
contains no points by the definition of $L_k$), so distinct summands in
\eqref{eq:palm} and \eqref{eq:stat_version} have disjoint supports and thus the sums in \eqref{eq:palm} and  \eqref{eq:stat_version} are locally finite.
A word of caution on the terminology: despite its name, $\NPalm $ is \emph{not} the Palm version of a Hawkes process in the sense of \cite{kirchner26}, where the reference points of the Palm distribution are the points of the Hawkes process itself.
The process $\NPalm $ has a regeneration time at zero and `glues' i.i.d.\ cycles forward and backward in time. That is, it uses as its spine the Palm version of the renewal process with interarrival distribution equal to the distribution $\P_L$ of $L$.

\begin{theorem}[Stationarity]
\label{thm:stationary_construction}
Let $N$ be a loglinear Hawkes process with enveloped memory and $\Nstat$ its stationary cycle glue. Assume \eqref{eq:rho_bound} and \eqref{eq:init_cond_s0}. Then
\begin{equation}\label{eq:stat}
  \E[f(\Nstat)] = \frac{1}{m}\E\biggl[\int_{0}^{\tau_1} f( \NPalm (\cdot +r) )\d r\biggr]
\end{equation}
for every nonnegative measurable functional $f$. Furthermore, $\Nstat$ is stationary.
\end{theorem}

\begin{theorem}\label{thm:int}
Let $N$ be a loglinear Hawkes process with enveloped memory. Assume that \eqref{eq:rho_bound} and \eqref{eq:init_cond_s0} hold.   Then the process
\begin{equation*}
\lambda(t):=\exp\biggl\{\nu+\int_{(-\infty,t)}h(t-s)\,\Nstat(\d s)\biggr\},\quad t\in\R,
\end{equation*}
is a stochastic $(\Fstat_t)_{t\in\R}$-intensity of $\Nstat$. In particular, $\Nstat$ is a loglinear Hawkes process with initial condition $N^- := \Nstat|_{(-\infty,0]}$.\end{theorem}
\begin{theorem}[Convergence in total variation]\label{thm:weak}
Let $N$ be a loglinear Hawkes process with enveloped memory. Assume \eqref{eq:rho_bound} and $\E[N^-((-s_0, 0])] < \infty$ for $s_0 = p\Delta$. Then
\begin{equation}\label{eq:TV_convergence}
\Big\|\L((N(\cdot + t))|_{(0,\infty)}) - \L(\Nstat|_{(0,\infty)})\Big\|_{\mathrm{TV}}
\ \longrightarrow\ 0, \qquad t\to\infty.
\end{equation}
In particular, $ N(\cdot + t)$ converges weakly to $\Nstat$ as $t\to\infty$.
\end{theorem}
The three theorems above give stability in variation of the intensity of a loglinear Hawkes process with enveloped memory under the stated assumptions. In particular, the processes from Example~\ref{ex:enveloped_memory} are stable provided \eqref{eq:rho_bound} holds.

We close this section with a simple necessary condition for stability of general loglinear Hawkes processes.
\begin{proposition}[Necessary stability condition]\label{prop:necessary_stability_condition}
Let $h$ and $\nu$ be the parameters of the stable dynamics of a  loglinear Hawkes process with $\int_0^\infty |h(t)| \d t < \infty$. Assume that $\Nstat $ is the stationary limit process  with $\{\mathcal{F}_t^{\Nstat}\}_{t\in\mathbb{R}}$-intensity $\lambdastat $ such that  $ \E[\lambdastat (t)] \equiv \E[\lambdastat (0)]= e^c$, $t\in\mathbb{R}$, for some $c\in\R$.
Then, necessarily,
\begin{equation}
\int_{0}^\infty h(s)\d s \leq (c - \nu ) e^{-c}\label{eq:ineq1}
\qquad\text{and}\qquad
\int_{0}^\infty h(s)\d s \leq  e^{-(1+\nu )}.
\end{equation}
\end{proposition}
\section{Exponential moments}\label{sec:exponential_moments}
Starting again from the bin count bounds of Theorem~\ref{thm:bin_count_domination} and following the
argument of the proof of Theorem~\ref{thm:moment_bound}, we obtain bounds for the exponential moments of the bin counts of
the loglinear Hawkes process with enveloped memory. In particular, the processes from Example~\ref{ex:enveloped_memory} have such moments provided their parameters satisfy \eqref{eq:rho_bound}.
\begin{theorem}[Exponential moment bound]\label{thm:exponential_moment_bound}
Let $N$ be a loglinear Hawkes process with enveloped memory and let $N_n = N(((n-1)\Delta, n\Delta])$, $n\in\Z$, denote its bin counts.
Assume that \eqref{eq:rho_bound} holds and that the initial condition of $N$ is such that
\begin{equation}\label{eq:initial_condition}
\sup_{{n} \leq 0} \E[ e^{\theta N_n}]\leq
m(\theta),\quad \theta \in [0, \theta_0],
\end{equation}
for some $\theta_0$ with $0 < \theta_0 < a$, and some function
$m: [0, \theta_0] \to [1, \infty)$ satisfying
\begin{equation}\label{eq:m_condition}
 c_{1}(\theta) :=  \prod_{k = 1}^\infty m(\rho^k\theta)< \infty,
 \quad \theta\in[0,\theta_0].
\end{equation}
 Then, for all $k\in\N$ and for all $\theta\in[0, \theta_0]$, we have
\begin{equation}\label{eq:exp_moment_bound}
\sup_{n\geq k} \E \biggl[\exp\biggl\{\frac{\theta}{k} N(((n- k)\Delta, n\Delta])\biggr\} \biggr]
\leq  c_{1}(\theta)c_2(\theta) < \infty,
\end{equation}
where
\begin{equation}\label{eq:c_2}
c_2(\theta)
:=
\exp\biggl\{\biggl(1 + \frac{ e^{\theta}e^\nu \Delta}{1 - e^{{\theta}-a}}\biggr)\frac{e^{\theta}\theta}{1-\rho}\biggr\}.
\end{equation}
\end{theorem}

\begin{theorem}[Exponential moment in stationary case]\label{thm:moment}
Assume \eqref{eq:rho_bound} and let
$\Nstat$ be the stationary loglinear Hawkes process. Then, for every bounded
Borel set $B\subset\R$ there exists $\theta_B \in(0,a)$ such that
$\E[\exp(\theta_B \Nstat(B))]<\infty$. In particular,
$\E[(\Nstat(B))^m]<\infty$ for every $m\in\N$.
\end{theorem}

\section{Proofs}\label{sec:proofs}

\begin{proof}[Proof of Lemma~\ref{lemma:poisson_embedding}]
Set $\mathcal{G}_t:=\sigma(\mathcal{F}_t^{\cN},\mathcal{F}_0^{N})$, $t>0$, and let
$\mathcal{P}(\mathcal{G})$ denote the $\{\mathcal{G}_t\}_{t>0}$-predictable $\sigma$-field on $\Omega\times (0,\infty)   $. Since $\cN$ has independent increments and is independent of $N^-$, the $\sigma$-field $\mathcal{G}_a$ is independent of the restriction of $\cN$ to $(a,\infty)\times\mathbb{R}_{\geq0}$ for every $a\geq0$. Hence, for $0\leq a<b$, $C\in\mathcal{B}(\mathbb{R}_{\geq0})$, and $A\in\mathcal{G}_a$,
$\E[1_A\,\cN((a,b]\times C)]=\P(A)\,(b-a)\,\int_C\d t$.
It is well known that for every $\mathcal{P}(\mathcal{G})\otimes\mathcal{B}( [0,\infty)   )$-measurable $W\colon\Omega\times (0,\infty)   \times\mathbb{R}_{\geq0}\to[0,\infty]$,
\begin{equation}\label{eq:campbell}
\E\bigg[\int_{\R_{> 0}\times\R_{\geq 0}} W(t,z)\,\cN(\d t\times\d z)\bigg]
=\E\biggl[\int_0^\infty\!\!\int_0^\infty W(t,z)\,\d z\,\d t\biggr].
\end{equation}
Consequently, as $\lambda$ is $\{\mathcal{G}_t\}_{t>0}$-predictable, we have $\{(\omega,t,z):z\leq\lambda(t,\omega)\}\in\mathcal{P}(\mathcal{G})\otimes\mathcal{B}(\mathbb{R}_{\geq0})$. Thus, for $0\leq a<b$ and $A\in\mathcal{G}_a$, the process $W(t,z):=1_A 1_{(a,b]}(t)1\{z\leq\lambda(t)\}$ is $\mathcal{P}(\mathcal{G})\otimes\mathcal{B}(\mathbb{R}_{\geq0})$-measurable, and applying \eqref{eq:campbell} together with \eqref{eq:pois_thinning_intensity} yields
\begin{equation*}
\E[1_A\,N((a,b])]
=\E\biggl[\int W(t,z)\,\cN(\d t\times\d z)\biggr]
=\E\biggl[1_A\int_a^b\lambda(t)\,\d t\biggr].
\end{equation*}
Since $A\in\mathcal{G}_a$ is arbitrary and both $N((a,b])$ and $\int_a^b\lambda(t)\,\d t$ are nonnegative, this is equivalent to \eqref{eq:intensity}, so  $\lambda$ is the $\{\mathcal{G}_t\}_{t>0}$-intensity of $N$. For the second statement, first note that $\mathcal{F}_a^N\subset\mathcal{G}_a$: by \eqref{eq:pois_thinning_intensity}, the restriction of $N$ to $(0,a]$ is a measurable function of $(\cN|_{(0,a]\times\mathbb{R}_{\geq0}},N^-)$. The tower property and \eqref{eq:intensity} give, for $0\leq a<b$,
\begin{equation*}
\E[N((a,b])\,|\,\mathcal{F}^N_a]
=\E\bigl[\E[N((a,b])\,|\,\mathcal{G}_a]\big|\,\mathcal{F}^N_a\bigr]
=\E\biggl[\int_a^b\lambda(t)\,\d t\,\bigg|\,\mathcal{F}^N_a\biggr].
\end{equation*}
This is \eqref{eq:intensity} for the internal history. As $\lambda$ is $\{\mathcal{F}_t^N\}_{t>0}$-predictable by assumption, it is the $\{\mathcal{F}_t^N\}_{t>0}$-intensity of $N$ and the proof is finished.
\end{proof}
\begin{proof}[Proof of Theorem~\ref{thm:representation}]
First, we observe that $\mathcal{I}$ defined by \eqref{eq:acceptance_count},  for fixed $\mu$ and $g$, can be extended to the measure on $\mathcal{B}([0, \infty))$.
To prove Theorem~\ref{thm:representation} we will use the following lemma.
\begin{lemma}
Let $\kappa \in \mathfrak{M}_h((-\infty,0])$ and
 $\mu\in\mathfrak{M}((0,\infty)\times [0,\infty)) $ be such that $\mu(\{t\} \times (0,\infty))\leq 1$  for every $t>0$.
For all $n\in \N$, if $\phi_{n-1}(\kappa, \mu) < \infty$, then $ \phi_{n-1}(\kappa,\mu) <\phi_n(\kappa,\mu ) $.
		Moreover  $\Psi(\kappa, \mu)$ is a simple point measure on  $\mathcal{B}([0, \infty))$ locally finite on $(0,\lim_{n\to\infty}\phi_n(\kappa,\mu))$ satisfying 	
\begin{equation}\label{eq:psithin}
		\Psi(\kappa, \mu )(A) = \kappa(A) + \mathcal{I}(\mu,\widehat{\lambda}(\cdot, \kappa, \mu);A),
		\end{equation}
		where $\widehat{\lambda}$ is given by
		\begin{equation*}
		\widehat{\lambda}(t, \kappa, \mu) = \sum_{ n
		\geq 1} \lambda^{(n)}(t, \kappa,\mu ) 1_{(\phi_{n-1}(\kappa, \mu), \phi_n(\kappa, \mu)]}(t).
		\end{equation*}
\end{lemma}
\begin{proof}
Since $\mu, \kappa$ are fixed we put $\phi_n:=\phi_n(\kappa,\mu )$ and $\phi_\infty :=\lim_{n\to\infty}\phi_n(\kappa,\mu).$
First we note that   for every $T < \infty $ we have
$
\sup_{t \in [0,T]} \lambda^{(n)}(t,\kappa, \mu) \leq C_T.
$
Indeed
\[
 \lambda^{(n)}(t,\kappa, \mu) \leq \exp\biggl\{\nu  +  \sup_{t \in [0,T]}\int_{(-\infty, 0]} h^+(t-s)\kappa(\d s) + ( n -1 )\sup_{ t \in [0,T] } h^+(t)   \biggr\} =: C_T < \infty
\]

Now we want to show that
if $\phi_{n-1}  < \infty$, then $\phi_{n-1}<\phi_n  $. To this end let $A_n= \{  s > \phi_{n-1}  : \mu(\{ s\} \times [0,\lambda^{(n)}(t,\kappa, \mu) ] ) \geq 1  \}.
$
Clearly $\phi_n  = \inf A_n$.
Each  $s \in A_n \cap  (\phi_{n-1}, T]$ for a fixed $T <\infty $ is a time coordinate of an  atom of $\mu$ in a bounded set $(\phi_{n-1} , T] \times [0,C_T]$.  Since $\mu$ is locally finite, the  number of points in  $(\phi_{n-1} , T] \times [0,C_T]$ is finite  and thus  $A_n \cap (\phi_{n-1}  , T]$ is a finite set.
Suppose that $\phi_{n}  < \infty$. Then one can choose $T <\infty $ such that $T > \phi_{n} $. Therefore  $A_n \cap (\phi_{n-1}, T]$ is finite and nonempty. Hence  $ \phi_n = \inf A_n  \in  A_n$
 and  $  \phi_{n-1} <\phi_{n}  	$.
 If $\phi_{n}  = \infty $,  then
$ \phi_{n-1}< \phi_{n} $ clearly holds.

This implies that $\Psi(\kappa,\mu)$ is a simple point measure which is locally finite on $(0, \phi_\infty)$. Indeed, for compact set $C \subseteq(0,\phi_\infty)$ we see that   $\sup C < \phi_\infty$
and thus by monotonicity of $\phi_n$ only for finitely many points it holds $\phi_n \leq \sup C$, so
\[
\Psi(\kappa, \mu)(C) \leq  \kappa(C) + \sum_{ n \geq 1} 1{\{\phi_n \leq \sup C \}} < \infty.
\]
Now to justify \eqref{eq:psithin} we need to prove that for every $B \in \mathcal{B}([0, \infty))$ we have
\[
\mathcal{I}(\mu,\widehat{\lambda}(\cdot, \kappa, \mu);B) = \sum_{ n \geq 1 } 1{\{ \phi_n \in  B \}}.
\]
 Towards this end we note that
$
(0,\infty) = \bigcup_{n\in\N} (\phi_{n-1}, \phi_n] \cup [\phi_\infty, \infty)
$
and thus for every $B \in \mathcal{B}(\R)$ we have
\[
\mathcal{I}(\mu,\widehat{\lambda}(\cdot, \kappa, \mu);B) = \sum_{n \geq 1 }  \mathcal{I}(\mu,\widehat{\lambda}(\cdot, \kappa, \mu);B \cap  (\phi_{n-1}, \phi_n]) + \mathcal{I}(\mu,\widehat{\lambda}(\cdot, \kappa, \mu);B \cap [\phi_\infty, \infty)).
\]
First we note that if $\phi_\infty = \infty$, then
$B \cap [\phi_\infty, \infty)  = \emptyset$ and thus $\mathcal{I}(\mu,\widehat{\lambda}(\cdot, \kappa, \mu);B \cap [\phi_\infty, \infty))=0$. Whereas if $\phi_\infty < \infty$, then $\phi_n < \phi_\infty$. Hence $\phi_n \notin B \cap [\phi_\infty, \infty) $ and thus  $\widehat{\lambda}(t,\kappa,\mu) = 0$ for $t\in B \cap [\phi_\infty, \infty)$. Consequently
\[
\mathcal{I}(\mu,\widehat{\lambda}(\cdot, \kappa, \mu);B \cap [\phi_\infty, \infty)) =  \int 1\{s\in B \cap [\phi_\infty, \infty) \}\,1\{z \leq 0 \}\,\mu(\d s\times \d z) = 0.
\]
Now it remains to show that  for every $n\geq 1 $ we have
\[
 \mathcal{I}(\mu,\widehat{\lambda}(\cdot, \kappa, \mu);B \cap  (\phi_{n-1}, \phi_n]) =   1{\{\phi_n \in B \}}.
 \]
 If $\phi_{n-1} = \infty$, then the above equality holds trivially.
 Now assume that $\phi_{n-1} < \infty$.
 Since $\widehat{\lambda} = \lambda^{(n)}$ on $(\phi_{n-1}, \phi_n]$ we have  \[\mathcal{I}(\mu,\widehat{\lambda}(\cdot, \kappa, \mu);B \cap   (\phi_{n-1}, \phi_n]) =\mathcal{I}(\mu,\lambda^{(n)}(\cdot, \kappa, \mu);B \cap  (\phi_{n-1}, \phi_n] ).\]
 If $\phi_n = \infty$ then $A_n = \emptyset$ and thus $\mathcal{I}(\mu,\lambda^{(n)}(\cdot, \kappa, \mu);B \cap  (\phi_{n-1}, \phi_n]) =0 =1{\{ \phi_n \in B \}}$.
 If $\phi_n < \infty$, then $\mathcal{I}(\mu,\lambda^{(n)}(\cdot, \kappa, \mu),(\phi_{n-1}, t]) = 0 $ for every $t \in (\phi_{n-1}, \phi_n )$ (since  $\phi_n = \inf A_n$), moreover $\mathcal{I}(\mu,\lambda^{(n)}(\cdot, \kappa, \mu),(\phi_{n-1}, \phi_n]) = \mu( \{ \phi_n \} , [0, \lambda^n ( \phi_n , \kappa, \mu ) ]) = 1$. Consequently   $\mathcal{I}(\mu,\lambda^{(n)}(\cdot, \kappa, \mu);B \cap  (\phi_{n-1}, \phi_n])  =1{\{ \phi_n \in B \}}$. The proof of the lemma is finished.
\end{proof}
We proceed with the proof of the theorem.

\noindent

\emph{Step 1 ($N$ is a simple point process satisfying \eqref{eq:jump_times}).}
In order to prove that $N$ is a random measure, it suffices to show measurability of the mapping $\Psi$. To this end, we first show by induction that $\phi_n$ is measurable for all $n\in\N$; measurability of $\Psi$ then follows, since $\Psi(\kappa,\mu)(A)$ is a measurable function of $(\kappa,\mu)$ for every $A\in\mathcal{B}(\R)$. As $(t,\kappa)\mapsto\lambda^{(1)}(t,\kappa,\mu)$ is measurable, $(\kappa,\mu)\mapsto\mathcal{I}(\mu,\lambda^{(1)}(\cdot,\kappa,\mu);(0,q])$ is measurable for every $q>0$ by \cite[Lemma 1.15]{kallenberg2017}. Since $r\mapsto \mathcal{I}(\mu,\lambda^{(1)}(\cdot,\kappa,\mu);(0,r])$ is right-continuous,
	\[
	\{(\kappa,\mu):\ \phi_1(\kappa,\mu)<a\} = \bigcup_{q\in\mathbb{Q},\ q<a}\{(\kappa,\mu):\ \mathcal{I}(\mu,\lambda^{(1)}(\cdot,\kappa,\mu);(0,q])>0\},\quad a>0,
	\]
	so $\phi_1$ is measurable. If $\phi_1,\dots,\phi_n$ are measurable, then $(t,\kappa,\mu)\mapsto\lambda^{(n+1)}(t,\kappa,\mu)$ is measurable, and the same argument
	shows that $\phi_{n+1}$ is measurable.	
	Next, we show that $N$ is a simple point process satisfying \eqref{eq:jump_times}. Fix $\kappa\in\mathfrak{M}_h((-\infty,0])$, $\mu\in\mathfrak{M}((0,\infty)\times[0,\infty))$ and $n\in\N$. By Assumption~\ref{ass:h} for every $T>0$, every interval $I\subset(0,T]$  and
	$t \in (\phi_{n-1}(\kappa,\mu), \phi_n(\kappa,\mu)) $ we see that
	$$
	\sup_{t\in I}\lambda^{(n)}(t,\kappa,\mu)\leq \exp\Bigl\{\nu+\sup_{t\in I}\int_{(-\infty,0]}h^+(t-s)\kappa(\d s)+(n-1)\sup_{t\in(0,T]}h(t)\Bigr\}<\infty,
	$$
	since $t-\phi_k(\kappa,\mu)\in(0,T]$ for $t\in I$. As $\mu$ is locally finite, only finitely many points of $\mu$ lie in $I\times[0,\sup_{t\in I}\lambda^{(n)}(t,\kappa,\mu)]$. Consequently, $\phi_1(\kappa,\mu)>0$, and $\phi_n(\kappa,\mu)<\phi_{n+1}(\kappa,\mu)$ on $\{\phi_n(\kappa,\mu)<\infty\}$. Thus, $\Psi(\kappa,\mu)$ is a simple point measure, locally finite on $(-\infty,\phi_\infty(\kappa,\mu))$ with $\phi_\infty(\kappa,\mu)=\lim_{n\to\infty}\phi_n(\kappa,\mu)>0$, and \eqref{eq:jump_times} holds for $N=\Psi(N^-,\cN)$.
	
	\emph{Step 2 ($N$ is a loglinear Hawkes process with initial condition $N^-$).}
	By Assumption~\ref{ass:h} the integral in \eqref{eq:constructed_intensity} is well-defined in $[-\infty,\infty)$.
	Now we note that $N = N^-$ on $(-\infty, 0]$ by definition of $\Psi$. Thus it remains to show that the $\{\mathcal{F}^N_t\}_{t>0}$-intensity of $N$ is given by \eqref{eq:constructed_intensity}. We use Lemma \ref{lemma:poisson_embedding}  to this end.
	Clearly, for  $\mathcal{N} $ we have $\mathcal{N}(\{t\} \times [0,\infty) ) =0$ for every $t \geq 0 $. Hence  by \eqref{eq:psithin} we have, for  $A \in \mathcal{B}((0,\infty) )$,
\[
N(A) =  \mathcal{I}(\mathcal{N},\widehat{\lambda}(\cdot,N^-,\mathcal{N});A)
\]
as $N^-(A) = 0$.
Thus, it suffices to show that  $ \widehat{\lambda}(\cdot ,N^-, \mathcal{N}) = \lambda(\cdot) $ and that it is predictable with respect to $\{\sigma(\mathcal{F}^{\mathcal{N}}_t , \mathcal{F}^{N}_0)\}_{t > 0 }$.
By definition of $\Psi$,  the points of $N$ in $(0, \infty)$, say   $\{ T_n\}_{n \geq 1}$, satisfy
$ T_n=	\phi_{n}(N^-, \mathcal{N})$, $n \geq 1$, and hence
\[
\widehat{\lambda}(t, N^-, \mathcal{N}) = \lambda^{(1)}(t, N^-, \mathcal{N}) 1_{(0, T_1]}(t) + \sum_{n \geq 2} \lambda^{(n)}(t, N^-, \mathcal{N}) 1_{(T_{n-1} , T_n]}(t).
\]
For all $t \in (0, T_1]$ we have $\lambda^{(1)}(t,N^-,\cN) = \lambda(t)$. For $n\geq 2$, by definitions of $\lambda^{(n)}$ and $T_k$, we obtain for $t\in(T_{n-1},T_n]$ that
\begin{align*}
	\lambda^{(n)}(t,N^-,\cN)&=\exp\biggl\{\nu+\int_{(-\infty,0]}h(t-s)N^-(\d s)+\sum_{k=1}^{n-1}h(t-T_k)1\{T_k<\infty\}\biggr\}\\
	&=\exp\biggl\{\nu+\int_{(-\infty,t)}h(t-s)N(\d s)\biggr\}=\lambda(t).
\end{align*}
which proves that $\widehat{\lambda}(\cdot ,N^-, \mathcal{N})= \lambda(\cdot) $.
Therefore $\lambda$  has left-continuous trajectories since  $\lambda^{(n)}(\cdot, N^-, \mathcal{N}) 1_{(T_{n-1} , T_n]}(\cdot)$, $n\in \N$, are such.
Hence $\lambda$ is  $\{\mathcal{F}_t^{N}\}_{t> 0}$-predictable since  it is $\{\mathcal{F}_t^{N}\}_{t>0}$-adapted.
Consequently, by Lemma~\ref{lemma:poisson_embedding}, $\lambda$ is the $\{\mathcal{F}_t^N\}_{t>0}$-intensity of $N$, i.e., $N$ is a loglinear Hawkes process with initial condition $N^-$.
\end{proof}

\begin{proof}[Proof of Theorem~\ref{thm:nonexplosion}]
\ref{item:sufficient} We show that $\P(N((0,t]) <\infty) = 1$ for all $t\geq 0$.
We prove by induction that
$N((0,k\delta])$ is a.s. finite for all  $k\geq 1$. For $k=1$, we have, for $t\in(0,\delta]$,
\begin{align*}
\log  \lambda(t)
& = \nu +\int \limits_{(-\infty, 0]} h (t-r)N^-(\d r)+\int \limits_{(0,t)}  \underbrace{h (t-r)}_{\leq 0 \text{ by assumption}}N(\d r)\\
& \leq \nu +\int \limits_{(-\infty, 0]} h (t-r)N^-(\d r) \leq \nu +\sup_{s\in(0,\delta]}\int \limits_{(-\infty, 0]} h^+ (s-r)N^-(\d r) =:M_0 <\infty,
\end{align*}
where finiteness of $M_0$ follows from Assumption~\ref{ass:h}.
Using the thinning construction of $N$ from Theorem~\ref{thm:representation}
we have
\begin{align*}
N((0, \delta])= \int_{(0,\delta]\times \mathbb{R}_{\geq0}} 1\{z\leq {\lambda(s)} \}
\mathcal{N}(\d s\times \d z)
&\leq \int_{(0,\delta]\times \mathbb{R}_{\geq0}} 1\{z\leq e^{M_0} \}
\mathcal{N}(\d s\times \d z) \\
&= \mathcal{N}((0,\delta]\times [0, e^{M_0}])<\infty.
\end{align*}
Assume that $N((0,l\delta])<\infty$ for $l = 1,\dots, k$. We deduce that $N((k\delta, (k+1)\delta])$ is finite. Indeed, for $s\in (k\delta, (k+1)\delta]$,
we have
\begin{align*}
&\log \lambda(s)=
\nu  + \int_{(-\infty,0]} h(s - r) N^-(\d r)+ \int_{(0,s)} h(s - r) N(\d r)\\
&=\nu  +{\int_{(-\infty,0]} h(s - r) N^-(\d r)} +
{
\int_{(0,k \delta]} h(s - r) N(\d r)
}  +\underbrace{\int_{(k\delta,s)} h(s - r) N(\d r)}_{\leq 0}\\
& \leq \nu + \sup_{t \in (k\delta,(k+1)\delta]} \int_{(-\infty,0]} h^+(t - r) N^-(\d r) + \sup_{t \in (0,(k+1)\delta)} h(t)  N((0,k \delta]) =: M_k < \infty,
\end{align*}
by Assumption~\ref{ass:h} and the induction hypothesis. Furthermore, $M_k$ does not depend on $s$.
  Therefore, we have
\begin{align*}
N((k\delta, (k+1) \delta])
&= \int_{(k\delta,(k+1)\delta]\times \mathbb{R}_{\geq0}} 1\{z\leq\lambda(s) \}
\mathcal{N}(\d s\times \d z)\\
 & \leq \mathcal{N}((k\delta,(k+1)\delta] \times (0, e^{M_k}]) < \infty.
\end{align*}
 Since, for $t\geq 0$,
\begin{align*}
&\P(N((0,t)) <\infty)\\
 &\geq \P\Bigl(\sum_{k = 1}^{\lceil t/\delta\rceil} N(((k-1) \delta,k \delta]) <\infty\Bigr)
 =  \P\Bigl(\bigcap_{k = 1}^{\lceil t/\delta\rceil} \{N(((k-1) \delta,k \delta])  <\infty\}\Bigr)
 = 1,
\end{align*}
being the probability of a finite intersection of events of probability one, we obtain
\begin{equation*}
\P(T_{\infty} < \infty )= \lim_{n\to\infty} \P(T_{\infty} < n ) \leq \lim_{n\to\infty}  \P(N((0,n)) = \infty) = 0
\end{equation*}
and thus $T_{\infty} = \infty$ a.s.
This finishes the proof of part \ref{item:sufficient}.

\ref{item:necessary} We first prove explosion of $N$ in the case $N^- = \kappa$, where $\kappa$ is a deterministic locally finite counting measure on $(-\infty, 0] $  such that the integrals
$I(t,\kappa):=\int_{(-\infty,0]}h(t-r)\,\kappa(\d r)$, $t\in(0,\delta]$, are well-defined and 
\begin{equation}\label{eq:J_def}
J:=J(\kappa):= \inf_{t\in(0,\delta]}I(t,\kappa)>-\infty.
\end{equation}
Let $N^\kappa= \Psi(\kappa, \mathcal{N})$ be the loglinear Hawkes process constructed in Theorem~\ref{thm:representation} with $N^-=\kappa$, and let $T_0:=0$ and $T_n := \phi_n(\kappa,\mathcal{N}),n\in\N,$ be the corresponding jump times of $N^\kappa$ in $(0,\infty]$ as defined in \eqref{eq:phi_def}. Set $S_k:=T_k-T_{k-1}$, $k\in\N$, fix $\varepsilon>0$ with $\varepsilon\sum_{k\geq1}k^{-2}<\delta$, and let $B_0:=\Omega$ and $B_n:=\bigcap_{k=1}^n\{S_k\leq\varepsilon/k^2\}$, $n\in\N$. Note that $ B_n\subseteq \{\sum_{k=1}^nS_k \leq \sum_{k=1}^n\varepsilon/k^2\}\subseteq\{T_n \leq \delta\}=\{ N^\kappa((0,\delta])\geq n \} $. As the latter events are non-increasing in $n$, we obtain
\begin{equation}\label{eq:phi_bound}
\P(N^\kappa((0,\delta])=\infty) \geq \lim_{n\to\infty}\P(B_n)=\lim_{n\to\infty}\prod_{k=1}^n \P\Bigl(S_k\leq \frac{\varepsilon}{k^2}\,\Big|\,B_{k-1}\Bigr),
\end{equation}
provided that $\P(B_n)>0$ for all $n\in\N_0$. We show that the right-hand side of \eqref{eq:phi_bound} is strictly positive by proving the following three results:
\begin{enumerate}[label = (\alph*)]
\item $\P(B_n)>0$ for all $n\in\N_0$,
\item $\P(S_k\leq \varepsilon/k^2\,|\,B_{k-1})\geq 1-\exp\{-\frac{\varepsilon}{k^2}c_k\}$ for all $k\in\N$, where $c_k:=e^{\nu+J+(k-1)h_\delta}$, and
\item $\prod_{k\geq1}(1-\exp\{-\frac{\varepsilon}{k^2}c_k\})>0$.
\end{enumerate}
{(a)} For $n=0$, $B_0 = \Omega$ and (a) holds. For $n\in\N$, since $B_n\supseteq\{T_n<\varepsilon/n^2\}=\{N^\kappa((0,\varepsilon/n^2))\geq n\}$, we bound $N^\kappa((0,\varepsilon/n^2))$ from below. On $\{T_\infty\leq \varepsilon/n^2\}$ we have $N^\kappa((0,\varepsilon/n^2))=\infty$. On $\{T_\infty> \varepsilon/n^2\}$ we obtain, for $s\in(0,\varepsilon/n^2)\subset(0,\delta)$,
$$
\lambda(s)=\exp\bigg\{\nu+I(s,\kappa)+\int_{(0,s)}\underbrace{h(s-r)}_{\geq h_\delta>0}N^\kappa(\d r)\bigg\}\geq e^{\nu+J},
$$
using that $s-r\in(0,\delta)$ and $I(s,\kappa)\geq J$.
Thus, in both cases
 $$
N^\kappa((0,\varepsilon/n^2)) = \Psi(\kappa, \mathcal{N})((0,\varepsilon/n^2))\geq \cN((0,\varepsilon/n^2)\times[0,e^{\nu+J}]),
$$
by \eqref{eq:psithin},
and the right-hand side is Poisson distributed with mean $\frac{\varepsilon}{n^2}e^{\nu+J}>0$, so that $\P(B_n)\geq\P(\cN((0,\varepsilon/n^2)\times[0,e^{\nu+J}])\geq n)>0.$

\noindent {(b)} Let $k\in\N$. For an
$\{\mathcal{F}^\cN_t\}_{t>0}$-stopping time $T$, we set
$\mathcal{F}^\cN_{T}
:=\{A\in\mathcal{F}:\ A\cap\{T\leq t\}\in\mathcal{F}^\cN_t
\text{ for all } t> 0\}.
$ As $T_j = \phi_j(\kappa, \mathcal{N})$, \eqref{eq:phi_def} gives that the event $\{T_j\leq t\}$ is
determined by $\cN|_{(0,t]\times[0, \infty)  }$. Hence
$T_1\leq T_2\leq\dotsb$ are $\{\mathcal{F}^\cN_t\}_{t>0}$-stopping times
and, since $T_j\leq T_{k-1}$ for $j\leq k-1$, the vector
$(T_1,\dots,T_{k-1})$, and thus $1_{B_{k-1}}$, is $\mathcal{F}^\cN_{T_{k-1}}$-measurable. Moreover, $T_{k-1}<\infty$ on $B_{k-1}$.
By the strong Markov property of the Poisson random measure $\cN$ with intensity 1
(see \cite[Theorem 7.4]{last17}), on the event $\{T_{k-1}<\infty\}$ the shifted random measure
$\theta_{T_{k-1}}\cN$, defined by
\begin{equation*}
\theta_{T_{k-1}}\cN(A\times C):=\cN((T_{k-1}+A)\times C),
\qquad A\in\mathcal{B}((0, \infty)  ),\ C\in\mathcal{B}([0, \infty)  ),
\end{equation*}
is again a Poisson random measure on $(0, \infty)  \times[0, \infty)$ with intensity 1 and is
independent of $\mathcal{F}^\cN_{T_{k-1}}$.
By \eqref{eq:phi_def}, on $B_{k-1}$,
\begin{equation}\label{eq:S_k_event}
\biggl\{S_k>\frac{\varepsilon}{k^2}\biggr\}=\biggl\{\mathcal{I}\Bigl(\cN,\lambda^{(k)}(\cdot,\kappa,\cN);\Bigl(T_{k-1},T_{k-1}+\frac{\varepsilon}{k^2}\Bigr]\Bigr)=0\biggr\},
\end{equation}
where $\lambda^{(k)}(t,\kappa,\cN)=\exp\bigl\{\nu+I(t,\kappa)+\sum_{j=1}^{k-1}h(t-T_j)\bigr\}$.
For $t\in\bigl(T_{k-1},T_{k-1}+\frac{\varepsilon}{k^2}\bigr]$ and $j=1,\dots,k-1$, we find that on $B_{k-1}$
$$
0<t-T_j\leq T_{k-1}+\frac{\varepsilon}{k^2}-T_j\leq\sum_{i=1}^{k}\frac{\varepsilon}{i^2}<\delta,
$$
and thus $h(t-T_j)\geq h_\delta$. Moreover, $t\in(0,\delta)$, so that $I(t,\kappa)\geq J$. Hence $\lambda^{(k)}(t,\kappa,\cN)\geq c_k$ on $B_{k-1}$, and \eqref{eq:S_k_event} together with the definition of $\theta_{T_{k-1}}$ yields, on $B_{k-1}$,
\begin{align*}
\biggl\{S_k>\frac{\varepsilon}{k^2}\biggr\}&\subseteq\biggl\{\cN\Bigl(\Bigl(T_{k-1},T_{k-1}+\frac{\varepsilon}{k^2}\Bigr]\times[0,c_k]\Bigr)=0\biggr\}\\
&=\biggl\{\theta_{T_{k-1}}\cN\Bigl(\Bigl(0,\frac{\varepsilon}{k^2}\Bigr]\times[0,c_k]\Bigr)=0\biggr\}.
\end{align*}
Since $1_{B_{k-1}}$ is $\mathcal{F}^\cN_{T_{k-1}}$-measurable and $\theta_{T_{k-1}}\cN$ is a Poisson random measure with intensity 1 independent of $\mathcal{F}^\cN_{T_{k-1}}$, we obtain
$$
\P\biggl(B_{k-1}\cap\biggl\{S_k>\frac{\varepsilon}{k^2}\biggr\}\biggr)\leq \P(B_{k-1})\exp\biggl\{-\frac{\varepsilon}{k^2}c_k\biggr\}.
$$
Dividing by $\P(B_{k-1})$, which is positive by (a), and rearranging gives (b).

\noindent{(c)} The infinite product $\prod_{k\geq1}(1-\exp\{-\varepsilon c_k/k^2\})$ is strictly positive if and only if $\sum_{k\geq1}\exp\{-\varepsilon c_k/k^2\}<\infty$ (see, e.g., \cite[Chapter 7]{knopp51}), which holds as $h_\delta>0$ and $J = J(\kappa) >-\infty$. By \eqref{eq:phi_bound} and (b), we conclude that $\P(N^\kappa((0,\delta])=\infty)>0$.

Finally, we use the representation $N=\Psi(N^-, \cN)$ from Theorem~\ref{thm:representation} with $N^-$ and $\cN$ independent. Writing $g(\kappa,\cN):=1\{N^\kappa((0,\delta])=\infty\}$, we therefore conclude for our original loglinear Hawkes process $N$ with initial condition $N^-$ that
$$
\P(N((0,\delta])=\infty)=\E[g(N^-,\cN)]= \E\bigl[\E[g(N^-,\cN)| \sigma(N^-) ]\bigr] = \E\big[\E[g(\kappa,\cN)]|_{\kappa = N^-}\big].
$$
By (c) we obtain that $\E[g(\kappa,\cN)]|_{\kappa = N^-} = \P(N^\kappa((0,\delta])=\infty)|_{\kappa = N^-}>0$ because  \eqref{eq:J_def} is satisfied by assumption \eqref{ass:initial_cond2}. Thus, $\P(N((0,\delta])=\infty)>0$, and $\{N((0,\delta])=\infty\}\subset\{T_\infty\leq\delta\}$, hence $T_\infty\leq\delta$ with positive probability and the process $N$ is explosive. This completes the proof.
\end{proof}

\begin{proof}[Proof of Lemma~\ref{lemma:monotonicity}] Observe that $M_i = \sum_{n=1}^\infty \mathbf{1}\{T^{(i)}_n<\infty\}\,\delta_{T^{(i)}_n},$ for
$T^{(i)}_0 = 0$ and, recursively for $n\in\N$,
\begin{equation}\label{eq:thinning_construction}
T^{(i)}_n = \inf\Big\{t > T^{(i)}_{n-1}:\ \mathcal{I}\Bigl(\mathcal{M},\lambda_i;\big(T^{(i)}_{n-1},t\big]\Bigr) > 0\Big\}.
\end{equation}
Since $h_1 \leq -a < 0$, $\nu_1\leq\nu_2<\infty$ and
$\int_{(0,t)}h_2(t-s)\,M_2(\d s) = -a\,M_2((0,t)) \leq 0$, we have
$\lambda_i(t) \leq e^{\nu_2}$, $t\geq 0$, $i = 1,2$.
Hence, only the points of $\mathcal{M}$ in the random set $\R_{\geq 0}\times[0,e^{\nu_2}]$
contribute to the construction of $M_1$ and $M_2$. Conditionally on $\mathcal{G}_0$, $\mathcal{M}(\cdot \times [0,e^{\nu_2}])$ is a Poisson process in time
 with intensity $e^{\nu_2}$; in particular, the points of $\mathcal{M}$ with the second coordinate in $[0,e^{\nu_2}]$ can be enumerated in increasing time
order as $(\tau_j, z_j)_{j\in\N}$. Consequently, the recursion
\eqref{eq:thinning_construction} is well-defined, and
$M_i([0,t]) \leq \mathcal{M}([0,t]\times[0,e^{\nu_2}]) < \infty$ for $ i = 1,2,$ so
that neither process explodes.

Set $D(t) := M_2([0,t]) - M_1([0,t])$, $t\geq 0$. The process $D$ is piecewise constant,
$D(0) = 0$, and changes only at points  $(\tau_j, z_j)$ of $\mathcal{M}$ in the band $\R_{\geq 0}\times [0,e^{\nu_2}]$. We
claim that $D(t)\geq 0$ for all $t\geq 0$. The key observation is
\begin{equation}\label{eq:equal_count_comparison}
M_1([0,t)) = M_2([0,t))
\quad\Longrightarrow\quad
\lambda_1(t) \leq \lambda_2(t).
\end{equation}
Indeed, if $M_1([0,t)) = M_2([0,t)) = k$ for some $k\in\N_0$, then
\[
\lambda_1(t)
= \exp\Bigl\{\nu_1(t) + \sum_{n=1}^{k} h_1(t - T^{(1)}_n)\Bigr\}
\leq e^{\nu_2 - k a}
= \exp\Bigl\{\nu_2 + \sum_{n=1}^{k} h_2(t - T^{(2)}_n)\Bigr\}
= \lambda_2(t),
\]
using $\nu_1\leq\nu_2$ and $h_1 \leq -a$.
Obviously, for $t \in (0, \tau_1)$ we have $D(t)= 0$.
Now we proceed by induction over $j\in\N$ for the points
$(\tau_j, z_j)$ in time order, with induction hypothesis $D(t)\geq 0$ for $t < \tau_j$. At
the atom $(\tau_j, z_j)$ there are four possibilities:
\begin{enumerate}[label =(\roman*)]
\item $z_j \leq \min\{\lambda_1(\tau_j), \lambda_2(\tau_j)\}$: both processes jump and $D$
is unchanged;
\item $z_j > \max\{\lambda_1(\tau_j), \lambda_2(\tau_j)\}$: neither process jumps and $D$
is unchanged;
\item $\lambda_1(\tau_j) < z_j \leq \lambda_2(\tau_j)$: only $M_2$ jumps and $D$ increases
by $1$;
\item $\lambda_2(\tau_j) < z_j \leq \lambda_1(\tau_j)$: only $M_1$ jumps and $D$ decreases
by $1$.
\end{enumerate}
Thus, for the induction step, we only have to check case (iv). Note that this event requires $\lambda_2(\tau_j) < \lambda_1(\tau_j)$, which by
\eqref{eq:equal_count_comparison} forces
$M_1([0,\tau_j)) \neq M_2([0,\tau_j))$. Consequently, taking into account the induction
hypothesis $D(\tau_j-)\geq 0$, we have $D(\tau_j-)\geq 1$. Thus, even in case (iv), when $D$ decreases by $1$ at time $\tau_j$, we obtain that $D(\tau_j)\geq 0$. Hence, in all four cases we have $D(\tau_j)\geq 0$. This completes the induction step. Finally, since $D$ is constant on the intervals $[\tau_j, \tau_{j+1})$, we conclude that $ M_2([0,t]) - M_1([0,t]) = D(t)\geq 0$ for all $t\geq 0$.
\end{proof}

\begin{proof}[Proof of Theorem~\ref{thm:bin_count_domination}]
\noindent\emph{Step 1 (collecting positive excitement from the past).}
Fix $n\in\N$. The interval $I^{(+)}_n$ contains all points of $N$ up to time $(n-1)\Delta$
that can positively excite the intensity of $N$ inside the bin $((n-1)\Delta,n\Delta]$.
Indeed, for $t\in((n-1)\Delta,n\Delta]$ and $s \notin I^{(+)}_n$, we have either
$s > (n - p^{(-)})\Delta$, in which case
$t - s < n\Delta - (n - p^{(-)})\Delta = p^{(-)}\Delta$,
or $s \leq (n - p - 1)\Delta$, in which case
$t - s > (n-1)\Delta - (n - p - 1)\Delta = p\Delta$.
In both cases $t-s\notin [p^{(-)}\Delta, p\Delta]$, and thus, by
Definition~\ref{def:memory_envelope},
\begin{equation}\label{eq:out_of_I}
h(t-s) \leq 0, \qquad t\in((n-1)\Delta,n\Delta],\ s\notin I^{(+)}_n.
\end{equation}
Consequently, for all $t\in ((n-1)\Delta,n\Delta]$,
\begin{align}\label{eq:past_bound}
\int_{(-\infty, (n-1)\Delta]} h(t-s)\, N(\d s)
&= \int_{I_{n}^{(+)}} h(t-s)\, N(\d s) + \int_{(-\infty,  (n-1)\Delta]\setminus I_{n}^{(+)}} h(t-s)\, N(\d s) \nonumber \\
&\leq \int_{I_{n}^{(+)}} h(t-s)\, N(\d s)
\leq b\, N(I_{n}^{(+)}).
\end{align}
\noindent\emph{Step 2 (localization on the bin).} Let $\mathcal{N}$ denote the driving Poisson random measure on $\R_{> 0}\times \R_{\geq 0}$ with intensity $1$ used in the
thinning construction of $N$ in Theorem~\ref{thm:representation}, and introduce Poisson random measures $\widetilde{\mathcal{M}}_n$ on $(n\Delta,\infty)\times\R_{\geq 0}$ with intensity 1, independent of each other and of $\mathcal{N}$.
Define the hybrid driving measures
\begin{equation*}
\mathcal{M}_n :=
\mathcal{N}\big|_{((n-1)\Delta,\, n\Delta]\times\R_{\geq 0}}
+ \widetilde{\mathcal{M}}_n,
\qquad n\in\N.
\end{equation*}
Since the processes $\mathcal{M}_n$  involve $\mathcal{N}$ only through the pairwise disjoint strips
$((n-1)\Delta, n\Delta]\times\R_{\geq 0}$, the family $\{\mathcal{M}_n\}_{n\in\N}$ consists of independent random measures.
Using $\mathcal{M}_n$ as the driving process, we construct a point process $M_{n,1}$ on $((n-1)\Delta, \infty)$ by setting $T^{(n,1)}_0 := (n-1)\Delta$ and defining recursively the jump times
\begin{equation}\label{eq:infimum1}
T^{(n,1)}_k = \inf\Big\{t> T^{(n,1)}_{k-1}:\
\mathcal{I}\Bigl(\mathcal{M}_{n},\lambda_{n,1};\big(T^{(n,1)}_{k-1},t\big]\Bigr) > 0 \Big\},\quad k\in\N,
\end{equation}
with
\begin{align*}
&\lambda_{n,1}(t)\nonumber\\
&:= \exp\biggl\{\nu+ 1\{t \leq n\Delta\} \int_{(-\infty, (n-1)\Delta]} h(t - s)\, N(\d s) + \int_{((n-1)\Delta,\, t)} h_1(t - s)\, M_{n,1}(\d s) \biggr\},
\end{align*}
for $t\geq (n-1)\Delta$, and with memory function
\begin{equation*}
h_1(t) :=
\begin{cases}
h(t), & t\in (0, \Delta],\\
-a, & t\in (\Delta, \infty).
\end{cases}
\end{equation*}
For $t\in((n-1)\Delta ,n\Delta]$ and $s \in ((n-1)\Delta, t)$ we have $t-s\in(0,\Delta)$, so that for these choices of $t$ and $s$ the
memory values $h(t-s)$ and $h_1(t-s)$ coincide. Since $\mathcal{M}_n$ coincides with $\mathcal{N}$ on
$((n-1)\Delta, n\Delta]\times\R_{\geq 0}$ and $M_{n,1}$ is constructed by
thinning from $\mathcal{M}_n$, $M_{n,1}$ coincides with $N$ on
$((n-1)\Delta ,n\Delta]$. To see this, denote by $T^{(n)}_1 < T^{(n)}_2 < \dots$ the jump times of $N$ after $T^{(n)}_0 := (n-1)\Delta$, i.e.,
\begin{equation}
T^{(n)}_k = \inf\Big\{t>T^{(n)}_{k-1}:\
\mathcal{I}\Bigl(\mathcal{N},\lambda;\big(T^{(n)}_{k-1},t\big]\Bigr) > 0 \Big\},
\label{eq:infimum}
\end{equation}
where $\lambda$ denotes the stochastic intensity of $N$.
We show by induction that $\min\{T^{(n)}_k ,n\Delta\}= \min\{T^{(n,1)}_k,n\Delta\}$ for all
$k\in\N_0$. The case $k = 0$ holds by definition. Assume
\begin{equation}\label{eq:i.h.1}
\min\{T^{(n)}_l,n\Delta\} = \min\{T^{(n,1)}_l,n\Delta\}\quad\text{
for all $l < k$.}
\end{equation} If the minima at $l = k-1$ equal $n\Delta$, then, as the jump time sequences are both nondecreasing, the minima at step $k$ will also both equal $n\Delta$. If $T^{(n)}_{k-1}=T^{(n,1)}_{k-1} < n\Delta$, then $N$ and $M_{n,1}$ coincide on $((n-1)\Delta, T^{(n)}_{k-1}]$, and for
every $t\in(T^{(n)}_{k-1}, n\Delta)$ such that neither process has a point in
$(T^{(n)}_{k-1}, t)$, we obtain
\begin{align*}
&\lambda(t)
= \exp\biggl\{\nu + \int_{(-\infty, (n-1)\Delta]} h(t - s)\, N(\d s)
+ \int_{((n-1)\Delta,\, t)} h(t - s)\, N(\d s)\biggr\} \nonumber\\
&= \exp\bigg\{\nu + 1\{t \leq n\Delta\} \int_{(-\infty, (n-1)\Delta]} h(t - s)\, N(\d s) + \int_{((n-1)\Delta,\, t)} h_1(t - s)\, M_{n,1}(\d s)\bigg\}\\
&= \lambda_{n,1}(t),
\end{align*}
using $t \leq n\Delta$, $h_1 = h$ on $(0,\Delta)$, and the induction hypothesis \eqref{eq:i.h.1}. Hence, since $\mathcal{M}_n$  coincides with $\mathcal{N}$ on
$((n-1)\Delta, n\Delta]\times\R_{\geq 0}$, the infima in \eqref{eq:infimum1} and in \eqref{eq:infimum} either
coincide or are both $\geq n\Delta$, which closes the induction step, and we have shown that $\min\{T^{(n)}_k, n\Delta\}=\min\{ T^{(n,1)}_k,n\Delta\}$. Thus, the point processes $N$ and $M_{n,1}$ have the same jump times in $((n-1)\Delta, n\Delta]$.
Hence, $N$ and $M_{n,1}$ coincide on the bin $((n-1)\Delta, n\Delta]$.

\noindent\emph{Step 3 (bin count domination).} By \eqref{eq:past_bound}, we have, for all $t\geq  (n-1)\Delta$,
\begin{equation*}
 \nu_{n,1}(t)
:= \nu + 1\{t \leq n\Delta\} \int_{(-\infty, (n-1)\Delta]} h(t - s)\, N(\d s)
\leq \nu + N(I^{(+)}_n)\, b =: \nu_{n,2}.
\end{equation*}
We aim to apply Lemma~\ref{lemma:monotonicity}, where, in the notation of the lemma, $\nu_{n,1}$ and $\nu_{n,2}$ play the role of $\nu_{1}$ and $\nu_2$, respectively, and time $(n-1)\Delta$ of the present
proof plays the role of time $0$. Also note that
$h_1(t) \leq -a =: h_2(t)$, $t > 0$.
In the notation of the lemma, we use the
driving measure $\mathcal{M} := \mathcal{M}_n$ (shifted accordingly),
$\mathcal{G}_0 := \F^{\cN}_{(n-1)\Delta} \vee \sigma(N^-)$, where
$\F^{\mathcal{N}}_{(n-1)\Delta} := \sigma(\mathcal{N}|_{[0,(n-1)\Delta]\times\R_{\geq 0}})$,
and $(M_1, M_2) := (M_{n,1}, M_{n,2})$. Note that $\mathcal{M}_n$ is independent of
$\mathcal{G}_0$, and that $\nu_{n,1}$ and $\nu_{n,2}$ are $\mathcal{G}_0$-measurable, as
required, and $M_{n,2}$ is the dominating point process on $((n-1)\Delta, \infty)$,
thinned from the same $\mathcal{M}_n$ as $M_{n,1}$, with stochastic intensity
\begin{equation}
\lambda_{n,2}(t) = \exp\biggl\{\nu_{n,2} + \int_{((n-1)\Delta,\, t)} h_2(t - s)\, M_{n,2}(\d s)\biggr\} = e^{\nu_{n,2} - M_{n,2}(((n-1)\Delta,\, t))\, a}\label{eq:lambda_n_2}
\end{equation}
for $t\geq (n-1)\Delta$.
Lemma~\ref{lemma:monotonicity} then gives, for $n\in\N$,
\begin{equation}\label{eq:domination}
N(((n-1)\Delta, n\Delta]) = M_{n,1}(((n-1)\Delta, n\Delta])
\leq M_{n,2}(((n-1)\Delta, n\Delta]).
\end{equation}
\noindent\emph{Step 4 (bin count bounds).}
From the random time change theorem (first part of \cite[Theorem 7.4.I]{daley03}), we know that the process $\widehat{N}_n(t) := M_{n,2}(\Lambda^{-1}_{n,2}(t))$,
$t\geq 0$, is a Poisson process with intensity 1, where
$\Lambda_{n,2}(t) := \int_{(n-1)\Delta}^{t}\lambda_{n,2}(s)\,\d s$. Let
$0 =: \widehat{S}_{n,0} < \widehat{S}_{n,1} < \widehat{S}_{n,2} < \dots$ denote the consecutive
jump times of $\widehat{N}_n$, so that $E_{n,k} := \widehat{S}_{n,k} - \widehat{S}_{n,k-1}\iidsim \mathrm{Exp}(1),
k\in\N$.
Denoting the event times of $M_{n,2}$ by
$(n-1)\Delta =: T^{(2)}_{n,0} < T^{(2)}_{n,1} < T^{(2)}_{n,2} < \dots$, we obtain
$\widehat{S}_{n,k} = \Lambda_{n,2}(T^{(2)}_{n,k})
= \int_{(n-1)\Delta}^{T^{(2)}_{n,k}} \lambda_{n,2}(s)\,\d s$. Consequently,
\begin{equation}\begin{split}\label{eq:E_nk}
E_{n,k} &= \widehat{S}_{n,k} - \widehat{S}_{n,k-1}
= \int_{T^{(2)}_{n,k-1}}^{T^{(2)}_{n,k}} \lambda_{n,2}(t)\, \d t\\
&= (T^{(2)}_{n,k} - T^{(2)}_{n,k-1})\,
e^{\nu_{n,2} - M_{n,2}(((n-1)\Delta,\, T^{(2)}_{n,k}))\, a}
= (T^{(2)}_{n,k} - T^{(2)}_{n,k-1})\, e^{\nu_{n,2} - (k-1)a}.
\end{split}
\end{equation}
We claim that the family $\{E_{n,k}\}_{n,k\in\N}$ consists of i.i.d.\ $\mathrm{Exp}(1)$ random variables. This is not
immediate from the independence of the $\mathcal{M}_n$ alone, since
$E_{n,k} = e^{\nu_{n,2}-(k-1)a}(T^{(2)}_{n,k}-T^{(2)}_{n,k-1})$ involves the past of $N$
through $\nu_{n,2}$. Set
\begin{align*}
&\mathcal{G}_n
 := \sigma\Big(\big\{N(A) = k,\
\mathcal{N}(B) = l,\
\mathcal{M}_1(C_1)= k_1,\dots,\mathcal{M}_{n-1}(C_{n-1})=k_{n-1}\big\}:  k,l,k_i\in \N_0,\\
& A\in \mathcal{B}((-\infty, 0] ), B\in \mathcal{B}((-\infty,(n-1)\Delta]\times\R_{\geq 0}), C_i\in\mathcal{B}(((n-1)\Delta,\infty)\times \R_{\geq 0}) \Big)
\end{align*}
for $ n\in\N$.
Then it is clear that
\begin{enumerate}[label = (\roman*)]
\item the random variable $\nu_{n,2} = \nu + N(I^{(+)}_n)\,b$ is $\mathcal{G}_n$-measurable, because
$N|_{(-\infty,(n-1)\Delta]}$ is derived from $N|_{(-\infty, 0] }$ and
$\mathcal{N}|_{(-\infty,(n-1)\Delta]\times\R_{\geq 0}}$, and
$I^{(+)}_n \subseteq(-\infty,(n-1)\Delta]$;
\item the sequence $(E_{m,k})_{k\in\N}$ is $\mathcal{G}_n$-measurable for every $m < n$, being derived from
$(\nu_{m,2}, \mathcal{M}_m)$; \item the Poisson random measure $\mathcal{M}_n$ is independent of
$\mathcal{G}_n$.
\end{enumerate}
Fix $n$ and work conditionally on $\mathcal{G}_n$. Introduce the filtration
\begin{equation*}
\mathcal{F}^{(n)}_t := \mathcal{G}_n \vee
\sigma(\mathcal{M}_n|_{((n-1)\Delta,\, t]\times\R_{\geq 0}}),
\qquad t\geq (n-1)\Delta.
\end{equation*}
By~(iii), $\mathcal{M}_n$ is a Poisson random measure with intensity 1 independent of
$\mathcal{G}_n$; hence $M_{n,2}$, thinned from $\mathcal{M}_n$ with the intensity
$\lambda_{n,2}$ given by \eqref{eq:lambda_n_2}, is adapted to $(\mathcal{F}^{(n)}_t)$ with
$(\mathcal{F}^{(n)}_t)$-stochastic intensity $\lambda_{n,2}$, and the event times
$T^{(2)}_{n,k}$ are $(\mathcal{F}^{(n)}_t)$-stopping times. Write
\begin{equation*}
c_{n,k} := e^{\nu_{n,2}-(k-1)a},
\qquad\text{so that}\qquad
\lambda_{n,2}\equiv c_{n,k}\ \text{on } (T^{(2)}_{n,k-1}, T^{(2)}_{n,k}].
\end{equation*}
Note that $c_{n,k}$ is $\mathcal{G}_n$-measurable, and $E_{n,k}=c_{n,k}(T^{(2)}_{n,k}-T^{(2)}_{n,k-1})$ by \eqref{eq:E_nk}.
Because, on $\{T^{(2)}_{n,k-1}<\infty\}$, $\lambda_{n,2}$ has the constant value $c_{n,k}$ until the next event, we obtain,  for $s\geq 0$,
\begin{align*}
\P(T^{(2)}_{n,k}-T^{(2)}_{n,k-1}>s \,|\, \mathcal{F}^{(n)}_{T^{(2)}_{n,k-1}})
= \P(M_{n,2}((T^{(2)}_{n,k-1},\, T^{(2)}_{n,k-1}+s])=0
   \,|\, \mathcal{F}^{(n)}_{T^{(2)}_{n,k-1}})
   = e^{-c_{n,k}s}.
\end{align*}
Since $c_{n,k}$ is $\mathcal{F}^{(n)}_{T^{(2)}_{n,k-1}}$-measurable,
\begin{equation*}
\P(E_{n,k}>x \,|\, \mathcal{F}^{(n)}_{T^{(2)}_{n,k-1}})
= \P(T^{(2)}_{n,k}-T^{(2)}_{n,k-1}>x/c_{n,k} \,|\, \mathcal{F}^{(n)}_{T^{(2)}_{n,k-1}})
= e^{-x}, \qquad x\geq 0.
\end{equation*}
This implies  that for every $k$ we have $
	\P(E_{n,k}>x_k |\, \mathcal{G}_n)
	=  e^{-x_k} = 	\P(E_{n,k}>x_k )$, $x_k\geq 0$.
As $E_{n,1},\dots,E_{n,k-1}$ and $c_{n,k}$ are
$\mathcal{F}^{(n)}_{T^{(2)}_{n,k-1}}$-measurable and $\mathcal{G}_n\subseteq
\mathcal{F}^{(n)}_{T^{(2)}_{n,k-1}}$, iterating over the indices $k=K,K-1,\dots,1$, the tower
property yields
\begin{equation}\label{eq:condtailprob}
\P(E_{n,1}>x_1,\dots,E_{n,K}>x_K \,|\, \mathcal{G}_n)
= \prod_{k=1}^{K} e^{-x_k}= \prod_{k=1}^{K} \P(E_{n,k}>x_k |\, \mathcal{G}_n)
\end{equation}
for every $K\in\N$ and $x_1,\dots,x_K\geq 0$.
Hence, conditionally on $\mathcal{G}_n$, the block $(E_{n,k})_{k\in\N}$ consists of i.i.d.\ $\mathrm{Exp}(1)$ random variables. Moreover, in view of \eqref{eq:condtailprob} we have
\begin{equation}\label{eq:condist}
\P(E_{n,1}>x_1,\dots,E_{n,K}>x_K \,|\, \mathcal{G}_n) =\P(E_{n,1}>x_1,\dots,E_{n,K}>x_K ),
\end{equation}
which in turn implies that  the block $B_n:=(E_{n,k})_{k\in\N}$ is independent of $\mathcal{G}_n$.
Consequently, using that
 by~(ii)
$\sigma(B_1,\dots,B_{n-1}) \subseteq \mathcal{G}_n$, we conclude that, for all $n\in\N$,
\begin{align*}
\E\bigg[\prod_{m=1}^{n} f_m(B_m)\bigg]
&= \E\bigg[\E\bigg[\prod_{m=1}^{n} f_m(B_m)\,\bigg|\,\mathcal{G}_n\bigg]\bigg]
= \E\bigg[\bigg(\prod_{m=1}^{n-1} f_m(B_m)\bigg)\,\E[f_n(B_n)\mid\mathcal{G}_n]\bigg] \nonumber\\
&= \E\bigg[\prod_{m=1}^{n-1} f_m(B_m)\bigg]\, \E[f_n(B_n)]
= \dots
= \prod_{m=1}^{n} \E[f_m(B_m)]
\end{align*}
for all bounded
measurable functions $f_1,\dots,f_n$. Hence, the blocks $B_n$, $n\in\N$, are independent, and the
whole family $(E_{n,k})_{n,k\in\N}$ consists of i.i.d.\ $\mathrm{Exp}(1)$ random variables.

Finally, since the internal history of $N$ satisfies $\F^N_{(n-1)\Delta}
\subseteq \mathcal{G}_n$ and $\sigma((E_{m,k})_{k\in\N},\ m<n)
= \sigma(B_1,\dots,B_{n-1}) \subseteq \mathcal{G}_n$ by~(ii), the block
$(E_{n,k})_{k\in\N}$ is also independent of
$\F^N_{(n-1)\Delta}\vee\sigma((E_{m,k})_{k\in\N},\ m<n)$, and in particular of
$N(I^{(+)}_n)$. Thus, with
\begin{equation*}
S_{n,k} := T^{(2)}_{n,k} - T^{(2)}_{n,k-1}
= e^{-(\nu_{n,2} - (k-1)a)}\, E_{n,k}
= e^{-(\nu + N(I^{(+)}_n)\, b - (k-1)a)}\, E_{n,k},
\qquad n,k\in\N,
\end{equation*}
we represent the bin count of the dominating process as
\begin{equation}\label{eq:dominating_bin_count}
M_{n,2}(((n-1)\Delta, n\Delta])
= \sum_{k=1}^\infty 1\{S_{n,1} + \dots + S_{n,k} \leq \Delta\}
= \sum_{k=1}^\infty 1\{S_{n,1} + \dots + S_{n,k} < \Delta\},
\end{equation}
where the last equality holds because the partial sums $S_{n,1}+\dots+S_{n,k}$ have continuous distributions. Combining \eqref{eq:domination} and
\eqref{eq:dominating_bin_count} yields \eqref{eq:bin_count_bound}, which completes the proof.\end{proof}

\begin{proof}[Proof of Theorem~\ref{thm:bin_count_domination_sequence}]
 Recall the notation $N_n =  N(((n-1)\Delta, n\Delta ])$ for $n\in\Z$. We prove \eqref{eq:bin_count_ub} by induction over $n$.
For $n\leq 0$, we have $N_n = \overline{N}_n$ by definition. Now assume $n\geq 1$, and $N_m \leq \overline{N}_m$ for all $m< n$. Then,
\begin{align*}
&\overline{S}_{n,k}
= \exp\Bigl\{-\Bigl(\nu + b\sum_{l = p^{(-)}}^{p}\overline{N}_{n-l}- (k-1)a\Bigr)\Bigr\} E_{n,k}\\
&\leq \exp\Bigl\{-\Bigl(\nu + b\sum_{l = p^{(-)}}^{p}{N}_{n-l}  - (k-1)a\Bigr)\Bigr\} E_{n,k} = e^{-(\nu + bN(I_n^{(+)}) - (k-1)a)} E_{n,k} = S_{n,k},
\end{align*}
where in the inequality we use the induction hypothesis together with $b\geq0$. Hence, using Theorem~\ref{thm:bin_count_domination}, we conclude that
\begin{equation*}
\overline{N}_n = \sum_{k = 1}^{\infty}1\{\overline{S}_{n,1} + \dots + \overline{S}_{n,k} < \Delta\}\geq  \sum_{k = 1}^{\infty}1\{{S}_{n,1} + \dots + {S}_{n,k} < \Delta\}\geq  N_n.
\end{equation*}
We have proven the induction step and thus \eqref{eq:bin_count_ub}.
\end{proof}
\begin{proof}[Proof of Theorem~\ref{thm:moment_bound}] For $n\in\N$, we have
\begin{align*}
&\E[\overline{N}_n|\sigma(\oN_m,\ m<n)] =  \E\Bigl[\sum_{k=1}^\infty \mathbf{1}\{\oS_{n,1}+\cdots+\oS_{n,k} < \Delta\} \Big| \sigma(\oN_m,\ m<n)\Bigr]\\
& =\sum_{k=1}^\infty\P(\oS_{n,1}+\cdots+\oS_{n,k} < \Delta | \sigma(\oN_m,\ m<n) ) \leq\sum_{k=1}^\infty\P(\oS_{n,k} < \Delta | \sigma(\oN_m,\ m<n) )\\
& {=}\sum_{k=1}^\infty\P\biggl(\exp\Bigl\{-\Bigl(\nu + b\sum_{l = p^{(-)}}^{p}\overline{N}_{n-l} - (k-1)a\Bigr)\Bigr\}E_{n,k} < \Delta \bigg| \sigma(\oN_m,\ m<n) \biggr)\\
&{=} \sum_{k=1}^\infty\biggl(1 - \exp\Bigl\{-\Delta \exp\Bigl(\nu + b\sum_{l = p^{(-)}}^p\oN_{n-l} - (k-1)a\Bigr)\Bigr\}\biggr)\\
&{=} \sum_{k=0}^\infty\biggl(1 - \exp\Bigl\{-\Delta \exp\Bigl(\nu + b\sum_{l = p^{(-)}}^p \oN_{n-l} -ka\Bigr)\Bigr\}\biggr).
\end{align*}
Here, the second-to-last equality uses the fact that the block $(E_{n,k})_{k\in\N}$ is independent
of $\sigma(\oN_m,\ m<n)$, as noted after Definition~\ref{def:bcds}, together with
$E_{n,k}\sim\mathrm{Exp}(1)$, so that
$\P(e^{-c}E_{n,k}<\Delta\mid\oN_{n-1}=x_1,\dots) = 1-\exp\{-\Delta e^{c}\}$ for the
$\sigma(\oN_m,\ m<n)$-measurable exponent
$c = \nu+b\sum_{l=p^{(-)}}^{p}\oN_{n-l}-(k-1)a$.

Next, write $W := \sum_{l = p^{(-)}}^{p} \oN_{n-l}$.
We split the sum in the exponent at the index
\begin{equation*}
K := \biggl\lfloor \frac{b}{a}W \biggr\rfloor,
\end{equation*}
so that
\begin{align*}
\E[\overline{N}_n\mid\sigma(\oN_m,\ m<n)]
&\leq
\sum_{k=0}^{K}
(1 - \exp\{-\Delta e^{\nu + bW - ka}\})\\
&\quad+\sum_{k=K+1}^{\infty}
(1 - \exp\{-\Delta e^{\nu +  bW - ka}\}).
\end{align*}
To bound the first sum, we use the inequality $1 - e^{-x} \le 1$, and for the second sum we use $1 - e^{-x} \le x$, so that we obtain
\begin{equation*}
\E[{{\oN}_n}\mid\sigma(\oN_m,\ m<n)]\le
(K+1)
+ \sum_{k=K+1}^{\infty}
\Delta e^{\nu + W b - k a}.
\end{equation*}
Reindexing the sum gives
\begin{equation*}
\E[ {\oN}_n\mid\sigma(\oN_m,\ m<n)]\leq
K+1
+ \Delta e^\nu \sum_{l=0}^{\infty}
e^{W b - (K+1+l)a}.
\end{equation*}
Since
$$
(K+1)a=  \biggl(\biggl\lfloor \frac{b}{a} W \biggr\rfloor + 1\biggr)a \geq  \biggl\lceil \frac{b}{a}W \biggr\rceil a \geq  bW,
$$ we have
$W b - (K+1+l)a \le W  b - b W - la  =  - la$,
so that

\begin{align*}
&\E[\oN_n\mid\sigma(\oN_m,\ m<n)]
\leq
K+1
+ \Delta e^\nu
\sum_{l=0}^{\infty}
e^{-la}
=K+1 + \frac{\Delta e^\nu}{1 - e^{-a}}\\
&\qquad \leq \frac{b}{a} W +1  +  \frac{\Delta e^\nu}{1 - e^{-a}}
 = \frac{b}{a} \sum_{l = p^{(-)}}^{p} \oN_{n-l}+ 1+ \frac{\Delta e^\nu}{1 - e^{-a}},
\end{align*}
which proves the claim.
\end{proof}

\begin{proof}[Proof of Lemma~\ref{lemma:markov_chain}] We denote the $k$-th coordinate of $X_n$ by $X_{n,k}$.
By definition, we have
$X_0 = (\overline{N}_0, \overline{N}_{-1}, \dots, \overline{N}_{-p+1})
= ({N}_0, {N}_{-1}, \dots, {N}_{-p+1}).
$
For $n\geq 1$, note that
\begin{equation*}
 X_{n,1}
 = \oN_{n} =  \sum_{k = 1}^{\infty}\mathbf{1}\{\overline{S}_{n,1} + \dots + \overline{S}_{n,k} < \Delta\}.
 \end{equation*}
 As
\begin{align*}
 \oS_{n,k}
 &=  \exp\Bigl\{-\Bigl(\nu + b\sum_{l = p^{(-)}}^{p}\overline{N}_{n-l} - (k-1)a\Bigr)\Bigr\}E_{n,k}\\
 &=  \exp\Bigl\{-\Bigl(\nu + b\sum_{l = p^{(-)}}^{p}X_{n-1, l} - (k-1)a\Bigr)\Bigr\}E_{n,k},
\end{align*}
the random variable $X_{n,1}$ is a function of $X_{n-1}$ and the innovation block $(E_{n,k})_{k\in\N}$.
The other components of $X_n$ are obtained from $X_{n-1}$  by a shift of indices:
$X_{n,k} = \oN_{n +1 - k} = X_{n -1, k - 1}$, $k = 2, 3, \dots , p$.
Summarizing, we have
$X_n = f(X_{n-1},( E_{n,k})_{k\in\N} )$, $n\in\N$,
for a deterministic function $f$. Moreover, since
$X_0, \dots, X_{n-1}$ are measurable with respect to $\F^N_0\vee\sigma((E_{m,k})_{k\in\N},\ m<n)$, the block $(E_{n,k})_{k\in\N}$ is independent of $\mathcal{F}^{X}_{n-1}$. Hence, $(X_n)_{n\geq 0}$ is a Markov chain. In addition, as the innovation blocks
$(E_{n,1}, E_{n,2}, \dots)$ are identically distributed in $n$ and $f$ does not depend
on $n$, the Markov chain is homogeneous.
\end{proof}

\begin{proof}[Proof of Lemma~\ref{lemma:foster}]
Fix $x = (x_1,\dots,x_p)\in\N_0^p$. Define
$$
e_j := \E_x[\oN_j],\quad j\geq -p+1,
\qquad\text{and}\qquad u^\star := \frac{C}{1-\rho}.
$$
If $X_0 = x$, then $e_j = x_{1-j} \leq V(x)$ for $j = -p+1,\dots,0$. For $j\geq 1$,
note that $\sigma(X_0)\subseteq \sigma(\oN_m,\ m\leq j-1)$, so that the tower
property together with Theorem~\ref{thm:moment_bound} yields
\begin{equation}
e_j = \E_x\bigl[\E[\oN_j\mid \sigma(\oN_{j-1},\oN_{j-2},\dots)]\bigr]
\leq \frac{b}{a}\sum_{l=p^{(-)}}^{p} e_{j-l} + C.
\label{eq:condrec}
\end{equation}
We claim that
\begin{equation}\label{eq:e_bound}
e_j \leq \rho^{\lceil j/p\rceil}\,V(x) + u^\star,\quad j\geq 1,
\end{equation}
and prove the claim by induction over $j$. Indeed, for $j = 1$ we have $\lceil j/p\rceil = 1$, and
$e_{1-l} = x_{l}$ for $l\in\{p^{(-)}, p^{(-)}+1, \dots, p\}$.  From \eqref{eq:condrec} we obtain
\begin{equation*}
e_1 \leq \frac{b}{a}\sum_{l = p^{(-)}}^{p} x_l + C
\leq \frac{b}{a} V(x) + C \leq \rho V(x) + u^\star,
\end{equation*}
where we use $
\sum_{l = p^{(-)}}^{p} x_l \leq \sum_{k = 1}^{p} x_k = V(x)$,
$\frac{b}{a} \leq \frac{b}{a}(p^{(+)}+1) = \rho$, and $
C \leq \frac{C}{1-\rho} = u^\star $.
So \eqref{eq:e_bound} holds for $j = 1$. Now, fix $j\geq 2$, set
$m := \lceil j/p\rceil \geq 1$, and assume that \eqref{eq:e_bound} holds for all
$1\leq i<j$.
Consider an index $j-l$  appearing in \eqref{eq:condrec}.
If $j-l\geq 1$, then
$\lceil (j-l)/p\rceil \geq \lceil (j-p)/p\rceil = m-1$,
so by the induction hypothesis together with $\rho<1$, we obtain
$e_{j-l} \leq \rho^{\lceil (j-l)/p\rceil}\,V(x) + u^\star \leq \rho^{m-1}V(x) + u^\star$.
If $j-l< 1$, then $j\leq p$, so $m = 1$ and we have
$e_{j-l} = x_{1-j+l} \leq V(x) = \rho^{m-1}V(x)$. In both cases,
$e_{j-l} \leq \rho^{m-1}V(x) + u^\star$.
As the sum in \eqref{eq:condrec} has $p - p^{(-)} + 1 = p^{(+)}+1$ terms, we obtain
$$
e_j \leq \frac{b}{a}(p^{(+)}+1)(\rho^{m-1}V(x) + u^\star) + C
= \rho^{m}V(x) + \rho u^\star + C = \rho^{m}V(x) + u^\star,
$$
since $\rho u^\star + C = u^\star$. This proves \eqref{eq:e_bound}.

Finally, $V(X_{mp}) = \sum_{k=1}^{p} X_{mp,k} = \sum_{j=(m-1)p+1}^{mp}\oN_j$, and
$\lceil j/p\rceil = m$ for exactly these $j$, so that, by \eqref{eq:e_bound},
$$
\E_x[V(X_{mp})] = \sum_{j=(m-1)p+1}^{mp} e_j
\leq p\rho^{m}V(x) + \frac{pC}{1-\rho}.
$$
The proof is finished.
\end{proof}

\begin{proof}[Proof of Theorem~\ref{thm:sufficient_stability_condition}]
Since $h$ is nonpositive, Theorem~\ref{thm:nonexplosion}\ref{item:sufficient} yields that $N$ is nonexplosive, so that \eqref{eq:log_lin_dynamics} may be rewritten as
 $$
\lambda(t) = \phi\biggl(\int_{(-\infty, t)} h(t -s) N(\d s)\biggr)
$$
with
$
\phi(x) =1_{(-\infty, 0]}(x)e^{\nu+x}  +
1_{(0, \infty)}(x) e^{\nu}$. This rate function $\phi$ is bounded and Lipschitz continuous (with Lipschitz constant $e^\nu$), and
$$
\phi\biggl(\int_{(-\infty, t)} h(t -s) N(\d s)\biggr) = \exp\biggl\{\nu + \int_{(-\infty, t)} h(t -s) N(\d s)\biggr\}
$$
holds for nonpositive $h$.
All hypotheses of \cite[Theorem 2]{bremaud96} are fulfilled, and thus  $N$ is stable in variation with respect to $N^-$. The same theorem gives us that the stationary limit is unique and does not depend on the choice of the initial condition $N^-$ as long as \eqref{eq:initial_condition_stability} holds.
\end{proof}
For the proof of Theorem~\ref{thm:regeneration_time}, we first state and prove the following lemma.
\begin{lemma}\label{lemma:uniform_prob_bound}
Let $(X_n)_{n\geq 0}$ be the bin count dominating Markov chain from~\eqref{eq:X_def_alt} and
let $D := \{x\in\N_0^p: \sum_{k = 1}^p x_k \leq R\}$ for some $R>0$. Then
\begin{equation}\label{eq:unif_zero_prob_bound}
\inf_{x\in D}\P_x(X_{p} =0 )
\geq \exp\{ - \Delta p e^{\nu + bR}\} = :\varepsilon >0.
\end{equation}
\end{lemma}
\begin{proof}[Proof of Lemma~\ref{lemma:uniform_prob_bound}]
Note that
$
\{X_{p} = 0\} = \{\oN_{1} = \oN_{2} = \dots = \oN_{p} = 0\}.
$
We prove the stronger statement
\begin{equation}\label{eq:lb_prob_empty_region}
\P_x(\oN_{1} = \dots = \oN_{j} = 0 )
\geq \exp\{ - \Delta j e^{\nu + bR}\},
\quad 1\leq j \leq p,\ x\in {D},
\end{equation}
by induction over $j$. Throughout, we use that
$\oN_j = 0$ if and only if $\oS_{j,1} \geq \Delta$, and that $E_{j,1}$ is independent of
$\sigma(\oN_m,\ m<j)\supseteq\F^X_{j-1}$.
For $j = 1$: given $X_0 = x $, we have $\oN_{1-l} = x_l$, $l = 1,\dots,p$, so that
\begin{align*}
\P_x(\oN_{1} = 0)
&= \P_x(\oS_{1,1}\geq \Delta)
= \P_x\Bigl(E_{1,1}\geq \Delta\exp\Bigl\{\nu + b\sum_{l = p^{(-)}}^{p}x_l\Bigr\}
\Bigr)\\
&= \exp\Bigl\{-\Delta\exp\Bigl(\nu + b\sum_{l = p^{(-)}}^{p}x_l\Bigr)\Bigr\}
\geq \exp\{-\Delta e^{\nu + bR}\},
\end{align*}
where the last step uses $b \geq 0$ together with
$\sum_{l = p^{(-)}}^{p}x_l \leq \sum_{l=1}^p x_l\leq R$ for $x\in D$.
For the induction step, let $2\leq j\leq p$ and assume that
\eqref{eq:lb_prob_empty_region} holds for $j - 1$. Then
\begin{align*}
&\P_x(\oN_{i} = 0,\ i = 1, \dots, j)
= \E_x\Bigl[\prod_{i=1}^{j}\mathbf{1}\{\oN_{i} = 0\}\Bigr]= \E_x\Bigl[\P(\oS_{j,1}\geq \Delta\,|\,\F^X_{j-1})
\prod_{i=1}^{j-1}\mathbf{1}\{\oN_{i} = 0\}\Bigr]\\
&= \E_x\Bigl[\exp\Bigl\{-\Delta\exp\Bigl(\nu + b\sum_{l = p^{(-)}}^{p}\oN_{j-l}\Bigr)\Bigr\}
\prod_{i=1}^{j-1}\mathbf{1}\{\oN_{i} = 0\}\Bigr]\\
&= \E_x\Bigl[\exp\Bigl\{-\Delta\exp\Bigl(\nu
+ b\sum_{l = \max\{j,\, p^{(-)}\}}^{p}x_{l-j+1}\Bigr)\Bigr\}
\prod_{i=1}^{j-1}\mathbf{1}\{\oN_{i} = 0\}\Bigr]\\
&\geq \exp\{-\Delta e^{\nu + bR}\}\,
\E_x\Bigl[\prod_{i=1}^{j-1}\mathbf{1}\{\oN_{i} = 0\}\Bigr]\\
&\geq
\exp\{-\Delta e^{\nu + bR}\}
\exp\{-\Delta(j-1)e^{\nu + bR}\}
= \exp\{-\Delta j e^{\nu + bR}\},
\end{align*}
where in the last equality we use that on the event $\{\oN_1 = \dots = \oN_{j-1} = 0\}$ the summands
with $j-l\geq 1$ vanish, while for $j-l\leq 0$ we have $\oN_{j-l} = x_{l-j+1}$, the
first inequality uses $\sum_{l = \max\{j,p^{(-)}\}}^{p}x_{l-j+1}\leq\sum_{l=1}^{p}x_l\leq R$ for
$x\in D$ together with $b\geq0$, and the second inequality uses the induction hypothesis. This concludes the induction and proves \eqref{eq:lb_prob_empty_region}. Letting $j = p$ and taking the infimum over $x\in D$ yields~\eqref{eq:unif_zero_prob_bound}, and the proof is finished.
\end{proof}

\begin{proof}[Proof of Theorem~\ref{thm:regeneration_time}]
Let $(X_n)_{n\in\N_0}$ be the bin count dominating Markov chain of $N$, and let $V:\N_0^p\to\N_0$ be the Lyapunov function defined in~\eqref{eq:lyapunov}. Let $C := 1 + \Delta e^\nu/(1-e^{-a})$. Choose $m\in\N$  such that
$p\rho^m < 1$, and choose $R>0$ such that
\begin{equation*}
R \geq \frac{pC/(1-\rho)+1}{1-p\rho^m} > 0.
\end{equation*}
For this fixed $R$ define the set
$D := D_R := \{x\in\N_0^p: V(x)\leq R\}$.
Then we can derive the following two bounds from
Lemma~\ref{lemma:foster}:
\begin{enumerate}
\item For $x\notin D$, i.e., $V(x) > R$, \eqref{eq:driftR} gives
\begin{align*}
\E_x[V(X_{mp})]
&\leq p\rho^m V(x) + \frac{pC}{1-\rho} = (1 - (1-p\rho^m))V(x) + \frac{pC}{1-\rho}\\
&\leq V(x) - (1-p\rho^m)R + \frac{pC}{1-\rho}\leq V(x) - 1.
\end{align*}
\item For $x\in D$, \eqref{eq:driftR} together with $V(x)\leq R$ gives
$$
\E_x[V(X_{mp})]\leq p\rho^m R + \frac{pC}{1-\rho}.
$$
\end{enumerate}
Abbreviating $c := p\rho^m R + \frac{pC}{1-\rho} + 1$, we summarize both bounds in the
drift criterion
\begin{equation}\label{eq:drift_criterion}
\E_x[V(X_{mp})] - V(x)\leq -1 + c\,\mathbf{1}_{\{x\in D\}},
\quad x\in\N_0^p.
\end{equation}
Consider now the $T$-skeleton $(X_{nT})_{n\geq0}$ with $T := mp$, again a homogeneous
Markov chain, and the \emph{return time}
$\tau_{D} := \inf\{n \geq 1: X_{nT}\in D\}$.
The skeleton is a homogeneous Markov chain on $\N_0^p$, the function $V$ from
\eqref{eq:lyapunov} is nonnegative and finite, and \eqref{eq:drift_criterion} is
precisely the drift condition (V2) of \cite[p.~262]{meyn93}. Hence the first assertion of \cite[Theorem~11.3.4]{meyn93} applies to
the skeleton and gives, with $\tau_{D}$ the first return time to $D$,
\begin{equation}\label{eq:tau_{D}}
\E_x[\tau_{D}]\leq V(x) + c\,\mathbf{1}_{\{x\in {D}\}} < \infty,
\quad x\in\N_0^p.
\end{equation}
Next, denote the successive visits of the skeleton to $D$ by $\tau_{D,0} := 0$ and
$\tau_{D,k} := \inf\{n > \tau_{D,k-1}: X_{nT}\in D\}$, $k\in\N$,
so that $\tau_{D,1} = \tau_{D}$. From \eqref{eq:tau_{D}}, we obtain
\begin{equation}\label{eq:first_excursion}
\E[\tau_{D,1}] \leq \E[V(X_0)] + c = \E[N((-s_0,0])] + c < \infty,
\end{equation}
using $V(X_0) = \sum_{k=1}^p N_{1-k} = N((-s_0, 0])$ together with the assumption on the
initial condition. Moreover, the strong Markov property of $(X_n)_n$ together with \eqref{eq:tau_{D}}, and $X_{\tau_{D,k-1}T}\in D$ gives, for $k\geq 2$,
\begin{equation}\label{eq:excursion_length}
\E[\tau_{D,k} - \tau_{D,k-1}\,|\,\F^X_{T\tau_{D,k-1}}]
\leq V(X_{T\tau_{D,k-1}}) + c \leq R + c.
\end{equation}
In particular, all $\tau_{D,k}$, $k \in \N$, are finite. Consider the index of the first visit of the skeleton to $D$ that is followed by an empty
window of $p$ bins,
$\kappa := \inf\big\{k\in\N: X_{ T\tau_{D,k} + p} = 0\big\}$,
and abbreviate $A_l := \{X_{T\tau_{D,l}+p}\neq 0\}$, $l\in\N$, so that
$\{\kappa > k\} = \bigcap_{l=1}^k A_l$. Note that for $l < k$,
$T\tau_{D,l} + p\leq T(\tau_{D,k}-1) + p \leq  T\tau_{D,k}$,
as $p \leq T = mp$. Consequently, $A_l\in\F^X_{ T\tau_{D,k}}$ for $l<k$. Thus,  by the tower property,
\begin{align*}
\P(\kappa > k)
&= \E\Bigl[\prod_{l=1}^{k}\mathbf{1}_{A_l}\Bigr]
= \E\Bigl[\prod_{l=1}^{k-1}\mathbf{1}_{A_l}\,
\E[\mathbf{1}_{A_k}\,|\,\F^X_{ T\tau_{D,k}}]\Bigr]\\
&= \E\Bigl[\prod_{l=1}^{k-1}\mathbf{1}_{A_l}
\bigl(1 - \P(X_{\tau_{{D},k}T+p} = 0\,|\, X_{ T\tau_{D,k}})\bigr)\Bigr]\\
&\leq (1-\varepsilon)\,
\E\Bigl[\prod_{l=1}^{k-1}\mathbf{1}_{A_l}\Bigr]
= (1-\varepsilon)\,\P(\kappa > k - 1),\quad k\in\N,
\end{align*}
where we use the strong Markov property of $(X_n)_n$, and $X_{ T\tau_{D,k}}\in {D}$ together with Lemma~\ref{lemma:uniform_prob_bound} in the
second-to-last step; here, $\varepsilon = \exp\{-\Delta p e^{\nu+bR}\}$ is the constant from \eqref{eq:unif_zero_prob_bound}. Iterating yields $\P(\kappa > k)\leq (1-\varepsilon)^k$, so that
$\kappa < \infty$ a.s. and
\begin{equation}\label{eq:kappa_bound}
\E[\kappa] = \sum_{k=0}^\infty\P(\kappa > k)\leq \frac{1}{\varepsilon}.
\end{equation}

Next, we control $\E[\tau_{D,\kappa}]$. Note that
$\{\kappa\geq k\}\subseteq\{\kappa\geq k-1\}$ and that
$\{\kappa\geq k-1\} = \bigcap_{l=1}^{k-2}A_l\in\F^X_{\tau_{D,k-1}T}$, since
$\tau_{D,l}T+p \leq \tau_{D,k-1}T$ for $l\leq k-2$. Consequently,
\begin{equation}\begin{split}
&\E[\tau_{D,\kappa}]
= \E\Bigl[\sum_{k=1}^{\kappa}(\tau_{D,k}-\tau_{D,k-1})\Bigr]
= \sum_{k=1}^\infty\E[\mathbf{1}\{\kappa\geq k\}
(\tau_{D,k}-\tau_{D,k-1})]\\
&\leq \E[\tau_{D,1}] + \sum_{k=2}^\infty\E\bigl[\mathbf{1}\{\kappa\geq k-1\}\,
\E[\tau_{{D},k}-\tau_{{D},k-1}\,|\,\F^X_{\tau_{D,k-1}T}]\bigr]\\
&\leq
\E[N((-s_0,0])] + c + (R+c)\sum_{j=1}^\infty\P(\kappa\geq j)\leq
\E[N((-s_0,0])] + c + \frac{R+c}{\varepsilon} < \infty.\label{eq:tau_D_kappa}
\end{split}
\end{equation}
Here, the second inequality follows from \eqref{eq:first_excursion} and \eqref{eq:excursion_length}, and the last one from \eqref{eq:kappa_bound}.

Finally, set $n^* := \inf\{n\geq p: N_n = N_{n-1} = \ldots = N_{n-p+1} = 0\}$. On
$\{n^* = n\}$ with $n \geq p$, the window $((n-p)\Delta, n\Delta]$ of length $s_0$ is empty
and $n\Delta \geq s_0$, so that $\tau\leq n\Delta$.
Together with the bin count domination \eqref{eq:bin_count_ub} and the equivalence
$X_n = 0\Leftrightarrow \oN_n = \ldots = \oN_{n-p+1} = 0$ for $n\geq p$, we conclude
\begin{align*}
&\E[\tau]
\leq \Delta\,\E[n^*]
= \Delta\,\E[\inf\{n\geq p: X_n = 0\}]\\
&\leq \Delta(T\,\E[\tau_{D,\kappa}] + p)
\leq
\Delta T\biggl(\E[N((-s_0,0])] + c + \frac{R+c}{\varepsilon}\biggr) + s_0 < \infty,
\end{align*}
where in the second-to-last inequality we use the fact that $X_{T\tau_{{D},\kappa}+p} = 0$ by definition
of $\kappa$, together with $T\tau_{{D},\kappa}+ p\geq p$, and in the last inequality we use \eqref{eq:tau_D_kappa}. This completes the proof.
\end{proof}
\begin{proof}[Proof of Theorem~\ref{thm:stationary_construction}]\strut

\noindent{\it Step 1 (proof of \eqref{eq:stat}).}
For $l\geq s_0$ and $w\in(0,l)$, define the deterministic-shift spine
\begin{equation}\label{eq:spine_lw}
\tau_n(l,w):=
\begin{cases}
-w-\sum_{k=n+1}^{0}L_k, & n\leq-1,\\
-w, & n=0,\\
-w+l+\sum_{k=2}^{n}L_k, & n\geq1,
\end{cases}
\end{equation}
which is a measurable function of $(l,w,\{C_k\}_{k\in\Z\setminus\{1\}})$.
Comparing \eqref{eq:spine_lw} with the definition of $\tau_n$ and of $\taustat_n$, one checks directly that for all $n\in\Z$ and $w\in(0,\ell(C_1))$
\begin{align}\label{eq:spine_identification}
\tau_n-w=\tau_n(\ell(C_1),w) \quad\text{and}\quad  \tau_n(\widetilde L_1,U\widetilde L_1)&=\taustat_n.
\end{align}
Let $\E_{\mathrm{rest}}$ denote expectation over
$\{C_k\}_{k\in\Z\setminus\{1\}}$ only. Using $\tau_1=L_1$,
\eqref{eq:palm}, the first identity in \eqref{eq:spine_identification}, and
conditioning on $C_1$ (which is independent of
$\{C_k\}_{k\in\Z\setminus\{1\}}$), we obtain for the right-hand side of \eqref{eq:stat}
\begin{equation}
\begin{split}
&\frac{1}{m}\E\biggl[\int_{0}^{\tau_1} f( \NPalm (\cdot +r) )\d r\biggr] =\frac{1}{m}\,\E\biggl[\int_{0}^{L_1}
f\Bigl(\sum_{k\in\Z}C_k(\cdot-\tau_{k-1}+w)\Bigr)\d w\biggr]
\\
&=\frac{1}{m}\,\E\biggl[\int_{0}^{L_1}
f\Bigl(\sum_{k\in\Z}C_k(\cdot-\tau_{k-1}(\ell(C_1),w))\Bigr)\d w\biggr]
\\
&=\frac{1}{m}\int\biggl(\int_{0}^{\ell(c)}
\E_{\mathrm{rest}}\Bigl[
f\Bigl(c(\cdot+w)+\sum_{k\in\Z\setminus\{1\}}
C_k(\cdot-\tau_{k-1}(\ell(c),w))\Bigr)\Bigr]\d w\biggr)
\P_C(\d c),
\label{eq:after_conditioning}
\end{split}
\end{equation}
where in the last step we use that, for fixed $(l,w)$, the argument of $f$
is a measurable functional of $\{C_k\}_{k\in\Z\setminus\{1\}}$ alone, and
that $\tau_0(\ell(c),w)=-w$, so that
$C_1(\cdot-\tau_0(\ell(c),w))$ becomes $c(\cdot+w)$ on $\{C_1=c\}$. Next, for fixed $c$, substituting $w=u\,\ell(c)$, $u\in(0,1)$, in the inner
integral of \eqref{eq:after_conditioning}, we obtain
\begin{equation}\begin{split}
&\frac{1}{m}\E\biggl[\int_{0}^{\tau_1} f( \NPalm (\cdot +r) )\d r\biggr]
\\
&=\int\int_{0}^{1}
\E_{\mathrm{rest}}\Bigl[
f\Bigl(c(\cdot+u\ell(c))+\sum_{k\in\Z\setminus\{1\}}
C_k(\cdot-\tau_{k-1}(\ell(c),u\ell(c)))\Bigr)\Bigr]
\d u\; \P_{\widetilde{C}_1}(\d c)
\\
&=\E\Bigl[
f\Bigl(\widetilde C_1(\cdot+U\widetilde L_1)+\sum_{k\in\Z\setminus\{1\}}
C_k(\cdot-\tau_{k-1}(\widetilde L_1,U\widetilde L_1))\Bigr)\Bigr],
\label{eq:after_sizebias}\end{split}
\end{equation}
where the last equality holds because $(\widetilde C_1,U)$ has law
$\P_{\widetilde{C}_1}\otimes\Unif(0,1)$ and is independent of
$\{C_k\}_{k\in\Z\setminus\{1\}}$.
By the second identity in \eqref{eq:spine_identification}, we have
$\tau_{k-1}(\widetilde L_1,U\widetilde L_1)=\taustat_{k-1}$ for all $k$. Finally,
$\widetilde C_1(\cdot+U\widetilde L_1)=\widetilde C_1(\cdot-\taustat_0)$ since
$\taustat_0=-U\widetilde L_1$. Hence, the argument of $f$ in \eqref{eq:after_sizebias} is exactly the right-hand side of \eqref{eq:stat_version}, that is, $\Nstat$. Thus, \eqref{eq:after_sizebias} reads
\begin{equation*}
\frac{1}{m}\E\biggl[\int_{0}^{\tau_1} f( \NPalm (\cdot +r) )\d r\biggr]
=\E[f(\Nstat)]
\end{equation*}
for every nonnegative measurable $f$, which is \eqref{eq:stat}.

\noindent\emph{Step 2 ($\Nstat$ is stationary).} This does not follow from \eqref{eq:stat} alone: in \eqref{eq:stat} the origin is only moved across the first cycle $[0,\tau_1)$ of $\NPalm$. We use in addition that the cycles $(C_k)_{k\in\Z}$ in the Palm glue are i.i.d., so that the law of $\NPalm$ does not change when the time origin is moved from $0$ to any other point $\tau_j$ of the Palm spine. This allows us to move the origin across $n$ cycles instead of one, and for $n\to\infty$ the effect of an additional fixed shift by $t$ becomes negligible.

\noindent\emph{Step 2a (moving the origin to $\tau_j$).} Fix $j\in\Z$. By \eqref{eq:palm}, for every $s\in\R$,
\begin{equation}\label{eq:palm_origin_moved}
\NPalm(\cdot+\tau_j+s)=\sum_{k\in\Z}C_k(\cdot+s-(\tau_{k-1}-\tau_j))
=\sum_{k\in\Z}C_{k+j}(\cdot+s-(\tau_{k+j-1}-\tau_j)),
\end{equation}
where we substitute $k\mapsto k+j$ in the second equality. By definition of the Palm spine, $\tau_i-\tau_{i-1}=L_i$ for all $i\in\Z$, so that $\tau_n-\tau_j=\sum_{i=j+1}^{n}L_i$ for $n>j$ and $\tau_n-\tau_j=-\sum_{i=n+1}^{j}L_i$ for $n<j$. Consequently,
\begin{equation*}
\tau_{k+j-1}-\tau_j=
\begin{cases}
-\sum_{i=k}^{0}L_{i+j}, & k\leq 0,\\
0, & k=1,\\
\sum_{i=1}^{k-1}L_{i+j}, & k\geq 2,
\end{cases}
\end{equation*}
which is precisely the Palm spine $\tau_{k-1}$ of Definition~\ref{def:cycle_glues} built from the sequence $(L_{i+j})_{i\in\Z}$ instead of $(L_i)_{i\in\Z}$. Moreover, $\tau_{j+1}-\tau_j=L_{j+1}$. Comparing \eqref{eq:palm_origin_moved} with \eqref{eq:palm}, we see that the pair consisting of the random measure $\NPalm(\cdot+\tau_j)$ and the random variable $\tau_{j+1}-\tau_j$ is the same measurable function of the sequence $(C_{k+j})_{k\in\Z}$ as the pair consisting of $\NPalm$ and $\tau_1$ is of the sequence $(C_k)_{k\in\Z}$. Both sequences are i.i.d.\ with marginal law equal to the law of $C$; hence they have the same law, and therefore, for every $j\in\Z$,
\begin{equation}\label{eq:one_cycle}
\begin{split}
\E\biggl[\int_{0}^{\tau_{j+1}-\tau_j} f(\NPalm(\cdot+\tau_j+s))\d s\biggr]
&=\E\biggl[\int_{0}^{\tau_1} f(\NPalm(\cdot+s))\d s\biggr]\\
&=m\,\E[f(\Nstat)],
\end{split}
\end{equation}
where the last equality is \eqref{eq:stat}.

\noindent\emph{Step 2b (moving the origin across $n$ cycles).} Let $n\in\N$. Since $\tau_n<\infty$ a.s., we may split $[0,\tau_n)=\bigcup_{j=0}^{n-1}[\tau_j,\tau_{j+1})$; substituting $r=\tau_j+s$ on the $j$-th piece and applying \eqref{eq:one_cycle} to each piece, we obtain
\begin{equation}\begin{split}\label{eq:n_cycles}
\E\biggl[\int_{0}^{\tau_n} f(\NPalm(\cdot+r))\d r\biggr]
&=\sum_{j=0}^{n-1}\E\biggl[\int_{0}^{\tau_{j+1}-\tau_j} f(\NPalm(\cdot+\tau_j+s))\d s\biggr]
\\
&=n\,m\,\E[f(\Nstat)].
\end{split}
\end{equation}
\noindent\emph{Step 2c (comparison with a fixed shift).} Fix $t\in\R$ and first let $f$ be bounded, say $0\leq f\leq c$ for some constant $c<\infty$. Then all expectations below are finite, since $\E[\tau_n]=nm<\infty$. The functional $\kappa\mapsto f(\kappa(\cdot+t))$ is again nonnegative, measurable, and bounded by $c$, so that we may apply \eqref{eq:n_cycles} to it. This yields, after the substitution $r\mapsto r-t$,
\begin{align*}
n\,m\,\E[f(\Nstat(\cdot+t))]
&=\E\biggl[\int_{0}^{\tau_n} f(\NPalm(\cdot+r+t))\d r\biggr]
\\
&=\E\biggl[\int_{t}^{\tau_n+t} f(\NPalm(\cdot+r))\d r\biggr].
\end{align*}
Subtracting \eqref{eq:n_cycles}, we find
\begin{align*}
&n\,m\,|\E[f(\Nstat(\cdot+t))]-\E[f(\Nstat)]|
\\
&\leq \E\biggl[\int_{[t,\tau_n+t]\setminus[0,\tau_n]} f(\NPalm(\cdot+r))\d r\biggr]
+\E\biggl[\int_{[0,\tau_n]\setminus[t,\tau_n+t]} f(\NPalm(\cdot+r))\d r\biggr]
\\
&\leq 2|t|\,c,
\end{align*}
because the intervals $[t,\tau_n+t]$ and $[0,\tau_n]$ have the same length
$\tau_n$ and are shifted relative to each other by $t$, so that pathwise
\[
\bigl|[t,\tau_n+t]\setminus[0,\tau_n]\bigr|
=\bigl|[0,\tau_n]\setminus[t,\tau_n+t]\bigr|
=\min\{\tau_n,|t|\}\leq |t| .
\]
(For $t>0$ the first set equals $(\tau_n,\tau_n+t]$ if $t\leq\tau_n$ and
equals $[t,\tau_n+t]$ if $t>\tau_n$; the case $t<0$ is symmetric.)
  The left-hand side does not depend on $n$ except through the factor $n$. Dividing by $nm>0$ and letting $n\to\infty$ gives
\begin{equation*}
\E[f(\Nstat(\cdot+t))]=\E[f(\Nstat)]
\end{equation*}
for every bounded nonnegative measurable $f$. Taking for $f$ the indicator functions of measurable sets of locally finite counting measures shows that the laws of $\Nstat(\cdot+t)$ and $\Nstat(\cdot)$ coincide. As $t\in\R$ was arbitrary, $\Nstat$ is stationary. (For general nonnegative measurable $f$ the identity follows by applying it to $\min\{f,c\}$ and letting $c\to\infty$.)
\end{proof}

\begin{proof}[Proof of Theorem~\ref{thm:int}]
Clearly, $\R \ni t \mapsto\lstat(t)$ is left-continuous and adapted to $(\Fstat_t)_{t\in\R}$,
hence predictable. We prove that
\[
  \E[\Nstat((a,b])\,\mathbf 1_A]
  =\E\biggl[\int_a^b\lstat(t)\,\d t\,\mathbf 1_A\biggr]
  \qquad\text{for all }A\in\Fstat_a, \quad a < b.
\]
By stationarity of $\Nstat$, it suffices to show the equality for $a=0$. Let $\Lambda(t;\cdot)$ denote the
\emph{intensity functional}
\begin{equation*}
\Lambda(t;c):=\exp\biggl\{\nu+\int_{(-\infty,t)}h(t-s)\,c(\d s)\biggr\},\quad t\in\R,
\end{equation*}
defined for locally finite counting measures $c$ on $\R$. In particular,
$\Lambda(t;\Nstat)=\lstat(t)$, and we write $\lambda^{(\mathrm{Palm})}(t):=\Lambda(t;N^{(\mathrm{Palm})})$. Furthermore,  $ 1_A  $ has the form $\mathbf 1_\Gamma(\Nstat|_{(-\infty,0]})$
for a measurable set $\Gamma$ of past configurations. Thus, we must show
\begin{equation}\label{eq:star}
  \E[\Nstat((0,b])\,\mathbf  1_\Gamma(\Nstat|_{(-\infty,0]})]
  =\E\biggl[\int_0^b\Lambda(t;\Nstat)\,\d t\,\mathbf  1_\Gamma(\Nstat|_{(-\infty,0]})\biggr]
\end{equation}
for all measurable sets $\Gamma$ of past configurations.
Both sides of \eqref{eq:star} are expectations of functionals of $\Nstat$. Indeed, setting
\begin{equation*}
f_1(c):=c((0,b])\,\mathbf 1_\Gamma(c|_{(-\infty,0]}),
\qquad
f_2(c):=\int_0^b\Lambda(t;c)\,\d t\;\mathbf 1_\Gamma(c|_{(-\infty,0]}),
\end{equation*}
which are both nonnegative and measurable, \eqref{eq:star} reads $\E[f_1(\Nstat)]=\E[f_2(\Nstat)]$. Applying \eqref{eq:stat} to $f_1$ and
to $f_2$, and using $\tau_1 = L_1$, turns \eqref{eq:star} into the equivalent identity
\begin{equation}\label{eq:star_palm}
\frac1m\,\E\biggl[\int_0^{L_1}f_1(N^{(\mathrm{Palm})}(\cdot+r))\,\d r\biggr]
=\frac1m\,\E\biggl[\int_0^{L_1}f_2(N^{(\mathrm{Palm})}(\cdot+r))\,\d r\biggr].
\end{equation}
Set
$$
A_r:=\bigl\{N^{(\mathrm{Palm})}(\cdot+r)\big|_{(-\infty,0]}\in\Gamma\bigr\}
\in\mathcal F^{(\mathrm{Palm})}_r:=\sigma(N^{(\mathrm{Palm})}|_{(-\infty,r]}).
$$
As
$N^{(\mathrm{Palm})}(\cdot+r)((0,b])=N^{(\mathrm{Palm})}((r,r+b])$
we have $\mathbf 1_\Gamma(N^{(\mathrm{Palm})}(\cdot+r)|_{(-\infty,0]})=\mathbf 1_{A_r}$. Furthermore, $\Lambda(t;N^{(\mathrm{Palm})}(\cdot+r))=\lambda^{(\mathrm{Palm})}(t+r)$.
Thus, the
integrands in \eqref{eq:star_palm} become
\begin{align*}
f_1(N^{(\mathrm{Palm})}(\cdot+r))&=\mathbf 1_{A_r}\,N^{(\mathrm{Palm})}((r,r+b]),\\
f_2(N^{(\mathrm{Palm})}(\cdot+r))&=\mathbf 1_{A_r}\int_r^{r+b}\lambda^{(\mathrm{Palm})}(u)\,\d u .
\end{align*}
Cancelling the factor $1/m\ (\neq 0)$, we see that \eqref{eq:star} is equivalent to
\begin{equation}\label{eq:starstar}
  \E\biggl[\int_0^{L_1}\mathbf 1_{A_r}\,N^{(\mathrm{Palm})}((r,r+b])\,\d r\biggr]
  =\E\biggl[\int_0^{L_1}\mathbf 1_{A_r}\!\int_r^{r+b}\!\lambda^{(\mathrm{Palm})}(u)\,\d u\,\d r\biggr].
\end{equation}

\smallskip
Recall that $L_1=\inf\{t\geq s_0:N^{(\mathrm{Palm})}((t-s_0,t])=0\}$. Hence, $\{r<L_1\}\in\mathcal F^{(\mathrm{Palm})}_r$ and
$B_r:=A_r\cap\{r<L_1\}\in\mathcal F^{(\mathrm{Palm})}_r$. Writing
$\int_0^{L_1}\d r=\int_0^\infty\mathbf 1\{r<L_1\}\d r$, the right-hand side of \eqref{eq:starstar} reads
\begin{align*}
  \E\biggl[\int_0^\infty \mathbf 1_{B_r}\!\int_r^{r+b}\!\lambda^{(\mathrm{Palm})}(u)\,\d u\,\d r\biggr]
&  =\int_0^\infty \E\biggl[\mathbf 1_{B_r}\!\int_r^{r+b}\!\lambda^{(\mathrm{Palm})}(u)\,\d u\biggr]\,\d r .
\end{align*}
Equation \eqref{eq:starstar}
will follow from
\begin{equation}\label{eq:A}
  \E[\mathbf 1_B\,N^{(\mathrm{Palm})}((r,r+b])]
  =\E\biggl[\mathbf 1_B\int_r^{r+b}\lambda^{(\mathrm{Palm})}(u)\,\d u\biggr],\quad r\geq 0, B\in\mathcal F^{(\mathrm{Palm})}_r.
\end{equation}
To prove \eqref{eq:A}, we first split $N^{(\mathrm{Palm})}$ into its forward and backward parts,
\[
  N^{(\mathrm{Palm})}_{>0}:=N^{(\mathrm{Palm})}|_{(0,\infty)},\qquad
  N^{(\mathrm{Palm})}_{\leq0}:=N^{(\mathrm{Palm})}|_{(-\infty,0]},
\]
and write $\mathcal F^{N^{(\mathrm{Palm})}_{>0}}_r=\sigma(N^{(\mathrm{Palm})}_{>0}|_{(0,r]})$, so that
$\mathcal F^{(\mathrm{Palm})}_r=\sigma(N^{(\mathrm{Palm})}_{\leq0})\vee\mathcal F^{N^{(\mathrm{Palm})}_{>0}}_r$, $r >0$. By construction, $N^{(\mathrm{Palm})}$ has an empty $s_0$-window ending at $0$, i.e.,
$N^{(\mathrm{Palm})}((-s_0,0])=0$, and forward of it $N^{(\mathrm{Palm})}$ is the glue of i.i.d.\ generic cycles. Hence, $N^{(\mathrm{Palm})}_{>0}$ is the loglinear Hawkes process started from
the void configuration at $0$, with  $(\mathcal F^{N^{(\mathrm{Palm})}_{>0}}_r)_{r>0}$-stochastic intensity
$\lambda^{(\mathrm{Palm})}(r)= \Lambda(r;N^{(\mathrm{Palm})}_{>0})$. By \eqref{eq:intensity}, we obtain,
for every $D\in\mathcal F^{N^{(\mathrm{Palm})}_{>0}}_r$,
\begin{equation}\label{eq:known}
  \E[\mathbf 1_{D}N^{(\mathrm{Palm})}((r,r+b])]
  =  \E\biggl[\mathbf 1_{D}\int_r^{r+b}\lambda^{(\mathrm{Palm})}(u)\,\d u\biggr],
\end{equation}
which is \eqref{eq:A} for $D\in\mathcal F^{N^{(\mathrm{Palm})}_{>0}}_r$.
Since $h$ is supported in $(0,s_0]$ and $N^{(\mathrm{Palm})}((-s_0,0])=0$, no point at a
time $\leq 0$ can contribute to $\lambda^{(\mathrm{Palm})}(u)$ for $u>0$. Consequently,
\begin{equation}\label{eq:Vmeas}
  N^{(\mathrm{Palm})}((r,r+b])\ \text{and}\ \int_r^{r+b}\lambda^{(\mathrm{Palm})}(u)\,\d u
  \quad\text{are } \sigma(N^{(\mathrm{Palm})}_{>0})\text{-measurable.}
\end{equation}
Moreover, the i.i.d.\ cycle construction makes the forward and backward parts
independent:
\begin{equation}\label{eq:indep}
  N^{(\mathrm{Palm})}_{>0}\ \perp\ N^{(\mathrm{Palm})}_{\leq0}.
\end{equation}
Finally, fix $B\in\mathcal F^{(\mathrm{Palm})}_r=\sigma(N^{(\mathrm{Palm})}_{\leq 0})\vee\mathcal F^{N^{(\mathrm{Palm})}_{>0}}_r$. Since
$\mathbf 1_B$ is a measurable function of $(N^{(\mathrm{Palm})}_{\leq 0},\,N^{(\mathrm{Palm})}_{>0}|_{(0,r]})$,
fixing $N^{(\mathrm{Palm})}_{\leq 0}=w$ leaves an event $B_w\in\mathcal F^{N^{(\mathrm{Palm})}_{>0}}_r$ with
$\mathbf 1_B=\mathbf 1_{B_w}$ on $\{N^{(\mathrm{Palm})}_{\leq 0}=w\}$.
Conditioning on $\{N^{(\mathrm{Palm})}_{\leq 0} = w\}$,
\begin{align*}
  &\E\biggl[\mathbf 1_B\biggl(N^{(\mathrm{Palm})}((r,r+b])-\int_r^{r+b}\lambda^{(\mathrm{Palm})}(u)\,\d u\biggr)
       \,\bigg|\,N^{(\mathrm{Palm})}_{\leq 0}=w\biggr]\\
  &   =  \E\biggl[\mathbf 1_{B_w}\biggl(N^{(\mathrm{Palm})}((r,r+b])-\int_r^{r+b}\lambda^{(\mathrm{Palm})}(u)\,\d u\biggr)
       \,\bigg|\,N^{(\mathrm{Palm})}_{\leq 0}=w\biggr]\\
 & = \E\biggl[\mathbf 1_{B_w}\biggl(N^{(\mathrm{Palm})}((r,r+b])-\int_r^{r+b}\lambda^{(\mathrm{Palm})}(u)\,\d u\biggr)\biggr]
  = 0
\end{align*}
for $\P_{N^{(\mathrm{Palm})}_{\leq 0}}$-almost all $w$. The second equality follows from \eqref{eq:indep} and
\eqref{eq:Vmeas}, and the last from \eqref{eq:known}. Integrating over the law of
$N^{(\mathrm{Palm})}_{\leq 0}$ yields
\eqref{eq:A}. This proves \eqref{eq:star} and the theorem.
\end{proof}

\begin{lemma}[Spread-out cycle length]\label{lemma:spread_out} Let $N$ be a loglinear Hawkes process with enveloped memory and log-baseline intensity $\nu\in\R$. Assume \eqref{eq:rho_bound} and let $L$ be the generic cycle length of $N$. Then the distribution $\P_L$ of $L$ satisfies
\begin{enumerate}[label = (\roman*)]
\item $\P_L(B) = 0$, $B\in\mathcal{B}((-\infty,s_0))$, i.e., $\P_L$ is concentrated on $[s_0,\infty)$,
\item $\P_L(\{s_0\}) = \exp\{-s_0e^\nu\}\in(0,1)$,
\item the restriction of $\P_L$ to $(s_0,\infty)$ is absolutely continuous with respect to
Lebesgue measure and of total mass $1-\exp\{-s_0e^\nu\}>0$.
\end{enumerate}
In particular, $\P_L$ is spread out in the sense of \cite[p.~98]{thorisson00}.
\end{lemma}

\begin{proof}[Proof of Lemma~\ref{lemma:spread_out}]
We first record two facts that will be used repeatedly.

\noindent\emph{(a) $C$ is concentrated on $(0,L]$, with $L<\infty$ and $C(\R)<\infty$.}
The generic cycle has a void initial condition, $C(\R_{\leq0})=0$, and the indicator $1\{t\leq L\}$ in
\eqref{eq:C_intensity} makes the stochastic intensity vanish on $(L,\infty)$, whence $C((L,\infty))=0$ and
$\supp C\subseteq(0,L]$. Moreover, $\E[L] = m<\infty$ as observed after Definition~\ref{def:cycle}, so
$L<\infty$, and by nonexplosion, $C(\R) = C((0,L])<\infty$.

\noindent \emph{(b) The infimum defining $L$ is attained.}
Write $N_C(u):=C((-\infty,u])$ for the counting process corresponding to $C$, so that
$t\mapsto C((t-s_0,t]) = N_C(t)-N_C(t-s_0)$ is piecewise constant and right-continuous. By (a) the set
$\{t\geq s_0:\ C((t-s_0,t])=0\}$ is nonempty, and right-continuity turns its infimum into a
minimum. Hence
\begin{equation}\label{eq:L_attained}
C((L-s_0,L]) = 0.
\end{equation}
Now we prove (i)--(iii).\\
\noindent(i) follows immediately from $L=\inf\{t\geq s_0:\ C((t-s_0,t])=0\}\geq s_0$.\\
(ii) By \eqref{eq:L_attained}, $L=s_0$ implies $C((0,s_0])=0$, and conversely $C((0,s_0])=0$ implies
$L=s_0$ by the definition of $L$. Thus, $\{L=s_0\} = \{C((0,s_0])=0\}$. On $(0,s_0]$ and before the first
point of $C$, the memory integral in \eqref{eq:C_intensity} vanishes, again because $C(\R_{\leq0})=0$, and
$1\{t\leq L\}=1$ because $L\geq s_0$. So the stochastic intensity equals the constant $e^\nu$ there, the
first point of $C$ is $\mathrm{Exp}(e^\nu)$-distributed, and
$\P_L(\{s_0\}) = \P(L = s_0) = \P(C((0,s_0])=0) = \exp\{-s_0e^\nu\}\in(0,1)$.\\
\noindent(iii) We first claim that, if $L>s_0$, the point $L-s_0$ belongs to the support of $C$.
Indeed, if
$L>s_0$, we have $C((0,s_0])\geq1$, so $C$ has at least one point, and as $C(\R)<\infty$ by (a), the support $\supp C$ is nonempty and finite, so that
$t_{\max}:=\max\supp C$ is well-defined. By \eqref{eq:L_attained} together with $C((L,\infty))=0$ from (a), this last point satisfies $t_{\max}\leq L-s_0<\infty$. If $t_{\max} < L-s_0$, then
$t_{\max}+s_0\in[s_0,L)$ would satisfy $C((t_{\max},t_{\max}+s_0])=0$, contradicting the
minimality of $L$. Hence $t_{\max}=L-s_0$, proving the claim.

Now let $A\subset(s_0,\infty)$ be a Borel set with Lebesgue measure zero. From the claim above,
$\{L\in A\}\subset\{C(A-s_0)\geq1\}$. Clearly, $A-s_0\subset(0,\infty)$ is again a Lebesgue
null set and
$$
\P_L(A)\ \leq\ \P(C(A-s_0)\geq1)\ \leq\ \E[C(A-s_0)]
= \E\biggl[\int_{A-s_0}\lambda_C(t)\,\d t\biggr] = 0 ,
$$
where the last equality holds because the integral of the nonnegative function $\lambda_C$ over the Lebesgue null set
$A-s_0$ vanishes.
Thus, the restriction of $\P_L$ to $(s_0,\infty)$ is absolutely continuous with respect to Lebesgue measure. Hence $\P_L$ has a nonzero absolutely continuous component, and by the Radon--Nikodym theorem there is a
Borel measurable $f_{s_0}\geq 0$ with
$\P_L(A) = \int_A f_{s_0}(t)\, \d t$, $A\in\mathcal{B}((s_0, \infty))$.
Extending $f_{s_0}$ by zero outside $(s_0,\infty)$ and combining with (i) and (ii) yields
\begin{equation*}
\P_L(\d t) = \exp\{-s_0e^\nu\}\delta_{s_0}(\d t) + f_{s_0}(t)\,\d t,
\qquad
\int_\R f_{s_0}(t)\,\d t = 1-\exp\{-s_0e^\nu\}>0.
\end{equation*}
 Consequently,
$$
\P_L(A)\ \geq\ \P_L(A\cap(s_0,\infty))\ =\ \int_A f_{s_0}(t)\, \d t, \quad A\in\mathcal{B}(\R),
$$
which is the spread-out property of \cite[p.~98]{thorisson00} with $r= 1$.
\end{proof}

\begin{proof}[Proof of Theorem~\ref{thm:weak}]
\noindent\emph{Step 1 (regenerative representation of $N$).} Our assumptions allow the application of Theorem~\ref{thm:regeneration_time}. Thus, the regeneration time $\tau = \inf\{t \geq s_0: N((t-s_0,t]) = 0\}$ satisfies $\E[\tau]<\infty$. The window $(\tau - s_0, \tau]$ contains no point of $N$, and since $h$ vanishes outside
$(0, s_0]$, the stochastic intensity on $(\tau,\infty)$ involves only points of $N$ in
$(\tau, \infty)$. Recall that $N$ is obtained by thinning from the driving Poisson random measure
$\mathcal{N}$. Thus, the process $N(\cdot + \tau)|_{(0,\infty)}$ is distributed as
$\sum_{k\in\N} C_k(\cdot - \tau_{k-1})$, the glue of i.i.d.\ generic cycles, and is
independent of $\F^N_\tau$. In other words, $N|_{(0,\infty)}$ is a \emph{delayed
regenerative process}: after a delay
$\gamma_1 := \tau$ the process $N$ consists of shifted versions of i.i.d.\ cycles $(C_k')_{k\in\N}$ with the same law as the $(C_k)_{k\in\N}$ of Definition~\ref{def:cycle_glues} with regeneration times $\gamma_{k+1} := \gamma_{k} + L_k'$, $k\in\N$, where $L_k':=\ell(C_k')$ are the cycle lengths. Note that $\gamma_1, L_1', L_2', \dots$ are independent. That is,
$(\gamma_k)_{k\in\N} \eqd (\tau + \tau_{k-1})_{k\in\N}$ with $(\tau_k)_{k\in\N_0}$ the forward part of the Palm spine.

\noindent\emph{Step 2 (regenerative representation of $\Nstat$).}
By Theorem~\ref{thm:stationary_construction} and \eqref{eq:stat_version},
$\Nstat|_{(0,\infty)}$ has the same structure: after a delay
$\gammastat_1 := \taustat_1 = (1-U)\widetilde L_1\in(0,\widetilde L_1)$ the process $\Nstat|_{(0,\infty)}$ consists of shifted versions of the cycles $(C_{k+1})_{k\in\N}$ with regeneration times $\gammastat_{k} := \taustat_{k}$, $k\in\N$, so that $\gammastat_{k+1} = \gammastat_{k} + L_{k+1}$, where $L_{k+1} = \ell(C_{k+1})$ are the cycle lengths. Note that $\gammastat_1, L_2, L_3, \dots$ are independent, and that $(\gammastat_k)_{k\in\N}$ is a delayed renewal sequence with the same increment law as $(\gamma_k)_{k\in\N}$.

\noindent\emph{Step 3 (coupling).}
Both delays $\gamma_1=\tau$ and $\gammastat_1$ are finite, and the common
cycle-length distribution $\P_L$ is spread out by Lemma~\ref{lemma:spread_out}. Thus, we may apply \cite[Chapter~3, Theorem~6.1]{thorisson00}
and obtain the existence of two copies of the two regeneration time sequences
$(\gamma_k)_{k\in\N}$ and $(\gammastat_k)_{k\in\N}$ on a common probability space together with a finite random index
$\kappa\in\N$
such that the two sequences meet at the common regeneration time
$\sigma := \gamma_{\kappa} = \gammastat_{\kappa} < \infty$.

It remains to paste a common future onto this shared regeneration time. By Steps 1 and 2, the part of
$N$ after $\gamma_{\kappa}$ and the part of $\Nstat$ after $\gammastat_{\kappa}$ are, in
either case, the glue of an i.i.d.\ sequence of generic cycles, independent of the respective
past and with the same law. We may therefore let both processes use one and the same
i.i.d.\ cycle sequence beyond $\sigma$, which leaves the two marginal laws unchanged. Note
that no conditioning on cycle lengths is involved: a cycle sequence is drawn first, and
the lengths are the deterministic functionals $\ell(\cdot)$ of its members. This yields
versions $\hat N$ and $\Nstathat$ of $N|_{(0,\infty)}$ and $\Nstat|_{(0,\infty)}$ with
\begin{equation}\label{eq:couple}
  \hat N|_{(\sigma,\infty)}=\Nstathat|_{(\sigma,\infty)},\qquad\sigma<\infty.
\end{equation}

\noindent\emph{Step 4 (coupling inequality).}
Let $A$ be a measurable set of locally finite counting measures on $(0,\infty)$. Using stationarity of $\Nstat$
(so that $(\Nstathat(\cdot + t))|_{(0,\infty)} \eqd \Nstat|_{(0,\infty)}$),
\begin{align*}
&\bigl|\P((N(\cdot + t))|_{(0,\infty)}\in A) - \P(\Nstat|_{(0,\infty)}\in A)\bigr|\\
&= \bigl|\E\bigl[\mathbf{1}\big\{(\hat N(\cdot + t))|_{(0,\infty)}\in A\big\}- \mathbf{1}\big\{( \Nstathat(\cdot + t))|_{(0,\infty)}\in A\big\}\bigr]\bigr|\leq 2\P(\sigma > t),
\end{align*}
since on $\{\sigma\leq t\}$ the two indicators coincide by \eqref{eq:couple}. Taking the
supremum over $A$ and letting $t\to\infty$ yields \eqref{eq:TV_convergence}.

\noindent\emph{Step 5 (weak convergence).}
Fix a continuous function $f\geq 0$ with $\supp f\subseteq[-K,K]$, $K > 0$, and set
$g := f(\cdot - K)$, so that $\supp g\subseteq[0,2K]$ and, for $t\geq K$,
$\int f(s)\, N(\d s + t) = \int g(s)\, N(\d s + t - K)$, while
$\int f(s)\, \Nstat(\d s) \eqd \int g(s)\, \Nstat(\d s)$ by stationarity. Both integrals are
measurable functions of the restrictions to $(0,\infty)$ {(the open interval, since $g(0)=f(-K)=0$ by continuity of $f$)}, so that
\begin{align}
&\biggl|\E\biggl[\exp\biggl\{-\int f(s)\, N(\d s + t)\biggr\}\biggr] - \E\biggl[\exp\biggl\{-\int f(s)\, \Nstat(\d s)\biggr\}\biggr]\biggr|\nonumber\\
&\qquad\leq \Big\|\L((N(\cdot + t - K))|_{(0,\infty)})
- \L(\Nstat|_{(0,\infty)})\Big\|_{\mathrm{TV}}
\xrightarrow[t\to\infty]{} 0.\label{eq:tv_bound}
\end{align}
Pointwise convergence of the Laplace functionals over continuous, nonnegative $f$ with
compact support is equivalent to weak convergence of point processes
\cite[Proposition 11.1.VIII]{daley03}, i.e., $N(\cdot + t)\wto\Nstat$ as $t\to\infty$. This fact together with \eqref{eq:tv_bound} completes the proof.
\end{proof}

\begin{proof}[Proof of Proposition~\ref{prop:necessary_stability_condition}]
By the stability assumption, there exists a stationary loglinear Hawkes process $\Nstat $ with finite average intensity $e^c = \E[\lambdastat (0)] $ for some $c$.
We see from Jensen's inequality that
\begin{align*}
e^c =  \E[\lambdastat (0)]& \geq \exp\biggl\{\nu  + \E\biggl[\int_{(-\infty, 0)}h(0- s) \Nstat (\d s)\biggr]\biggr\}\\
& = \exp\biggl\{\nu  + e^c \int_{0}^\infty h(s)\d s\biggr\}.
\end{align*}
Taking logarithms yields the first inequality in \eqref{eq:ineq1}. Furthermore, the right-hand side $(c-\nu)e^{-c}$ attains its maximum at $c = 1 + \nu$,
which gives the second inequality.
 \end{proof}

\begin{proof}[Proof of Theorem~\ref{thm:exponential_moment_bound}]
Fix $\theta\in(0,\theta_0]$ throughout and let $I^{(+)}_n= ((n - p - 1)\Delta , (n - p^{(-)})\Delta ]$ be the interval containing
all points that can excite the intensity of $N$ in the bin $((n-1)\Delta, n\Delta]$. From Theorem~\ref{thm:bin_count_domination}, we have
\begin{equation*}
N_n
\leq  \widetilde  {N}_n:=
\sum_{k=1}^\infty \mathbf{1}\{S_{n,1}+\cdots+S_{n,k} < \Delta\},\quad n\in\N,
\end{equation*}
where
$S_{n,k} =  \exp\{-(\nu + N(I^{(+)}_n)\, b - (k-1)a)\}E_{n,k}$,
and
$(E_{n,k})_{k\in\N}$ are i.i.d.\ $\mathrm{Exp}(1)$ random variables.
Next, we decompose the interval
$I^{(+)}_n$ into
$(p^{(+)}+1)$ blocks of size $\Delta$, and denote the total positive contribution of the past to the $n$-th bin by
 \begin{equation*}
W_n := N(I^{(+)}_n)
=
\sum_{j=0}^{p^{(+)}}
N(((n - p^{(-)} - j - 1)\Delta,(n - p^{(-)} - j)\Delta])=
\sum_{j=0}^{p^{(+)}} N_{n-p^{(-)}-j}.
\end{equation*}
Using for  $j\in\N_0$ the identity
\begin{align*}
e^{\theta j} = 1 + \sum_{k=1}^j (e^{\theta k} - e^{\theta (k-1)})
&= 1 + \sum_{k=1}^j (e^\theta - 1)e^{\theta (k-1)} \\
&= 1 + \sum_{k=1}^\infty(e^\theta - 1) e^{\theta (k-1)}\,\mathbf{1}\{j \geq k\},
\end{align*}
together with monotone convergence, we obtain
\begin{align*}
\E[e^{\theta \widetilde  {N}_n}\mid \mathcal{F}^N_{(n-1)\Delta}]&=
1 + \sum_{k=1}^\infty (e^\theta -1)e^{\theta (k-1)}
\P(\widetilde  {N}_n \geq k| \mathcal{F}^N_{(n-1)\Delta})\\
&=1 + \sum_{k=1}^\infty (e^\theta -1)e^{\theta (k-1)}\P(S_{n,1}+\cdots+S_{n,k} < \Delta |  \mathcal{F}^N_{(n-1)\Delta})\\
&\le 1 + \sum_{k=1}^\infty (e^\theta -1)e^{\theta (k-1)}
\P(S_{n,k} < \Delta| \mathcal{F}^N_{(n-1)\Delta})\\
&=1 + \sum_{k=1}^\infty (e^\theta -1)e^{\theta (k-1)}
(1 - \exp\{-\Delta e^{\nu + W_n b - (k-1)a}\})\\
&= 1 + \sum_{k=0}^\infty (e^\theta -1)e^{\theta k} (1 - \exp\{-\Delta e^{\nu + W_n b - ka}\}).
\end{align*}
In order to obtain a more convenient dependence on the past, we split the sum at the random
index
\begin{equation*}
K_n := \biggl\lfloor \frac{b}{a} W_n \biggr\rfloor,
\end{equation*}
so that
\begin{align*}
\E[e^{\theta \widetilde  {N}_n}\mid \mathcal{F}^N_{(n-1)\Delta}]&\leq
I_1 + I_2,
\end{align*}
where
\begin{align*}
I_1 &:= 1 + (e^\theta-1)\sum_{k=0}^{K_n} e^{\theta k}
(1 - \exp\{-\Delta e^{\nu + W_n b - ka}\}),\\
I_2 & := (e^\theta-1)\sum_{k=K_n+1}^{\infty} e^{\theta k}
(1 - \exp\{-\Delta e^{\nu + W_n b - ka}\}).
\end{align*}
Since $1 - e^{-x} \le 1$, we have
\begin{align}\label{eq:I_1}
I_1 \leq e^{\theta(K_n+1)}.
\end{align}
To bound $I_2$, we use $1 - e^{-x} \le x$ and obtain
\begin{align*}
I_2 \leq (e^\theta-1)\sum_{k=K_n+1}^{\infty} e^{\theta k}
\Delta e^{\nu + W_n b - ka} = (e^\theta-1)\Delta e^\nu
\sum_{l=0}^{\infty}
e^{\theta(K_n+1+l)}
e^{W_n b - (K_n+1+l)a}.
\end{align*}
Since
$$
(K_n+1)a=  \biggl(\biggl\lfloor \frac{b}{a} W_n \biggr\rfloor + 1\biggr)a \geq  \biggl\lceil \frac{b}{a} W_n \biggr\rceil a \geq  b W_n,
$$ we have
$W_n b - (K_n+1+l)a \le W_n b - b W_n - la  =  - la$,
so that, using the assumption $\theta \leq \theta_0 < a$, we conclude that
\begin{align}\label{eq:I_2}
I_2 \leq
(e^\theta-1)\Delta e^\nu e^{\theta(K_n+1)}
\sum_{l=0}^{\infty} e^{l(\theta-a)}=e^{\theta(K_n+1)}
\frac{(e^\theta-1)e^\nu \Delta}{1 - e^{\theta-a}}.
\end{align}
 Finally, using \eqref{eq:I_1}, \eqref{eq:I_2} and
$
K_n+1=  \bigl\lfloor \frac{b}{a} W_n \bigr\rfloor + 1 \le \frac{b}{a} W_n + 1,
$, we conclude that, for all $n\in\N$,
\begin{equation}\label{eq:exp_moment_thing}
\E[e^{\theta N_n}\mid \mathcal{F}^N_{(n-1)\Delta}]
\leq \E[e^{\theta \widetilde  {N}_n}|\mathcal{F}^N_{(n-1)\Delta}]
\le \exp\biggl\{
\theta\frac{b}{a}
\,W_n
\biggr\}e^\theta
\biggl(
1 + \frac{(e^\theta-1)e^\nu \Delta}{1 - e^{\theta-a}}
\biggr).
\end{equation}
Define the function $c: (0,\theta_0]\to \R_{\geq 0}$ by
\begin{equation*}
c(\theta)
:=
e^\theta
\biggl(
1 + \frac{(e^\theta-1)e^\nu \Delta}{1 - e^{\theta-a}}
\biggr).
\end{equation*}
Taking expectations in \eqref{eq:exp_moment_thing} gives
\begin{equation}
\E[e^{\theta N_n}]
\le
c(\theta)\,
\E\biggl[
\exp\biggl\{
\theta\frac{b}{a} W_n
\biggr\}
\biggr],\quad n\in\N.\label{eq:first_bound}
\end{equation}
Recalling that $W_n =
\sum_{j=0}^{p^{(+)}} N_{n-p^{(-)}-j}$
and applying H\"older's inequality yields
\begin{align*}
\E\biggl[
\exp\biggl\{
\theta\frac{b}{a} W_n
\biggr\}
\biggr] &=
\E\biggl[
\prod_{j=0}^{p^{(+)}}
\exp\biggl\{\theta\frac{b}{a} N_{n-p^{(-)}-j}\biggr\}
\biggr]\\
&\le
\biggl(
\prod_{j=0}^{p^{(+)}}
\E\biggl[
\exp\biggl\{
\theta\frac{b}{a}(p^{(+)}+1)\, N_{n-p^{(-)}-j}
\biggr\}
\biggr]
\biggr)^{\frac{1}{p^{(+)}+1}}.
\end{align*}
Because $p^{(-)}\geq 1$ by assumption, all indices $n - p^{(-)} - j$, $j = 0,\dots,p^{(+)}$,
are at most $n-1$, and we obtain
\begin{equation}\label{eq:hoelder}
\E\biggl[
\exp\biggl\{
\theta\frac{b}{a} W_n
\biggr\}
\biggr]
\le
\sup_{i\leq n-1}
\E\biggl[
\exp\biggl\{
\theta\frac{b}{a}(p^{(+)}+1)\, N_i
\biggr\}
\biggr]
=
\sup_{i\leq n-1}
\E[
e^{\rho \theta N_i}],
\end{equation}
recalling $\rho = \frac{b}{a}(p^{(+)}+1)$. Plugging
\eqref{eq:hoelder} into \eqref{eq:first_bound}, we obtain
\begin{equation}
\E[e^{\theta N_n}]
\le
c(\theta)\,
\sup_{i\leq n-1}
\E[
e^{\rho \theta N_i}],\quad n\in \N,\ \theta\in(0,\theta_0].\label{eq:bound}
\end{equation}
Now, for $ \theta\in(0,\theta_0]$, we define $M_0(\theta) := 1$ and, for $n\in\N$,
\begin{equation*}
M_n(\theta) := \max_{1\leq j\leq n}\E[e^{\theta N_j}].
\end{equation*}
We show by induction that
\begin{equation}\label{eq:result}
M_n(\theta) \leq m(\rho\theta) c(\theta) M_{n-1}(\rho\theta),\quad n\in\N,\
\theta\in(0,\theta_0],
\end{equation}
where $m(\cdot)$ is the function satisfying \eqref{eq:initial_condition}. Note that
$\rho\theta\leq\theta_0$, so all evaluations of $m$ below are legitimate.
For $n = 1$, inequality \eqref{eq:result} follows from \eqref{eq:bound} together with assumption \eqref{eq:initial_condition}:\begin{equation*}
M_1(\theta) = \E[e^{\theta N_1}]
\le c(\theta)\,
\sup_{i\leq 0} \E[ e^{\rho \theta N_i}]
\leq   c(\theta)m(\rho \theta)
= m(\rho\theta)c(\theta)M_0(\rho\theta).
\end{equation*}
With the induction hypothesis that \eqref{eq:result} holds for $n$, we obtain from \eqref{eq:bound} the
induction step
\begin{align*}
M_{n+1}(\theta)
&=  \max\bigl\{M_n(\theta), \E[e^{\theta N_{n+1}}]\bigr\}
\leq  \max\bigl\{ m(\rho\theta)c(\theta)M_{n-1}(\rho \theta), c(\theta)\,\sup_{i\leq n}\E[e^{\rho \theta N_i}]\bigr\}\\
&= c(\theta) \max\bigl\{ m(\rho\theta) M_{n-1}(\rho \theta), \sup_{i\leq 0 }\E[e^{\rho \theta N_i}], M_{n}( \rho\theta)\bigr\}\\
&\leq c(\theta) \max\bigl\{ m(\rho\theta) M_{n-1}(\rho \theta), m(\rho\theta), m(\rho\theta) M_{n}( \rho\theta)\bigr\}\leq m(\rho\theta) c(\theta) M_{n}(\rho \theta).
\end{align*}
Thus, we have proven \eqref{eq:result}.
Iterating the inequality along $\theta,\ \rho\theta,\ \rho^2\theta,\dots$, we find for $n\in\N$
\begin{equation}\begin{split}
&M_n(\theta)
\leq c(\theta)m(\rho\theta)M_{n-1}(\rho \theta)
\leq c(\theta) c(\rho \theta)m(\rho\theta) m(\rho^2\theta) M_{n-2}(\rho^2 \theta)\\
&\leq \dots
 \leq   \Bigl(\prod_{k = 0}^{n-1} c(\rho^k\theta)\Bigr)\Bigl( \prod_{k = 1}^{n} m(\rho^k\theta)\Bigr) M_{0} (\rho^n\theta)=  \Bigl(\prod_{k = 0}^{n-1} c(\rho^k\theta)\Bigr)\Bigl( \prod_{k = 1}^{n} m(\rho^k\theta)\Bigr),\label{eq:product}
 \end{split}
\end{equation}
using $M_0\equiv 1$. Both products in \eqref{eq:product} are nondecreasing in $n$ because
$c(\rho^k\theta) \geq 1$ and $m(\rho^k\theta)\geq 1$. By \eqref{eq:m_condition}, the second product is bounded by
 $
   \prod_{k = 1}^{\infty} m(\rho^k\theta) = c_1(\theta) <\infty.
 $
 It remains to bound the first product. Set
$x_k := \rho^k \theta$. Note that $x_k \leq \theta$ and $x_k \to 0$ (as $\rho \in (0,1)$). Using
 $e^x - 1 \leq xe^x$ together with
${x_k-a} \leq \theta - a < 0$, and defining
\begin{equation*}
C := e^{\theta}\biggl(1 + \frac{ e^{\theta}e^\nu \Delta}{1 - e^{{\theta}-a}} \biggr),
\end{equation*}
we have
\begin{align*}
c(x_k) -1
&= e^{x_k}\biggl(1 + \frac{(e^{x_k}-1)e^\nu \Delta}{1 - e^{{x_k}-a}}\biggr) - 1
= (e^{x_k} - 1)+ \frac{(e^{x_k}-1)e^{x_k}e^\nu \Delta}{1 - e^{{x_k}-a}}\\
&\leq x_k e^{x_k} + \frac{x_k e^{2x_k}e^\nu \Delta}{1 - e^{{x_k}-a}}
\leq x_k e^{\theta} + \frac{x_k e^{2\theta}e^\nu \Delta}{1 - e^{{\theta}-a}} = C x_k.
\end{align*}
Hence, as $\log x \leq  x -1$ for $x>0$ and as $c(x_k) > 0$, we obtain
$\log c(x_k)  \leq c(x_k) - 1 \leq C x_k = C \theta \rho^k$.
Consequently,
\begin{equation*}
\prod_{k=0}^\infty c(\rho^k\theta)
= \exp\Bigl\{\sum_{k=0}^\infty \log c(\rho^k\theta)\Bigr\}
\leq \exp\biggl\{\frac{C\theta}{1-\rho}\biggr\}
= c_2(\theta)
\end{equation*}
with $c_2$ defined in \eqref{eq:c_2}.
Combining the two product bounds with \eqref{eq:product}, we obtain
\begin{equation}
\sup_{n\in\N} \E[e^{\theta N_n}]
= \sup_{n\in\N} M_n(\theta)
 \leq  c_1(\theta)c_2(\theta).\label{eq:onebinbound}
  \end{equation}
Thus, we have shown the exponential moment bound for the single bin counts $N_n$, $n\in\N$.

For intervals of length  $k\Delta$, we proceed as in \eqref{eq:hoelder} and apply Hölder's inequality. For
$n\geq k$ we have
 \begin{align*}
 &\E\biggl[\exp\biggl\{\frac{\theta}{k} N(((n-k)\Delta, n\Delta])\biggr\}\biggr]
 = \E\biggl[\exp\biggl\{\frac{\theta}{k} \sum_{l = 1}^{k} N_{n-k+l}\biggr\}\biggr]\\
 &\leq \Bigl(\prod_{l = 1}^{k} \E[e^{\theta N_{n-k+l}}]\Bigr)^{{1}/k}
 \leq \sup_{1\leq l\leq k}\E[e^{\theta N_{n-k+l}}]
   \leq  c_1(\theta) c_2(\theta),
 \end{align*}
 where the last inequality follows from \eqref{eq:onebinbound}, since all indices $n-k+l$,
 $l = 1,\dots,k$, belong to $\N$. Thus, we have proven \eqref{eq:exp_moment_bound}, and the proof of the theorem is complete.
\end{proof}

\begin{proof}[Proof of Theorem~\ref{thm:moment}]
Fix a bounded Borel set $B\subset\R$. Let $N$ be a loglinear Hawkes process with enveloped
memory and with a \emph{void} initial condition. Then
\eqref{eq:initial_condition} and \eqref{eq:m_condition} hold with $m\equiv 1$ and
$c_1\equiv 1$ for any $\theta_0\in(0,a)$. Fix such a $\theta_0$, set
$k := \lceil \operatorname{diam}(B)/\Delta\rceil + 2$ and $\theta_B := \theta_0/k$. For $t\geq 0$, put
$
n(t) := \max\big\{k,\ \lfloor(\sup B + t)/\Delta\rfloor + 1\big\}.
$ Then
$k\Delta\geq \operatorname{diam}(B) + 2\Delta$ implies $(B+t)\cap\R_{> 0}\subseteq ((n(t)-k)\Delta,\ n(t)\Delta]$, and, as $N$ has no
points in $(-\infty, 0] $,
$$
N(B+t)\leq N(((n(t)-k)\Delta,\ n(t)\Delta]),\quad t\geq 0.
$$
Since $n(t)\geq k$, Theorem~\ref{thm:exponential_moment_bound} yields
\begin{equation}\label{eq:sup_bound}
\sup_{t\ge0}\E[e^{\theta_B\,N(B+t)}]
\leq \sup_{n\geq k}\E\biggl[\exp\biggl\{\frac{\theta_0}{k}
N(((n-k)\Delta, n\Delta])\biggr\}\biggr]
\leq c_2(\theta_0) =: c < \infty,
\end{equation}
using $c_1(\theta_0) = 1$. Without loss of generality, we may assume $B\subset[0,\infty)$: otherwise, replace $B$
by $B - \inf B$, which changes neither $\E[e^{\theta_B\Nstat(B)}]$ (by stationarity
of $\Nstat$) nor the right-hand side of \eqref{eq:sup_bound} (whose supremum ranges
over all shifts $t\geq0$). With the shift notation introduced at the beginning of
Section~\ref{sec:stability}, the shifted process $N(\cdot + t)$ takes the value $N(B+t)$ at $B$,
$t\geq0$.

Fix $M>0$ and define $g$, a functional on the space
$\mathfrak{M}(\R)$
 of locally finite counting measures
$\mu$ on $\R$, as
$g(\mu):=e^{\theta_B\,\mu(B)}\wedge M$.
Since $B\subseteq[0,\infty)$, the value $g(\mu)$ depends on $\mu$ only through
$\mu|_{[0,\infty)}$.
For every $t\geq0$, we decompose
\begin{equation}\label{eq:decompose}
\E[e^{\theta_B\Nstat(B)}\wedge M]
=\E\big[g(\Nstat)\big]
=\E\big[g(N(\cdot + t))\big]
+\int_{\mathfrak{M}(\R)}  g(\mu)\,\d\big(P^{\mathrm{stat}}-P_t\big)(\mu),
\end{equation}
where $P^{\mathrm{stat}}:=\L(\Nstat|_{(0,\infty)})$ and
$P_t:=\L((N(\cdot + t))|_{(0,\infty)})$.
We estimate the first term in \eqref{eq:decompose} using  \eqref{eq:sup_bound}:
$$
\E\big[g(N(\cdot + t))\big]
=\E[e^{\theta_B N(B+t)}\wedge M]
\leq\E[e^{\theta_B N(B+t)}]
\leq c.
$$
For the second term, we obtain
$$
\biggl|\int_{\mathfrak{M}(\R)}  g(\mu)\,\d\big(P^{\mathrm{stat}}-P_t\big)(\mu)\biggr|
\leq 2M\,\big\|P^{\mathrm{stat}}-P_t\big\|_{\mathrm{TV}}.
$$
By Theorem~\ref{thm:weak},
$\|P^{\mathrm{stat}}-P_t\|_{\mathrm{TV}}\to0$ as $t\to\infty$. Combining the two
bounds in \eqref{eq:decompose} and letting $t\to\infty$, we obtain
$
\E[e^{\theta_B\Nstat(B)}\wedge M]\leq c\quad\text{for every } M>0,
$
and monotone convergence ($M\to\infty$) gives
$\E[e^{\theta_B\Nstat(B)}]\leq c=c_2(\theta_0)<\infty$.
This completes the proof.
\end{proof}

\begin{acks}[Acknowledgments]
 The authors would like to thank Wulfram Gerstner for sharing his expertise on neuron modeling and for his feedback on an earlier draft of the paper.
\end{acks}
\begin{funding}
Supported by the University of Warsaw grant IDUB-POB3-D110-003/2022.
\end{funding}

\bibliographystyle{imsart-number} 
\bibliography{biblio}             

@book{thorisson00,
  author    = {Thorisson, Hermann},
  title     = {Coupling, Stationarity, and Regeneration},
  series    = {Probability and its Applications (New York)},
  publisher = {Springer-Verlag},
  address   = {New York},
  year      = {2000},
}

@inproceedings{mei17,
  author = {Hongyuan Mei and Jason Eisner},
  title = {The Neural {H}awkes Process: {A} Neurally Self-Modulating Multivariate Point Process},
  booktitle = {Proc. ICML 2017},
  pages = {2534--2543},
  year = {2017},
  series = {PMLR 70},
  publisher = {PMLR},
}

@article{gerstner92,
  author = {Wulfram Gerstner and J. Leo van Hemmen},
  title = {Associative memory in a network of 'spiking' neurons},
  journal = {Network: Computation in Neural Systems},
  volume = {3},
  number = {2},
  pages = {139--164},
  year = {1992},
}

@article{jolivet06,
  author = {Renaud Jolivet and Alexander Rauch and Hans-Rudolf Lüscher and Wulfram Gerstner},
  title = {Predicting spike timing of neocortical pyramidal neurons by simple threshold models},
  journal = {Journal of Computational Neuroscience},
  volume = {21},
  number = {1},
  pages = {35--49},
  year = {2006},
}

@article{naud12,
  author = {Richard Naud and Wulfram Gerstner},
  title = {Coding and Decoding with Adapting Neurons: A
Population Approach to the Peri-Stimulus Time
Histogram},
  journal = {PLOS Computational Biology},
  volume = {8},
  number = {10},
  pages = {e1002711},
  year = {2012},
}

@book{daley03,
	address = {New York},
	author = {Daley, D.J. and Vere-Jones, D.},
	edition = {2nd},
	publisher = {Springer},
	title = {An Introduction to the Theory of Point Processes},
	volume = {I and II},
	year = {2003/2009}}

@article{hawkes71b,
	author = {Hawkes, A.G.},
	journal = {Biometrika},
	pages = {83--90},
	title = {{S}pectra of some self-exciting and mutually-exciting point processes},
	volume = {58},
	year = {1971}}

@article{hawkes71a,
	author = {Hawkes, A.G.},
	journal = {Journal of the Royal Statistical Society: Series B (Statistical Methodology)},
	pages = {438--443},
	title = {{P}oint spectra of some mutually-exciting point processes},
	volume = {33},
	year = {1971}}

@book{knopp51,
  Author = {Knopp, Konrad},
  Title = {Theory and Application of Infinite Series},
  Year = {1951},
  Publisher = {Blackie \& Son},
  Address = {London},
}

@book{meyn93,
  author    = {Meyn, Sean P. and Tweedie, Richard L.},
  title     = {Markov Chains and Stochastic Stability},
  series    = {Communications and Control Engineering Series},
  publisher = {Springer-Verlag},
  address   = {London},
  year      = {1993},
  isbn      = {978-1-4471-3269-1},
}

@article{bremaud96,
	author = {Br{\'{e}}maud, P. and Massouli{\'{e}}, L.},
	journal = {The Annals of Probability},
	number = {3},
	pages = {1563--1588},
	title = {Stability of nonlinear {H}awkes processes},
	volume = {24},
	year = {1996}}

@article{costa20,
	author = {Costa, Manon and Graham, Carl and Marsalle, Laurence and Tran, Viet Chi},
	journal = {Advances in Applied Probability},
	pages = {879--915},
	title = {Renewal in {H}awkes processes with self-excitation and inhibition},
	volume = {52},
	year = {2020}}

@misc{deutsch23,
	archiveprefix = {arXiv},
	author = {Isabella Deutsch and Gordon J. Ross},
	eprint = {2201.05009},
	primaryclass = {stat.AP},
	title = {Estimating Product Cannibalisation in Wholesale using Multivariate {H}awkes Processes with Inhibition},
	year = {2023}}

@article{ditlevsen17,
	author = {Susanne Ditlevsen and Eva L{\"o}cherbach},
	issn = {0304-4149},
	journal = {Stochastic Processes and their Applications},
	number = {6},
	pages = {1840--1869},
	title = {Multi-class oscillating systems of interacting neurons},
	volume = {127},
	year = {2017}}

@article{gerhard17,
	address = {San Francisco},
	author = {Gerhard, Felipe and Deger, Moritz and Truccolo, Wilson},
	journal = {Plos Computational Biology},
	number = {2},
	pages = {31. e1005390},
	publisher = {Public Library of Science},
	title = {On the stability and dynamics of stochastic spiking neuron models: Nonlinear {H}awkes process and point process {GLM}s},
	volume = {13},
	year = {2017}}

@book{gerstner14,
	author = {Wulfram Gerstner and Werner M. Kistler and Richard Naud and Liam Paninski},
	publisher = {Cambridge University Press},
	title = {Neuronal Dynamics: From Single Neurons To Networks And Models Of Cognition},
	year = {2014}}

@article{massoulie98,
	author = {Laurent Massouli{\'e}},
	issn = {0304-4149},
	journal = {Stochastic Processes and their Applications},
	number = {1},
	pages = {1-30},
	title = {Stability results for a general class of interacting point processes dynamics, and applications},
	volume = {75},
	year = {1998}}

@article{pillow08,
	author = {Jonathan W. Pillow and Jonathon Shlens and Liam Paninski and Alexander Sher and Alan M. Litke and E. J. Chichilnisky and Eero P. Simoncelli},
	journal = {Nature},
	pages = {995-999},
	title = {Spatio-temporal correlations and visual signalling in a complete neuronal population},
	volume = {454},
	year = {2008}}

@article{truccolo05,
	author = {Truccolo, Wilson and Eden, Uri T. and Fellows, Matthew R. and Donoghue, John P. and Brown, Emery M.},
	journal = {Journal of Neurophysiology},
	pages = {1074--1089},
	title = {A Point Process Framework for Relating Neural Spiking Activity to Spiking History, Neural Ensemble, and Extrinsic Covariate Effects},
	volume = {93},
	year = {2005}}

@article{truccolo16,
	author = {Truccolo, Wilson},
	journal = {Journal of physiology, Paris},
	month = {03},
	pages = {336--347},
	title = {From point process observations to collective neural dynamics: Nonlinear {Hawkes} process {GLM}s, low-dimensional dynamics and coarse graining},
	volume = {110},
	year = {2016}}

@article{duarte19,
	author = {Duarte, Aline and L\"ocherbach, Eva and Ost, Guilherme},
	journal = {ESAIM: PS},
	pages = {770-796},
	title = {Stability, convergence to equilibrium and simulation of non-linear {H}awkes processes with memory kernels given by the sum of Erlang kernels},
	volume = 23,
	year = 2019}

@phdthesis{zhu13,
	author = {Lingjiong Zhu},
	title = {Nonlinear {H}awkes {P}rocesses},
	school = {New York University},
	year = {2013}}

@misc{gao23,
	archiveprefix = {arXiv},
	author = {Fuqing Gao and Yunshi Gao and Lingjiong Zhu},
	eprint = {2307.15903},
	primaryclass = {math.PR},
	title = {Fluctuations and moderate deviations for the mean fields of {H}awkes processes},
	year = {2023}}

@book{last17,
	address = {Cambridge},
	author = {Last, G{\"u}nter and Penrose, Mathew},
	publisher = {Cambridge University Press},
	series = {Institute of Mathematical Statistics Textbooks},
	title = {Lectures on the {P}oisson Process},
	volume = {7},
	year = {2017}}

@article{kirchner26,
	author = {Kirchner, Matthias},
	title = {{P}alm versions of {H}awkes processes},
	journal = {Statistics \& Probability Letters},
	volume = {227},
	pages = {110531},
	year = {2026},
	issn = {0167-7152},
}

@Book{kallenberg2017,
  title     = {Random Measures, Theory and Applications},
  publisher = {Springer International Publishing},
  year      = {2017},
  author    = {Kallenberg, Olav},
  volume    = {77},
  series    = {Probability Theory and Stochastic Modelling},
  address   = {Cham},
  isbn      = {978-3-319-41596-3},
}

\end{document}